\documentclass{amsart}[11pt]
\usepackage{hyperref}
\hypersetup{backref,colorlinks=true,allcolors=black}
\usepackage[foot]{amsaddr}
\usepackage{mathrsfs}
\usepackage{enumitem}
\usepackage{dsfont}
\usepackage{mathtools}
\usepackage{xcolor}
\usepackage[final]{graphicx}
\usepackage{upgreek}
\usepackage{amsmath,amssymb}

\newtheorem{proposition}{Proposition}
\newtheorem{lemma}{Lemma}
\newtheorem{corollary}{Corollary}
\newtheorem{theorem}{Theorem}

\theoremstyle{definition}
\newtheorem{definition}{Definition}

\newtheorem{remark}{Remark}

\newcommand{\ignore}[1]{}
\newcommand{\R}{\mathbb{R}}
\newcommand{\N}{\mathbb{N}}

\newcommand{\eu}{{\epsilon,w^\epsilon}}

\newcommand{\norm}[1]{\left\lVert #1 \right\rVert}
\newcommand{\abs}[1]{\left\vert #1 \right\rvert}
\newcommand{\E}[1]{\mathbb{E}{\left[ #1\right]}}

\DeclarePairedDelimiter\autobracket{(}{)}
\newcommand{\brac}[1]{\autobracket*{#1}}
\newcommand{\inner}[1]{\left\langle #1 \right\rangle}

\theoremstyle{definition}

\makeatletter
\@addtoreset{proofcase}{theorem}
\makeatother

\newtheorem{innercustomthm}{Condition}
\newenvironment{customcondition}[1]
{\renewcommand\theinnercustomthm{#1}\innercustomthm}
{\endinnercustomthm}

\theoremstyle{definition}

\makeatletter
\@addtoreset{proofstep}{theorem}
\makeatother

\usepackage{microtype}
\usepackage{parskip}
\usepackage[symbol]{footmisc}
\title[Moderate deviations for fractional multiscale systems]{Moderate deviation principle for multiscale systems driven by fractional Brownian motion}
\author{Solesne Bourguin}
\email[Solesne Bourguin]{bourguin@math.bu.edu}
\author{Thanh Dang}
\email[Thanh Dang]{ycloud77@bu.edu}
\author{Konstantinos Spiliopoulos}
\email[Konstantinos Spiliopoulos]{kspiliop@math.bu.edu}
\address{Boston University, Department of Mathematics and Statistics, 111 Cummington Mall, Boston, MA 02215, USA}
\thanks{\textit{Corresponding author}: Solesne Bourguin (\url{bourguin@math.bu.edu})}
\begin{document}
	
	\begin{abstract}
		In this paper we study the moderate deviations principle (MDP) for slow-fast stochastic dynamical systems where the slow motion is governed by small fractional Brownian motion (fBm) with Hurst parameter $H\in(1/2,1)$. We derive conditions on the moderate deviations scaling and on the Hurst parameter $H$ under which the MDP holds. In addition, we show that in typical situations the resulting action functional is discontinuous in $H$ at $H=1/2$, suggesting that the tail behavior of stochastic dynamical systems perturbed by fBm can have different characteristics than the tail behavior of such systems that are perturbed by standard Brownian motion.
	\end{abstract}
	

	\bibliographystyle{amsalpha}
	\maketitle
	\section{Introduction}
	\label{section_intro}
	The goal of this paper is to study the asymptotic behavior, in the moderate deviations regime, of the following system of slow-fast dynamics
	\begin{align}
		\label{originalsystem}
		dX^\epsilon_t&=g\brac{X^\epsilon_t,Y^\epsilon_t}dt+\sqrt{\epsilon} f\brac{X^\epsilon_t,Y^\epsilon_t}d W^H_t, \quad X^{\epsilon}_{0}=x_{0}\nonumber\\		dY^\epsilon_t&=\frac{1}{\epsilon}c\brac{Y^\epsilon_t}+\frac{1}{\sqrt{\epsilon}}\sigma\brac{Y^\epsilon_t}dB_t, \quad Y^{\epsilon}_{0}=y_{0}.
	\end{align}
	
	Here $\epsilon$ is a small parameter that goes to zero. We assume that $t\in[0,1]$ and $(X^\epsilon,Y^\epsilon)\in \R^n\times \R^d$. Also, $B$ is a standard $m$-dimensional Brownian motion, while $W^{H}$ is a $p$-dimensional fractional Brownian motion (fBm) with Hurst parameter $H\in(1/2,1)$ independent of $B$. As is known, if $H=1/2$ then $W^{1/2}$ will be a standard Brownian motion. Moreover, the integral with respect to $W^H$ is a pathwise Riemann–Stieltjes integral and is commonly known as a Young integral (see Appendix \ref{A:Appendix} for a brief introduction).
	
	Since, $1/\epsilon\uparrow\infty$ as $\epsilon\downarrow 0$, we expect that under the appropriate conditions, the distribution of  $Y^\epsilon$ will be converging to its invariant distribution, while the equation that $X^\epsilon$ satisfies can be viewed as a perturbation of a dynamical system by small multiplicative noise of magnitude $\sqrt{\epsilon}$. We can think of $X^\epsilon$ as the slow component and of $Y^\epsilon$ as the fast component. Model (\ref{originalsystem}) is a prototypical dynamical system that exhibits multiple characteristic scales in  time and is perturbed by small noise to account for imperfect information and to capture random phenomena. Such systems arise naturally as models
	in a great variety of applied fields, including physics, chemistry, biology, neuroscience, meteorology, and
	mathematical finance, to name a few.

	The novelty of this paper lies in the consideration of the tail behavior of (\ref{originalsystem}) in the case where $H\neq 1/2$. In the case of $H=1/2$, i.e., when both the $X^{\epsilon}$ and $Y^{\epsilon}$ components are driven by Brownian motions, the asymptotic behavior of systems like (\ref{originalsystem}) have been extensively studied in the literature. We refer the interested reader to
	\cite{BaierFreidlin,DupuisSpiliopoulos,Freidlin1978,FS,FreWen,Guillin,KlebanerLiptser,MorseSpiliopoulosMDP,LiptserPaper,ParVer1,ParVer2,spiliopoulos2014fluctuation,Spiliopoulos2012}, which contain results on
	related typical averaging dynamics, central limit theorems, moderate and large
	deviations. Choosing the noise that perturbs the system to be standard Bronwian motion, we embed the Markov property and semimartingale structure of the standard Brownian motion in the system. However, many
	physical dynamical system exhibit long-range dependence or a particular sort of self-similarity that may not be amenable to accurate description by a model driven by standard Brownian noise.
	
	One way to account for this issue, is to perturb the dynamical system by fractional Brownian motion. Such practice has seen growing interest in literature, for example the references \cite{BFGH,BudSon,Che03,CR,Fuk, FZ,GJRS,HJL19,SV} to name a few. However, the corresponding literature for multiscale systems like (\ref{originalsystem}) in the case of perturbation
	by fractional Brownian motion is quite sparse and still in its infancy. We refer the interested reader to the very recent papers of \cite{BGS,HaiLi, PIX20} for results concerning typical averaging behavior, homogenization, and fluctuations corrections for multiscale models like (\ref{originalsystem}) under different sets of assumptions for the model coefficients. As discussed in these papers, replacing Brownian motion by fractional Brownian motion creates a number of issues that need to be overcome. These issues are mainly related to the partial loss of the Markovian structure as well as to the proper averaging of the integral with respect to the fractional Brownian motion $W^H$ which originates from the interaction of ergodicity and fBm.
	
	The intent of this paper is to study the tail behavior of $X^{\epsilon}$ in (\ref{originalsystem}) as $\epsilon\downarrow 0$ in the moderate deviation setting. To be more precise, letting $h(\epsilon)\rightarrow\infty$ such that $\sqrt{\epsilon}h(\epsilon)\rightarrow 0$ and defining $\bar{X}_{t}=\lim_{\epsilon\rightarrow 0}X^{\epsilon}_{t}$ (the limit in the appropriate sense), we define the moderate deviations process
	\begin{align}
		\label{MDprocess}	\eta^\epsilon_t=\frac{X^\epsilon_t-\bar{X}_t}{\sqrt{\epsilon}h\brac{\epsilon}}.
	\end{align}
	
	Moderate deviations for $X^{\epsilon}$ refer to large deviations for $\eta^{\epsilon}$. In fact the scaling by $\sqrt{\epsilon}h(\epsilon)$ implies that moderate deviations is in the regime between central limit theorem (corresponding to the choice $h(\epsilon)=1$) and large deviations (corresponding to the choice $h(\epsilon)=1/\sqrt{\epsilon}$). Moderate deviations for systems like (\ref{originalsystem}) and for $H=1/2$, i.e., when both slow and fast components are driven by standard Brownian motions, have been considered in \cite{Guillin,MorseSpiliopoulosMDP}. An interesting conclusion of our results for the case $H\neq 1/2$, which will be discussed in Remark \ref{R:BM_fBM_comparison} is that the resulting action functional is not continuous in $H$ at $H=1/2$. At this point we also mention the recent work \cite{SiraganGasteratos} that considers the large deviations counterpart for stochastic dynamical systems similar to (\ref{originalsystem}).
	
	In order to study the moderate deviations principle for $X^{\epsilon}$, we shall follow the weak convergence method of \cite{DupEll}. The core of this approach lies in the use of a variational representation of exponential functionals of the driving noise $(W^{H},B)$, see \cite{DupEll,Zha}. In our case, such a representation leads to a representation of the exponential functional of the moderate deviations process $\eta^{\epsilon}$ that appears in the Laplace principle (which is equivalent to the moderate deviations) as a variational infimum of a family of controlled moderate deviations processes $\eta^\eu$ together with a quadratic cost over a suitable family of stochastic controls $w^{\epsilon}$. To be more precise, letting $a$ be a bounded Borel
	function on $C\brac{[0,1];\R^n}$, we have the representation
	\begin{align}
		\label{equation_varrep0}
		-\frac{1}{h^2(\epsilon)}\ln \E{\exp\brac{-h^2(\epsilon)a(\eta^\epsilon)}}&=\inf_{w^\epsilon\in \mathcal{S}}\E{\frac{1}{2}\norm{w^\epsilon}^2_{\mathcal{S}}+a\brac{\eta^\eu}},
	\end{align}
	where $\mathcal{S}$
	denotes the Cameron-Martin
	space associated with the process $\left\{ (W^H_t,B_t)\colon t
	\in [0,1]\right\}$ (see
	\eqref{def_Cameron_Martin_joint}) and the controlled deviations process $\eta^\eu$ is defined by
	\begin{align}
		\label{def_control_eta0}
		\eta^\eu_t=\frac{X^\eu_t-\bar{X}_t}{\sqrt{\epsilon}h\brac{\epsilon}}
	\end{align}
	with the controlled processes $(X^\eu,Y^\eu)$  defined by	
	(\ref{equation_controlled_XY}).
	
	Essentially, proving the moderate deviations principle for $X^{\epsilon}$ amounts to finding the limit as $\epsilon\rightarrow 0$ to (\ref{equation_varrep0}). When $H\neq 1/2$, i.e., when the standard Brownian motion in the slow component is replaced by fBm, a number of additional technical issues come up and the standard methodology needs to be modified. After we introduce proper notation, we explain in Remark \ref{R:CoreIdea} of Section 3 one of the core ideas that allow us to study the $H\neq 1/2$ case in a way that naturally extends the $H=1/2$ setting.
	
	The rest of the paper is organized as follows. In Section \ref{section_prelim} we establish necessary notation, go over our assumptions and present the main result, Theorem \ref{theorem_rate_function}, on the moderate deviations principle with an explicit representation of the action functional, as well as a corollary of the aforementioned theorem. Section \ref{section_proof_main} contains the details of the weak convergence approach for the problem at hand, introduces the appropriate controlled processes and presents Theorem \ref{theobeforevaria} which has a variational representation of the moderate deviations action functional. Theorem \ref{theorem_rate_function} can be viewed as a direct consequence of Theorem \ref{theobeforevaria}. In Section \ref{section_proof_main} we also go over one of the main ideas that essentially unlock the computation for the case $H\neq 1/2$, in a way that naturally extends the standard $H=1/2$ framework, see Remark \ref{R:CoreIdea}. Section \ref{S:Examples} contains examples that demonstrate our theoretical results.
	
	Section \ref{proofoftheo2} contains the proof of Theorem  \ref{theobeforevaria} and consequently of
	Theorem \ref{theorem_rate_function} as well. In particular, in Section  \ref{proofoftheo2} we prove tightness of the appropriate controlled processes and occupational measures introduced in Section \ref{section_proof_main}, we identify their weak limit which then allows to prove the limit Laplace principle lower and the upper bound of (\ref{equation_varrep0}). The proof of the Laplace upper bound leads to the exact representation of Theorem \ref{theorem_rate_function}. Next, Section \ref{section_proof_cor_g(x)} provides the proof of Corollary \ref{corollary_g(x)}. Section \ref{S:Conclusions} discusses avenues for future work on this topic.
	
	Appendix \ref{A:Appendix} recalls several aspects of fBm and necessary results on pathwise stochastic integration with respect to fBm used in this paper. In Appendix \ref{A:Appendix}, we also discuss the Cameron-Martin space of fBm and prove associated results that are potentially of independent interest as well. Appendix \ref{A:AppendixB} recalls regularity results of \cite{ParVer1,ParVer2} on Poisson equations, proves necessary a-priori bounds of the slow-fast process $(X^{\epsilon}, Y^{\epsilon})$ in (\ref{originalsystem}) as well as necessary a-priori estimates on $\eta^\eu$ from (\ref{def_control_eta0}) that allows to establish the necessary tightness results in Section \ref{section_proof_main}.



	
	
	\section{Notation, conditions and main results}
	\label{section_prelim}
	In this section, we introduce some notation, present the main
	assumptions we make, and state our main results. We work with a canonical probability space $(\Omega, \mathcal F, P)$ equipped with a filtration $\{ \mathcal F_t \}_{0 \leq t \leq T}$ satisfying the usual conditions (namely, $\{ \mathcal F_t \}_{0 \leq t \leq T}$ is right continuous and $\mathcal F_0$ contains all $P$-negligible sets).
	
	We will denote by $A:B$ the Frobenius inner product $\Sigma_{i,j}[a_{i,j}\cdot b_{i,j}]$ of matrices $A=(a_{i,j})$ and $B=(b_{i,j})$.
	We will use single bars $|\cdot|$ to denote the Frobenius (or Euclidean) norm of a matrix and double bars $||\cdot||$ to denote the operator norm. For $\alpha\in (0,1)$, $|\cdot|_\alpha$ is the standard H\"{o}lder semi-norm, i.e.
	\begin{align*}
		\abs{h}_\alpha=\sup_{0\leq s\neq t\leq 1}\frac{\abs{h(s)-h(t)}}{\abs{s-t}^\alpha}.
	\end{align*}
	For some set $A$ and $\alpha\in (0,1)$, $\mathcal{C}^\alpha(A)$ is the H\"{o}lder space for H\"{o}lder functions defined on $A$. Meanwhile, for $k\in \N$, $\mathcal{C}^k(A)$ denote the usual space of k-times continuously differentiable functions on $A$. In addition, for given sets $A,B$ and $i,j\in\N$ and $\zeta\in(0,1)$, $\mathcal{C}^{i,j+\zeta}(A\times B)$ is the space of functions on $A\times B$ with $i$ bounded derivatives in $x$ and $j$ derivatives in $y$, with all partial derivatives being $\zeta$-H\"{o}lder continuous with respect to $y$, uniformly in $x$.
	\subsection{Conditions}
	

	
	We start by stating the assumptions we make on the coefficients of
	$Y^\epsilon$ ensuring its ergodicity. Note that these assumptions are
	satisfied in the context of the multi-scale models studied in
	\cite[Theorem 1]{BGS} and \cite[Theorem A]{HaiLi}.
	
	
	\begin{customcondition}{H1}~
		\label{condH1}
		\begin{itemize}
			\item[-] $c(y)=-\Gamma y+\zeta(y)$, for which $\Gamma$ be a $d\times d$ positive matrix with bounded entries and $\zeta(y)$ is a uniformly Lipschitz function with Lipschitz coefficient $L_\zeta$. Moreover, $\zeta(y)\leq C\abs{y}$ and $\inner{(\Gamma-L_\zeta I)\xi,\xi}\geq \gamma_0 \abs{\xi}^2$ for some $\gamma_0>0$.
			\item[-] $c(y),\sigma(y)$ have first and second derivatives that are $\alpha$-H\"{o}lder continuous for some $\alpha>0$.
			\item[-] $\sigma(y)\sigma^\top(y)$ is uniformly continuous, bounded and non-degenerate.
			\item[-] There are positive constants $\beta_1,\beta_2$ such that $0<\beta_1\leq \frac{\inner{\sigma(y)\sigma^\top(y) y,y}}{\abs{y}^2}\leq \beta_2,\forall y\in \R^d \setminus \{0\}$.
		\end{itemize}
	\end{customcondition}
	\begin{remark}
		\label{remarkoninvmeasure}
		Condition \ref{condH1} guarantees that the fast process
		$Y^{\epsilon}$ has a unique invariant measure, which we denote
		by $\mu(dy)$ in the sequel.
	\end{remark}
	Denote by $\mathcal{L}$ the normalized infinitesimal generator of the fast motion
	$Y^{\epsilon}$ (with respect to which averaging is being performed). It is given by
	\begin{align}
		\label{def_generator_normalized_Y}
		\mathcal{L}F(y)&=\nabla_y F(y)^\top c(y) +\frac{1}{2}\sigma(y)\sigma^\top(y):\nabla^2_yF(y),
	\end{align}
	where $F\in \mathcal{C}^2(\R^d)$. Set $\mathcal{Y}=\R^d$. For any function $G(x,y)$, define the averaged function $\bar{G}$ by
	\begin{align*}
		\bar{G}(x) = \int_{\mathcal{Y}}^{}G(x,y)\mu(dy).
	\end{align*}
	In particular, the averaging of the drift term $g$ in the slow motion
	$X^{\epsilon}$ with respect to $\mu$ will be given by
	\begin{align*}
		\bar{g}(x) = \int_{\mathcal{Y}} g(x,y)\mu(dy).
	\end{align*}
	\begin{remark}
		Under the growth assumption on $g$ and its derivatives in either the upcoming Condition \ref{condH2-A} or \ref{condH2-B}, Theorem \ref{theorem_Poisson_equation} implies that the partial differential equation
		\begin{equation}
			\label{equation_Poisson_g}
			\mathcal{L}\phi(x,y)=g(x,y)-\bar{g}(x),\quad
			\int_\mathcal{Y}\phi(x,y)\mu(dy)=0
		\end{equation}
		has a unique, twice differentiable solution (that we denote by
		$\phi(x,y)$ in the sequel) in the class
		of functions that grow at most polynomially in $\abs{y}$.
	\end{remark}

	Finally, we provide two different sets of assumptions on the
	coefficients of $X^\epsilon$, each of which is based on the available
	averaging results for $X^\epsilon$ appearing in \cite{BGS} and
	\cite{HaiLi}, respectively. Depending on the specific multi-scale
	model at hand, one may choose to work with one set of assumptions or the other.
	\begin{customcondition}{H2-A}
		\label{condH2-A}
		The assumptions below relate to the setting of \cite{HaiLi}.
		\begin{itemize}
			\item[-]  $f(x,y)$ and $g(x,y)$ are uniformly bounded with bounded first and second partial derivatives.
			\item[-]  There exists $\beta$ in $[0,1]$ such that $\beta+H>1$ and $h(\epsilon)^{-1}\epsilon^{-\frac{\beta}{2}}\to 0$ as $\epsilon\to 0$.
		\end{itemize}
	\end{customcondition}	
	
	\begin{customcondition}{H2-B}
		\label{condH2-B}
		The assumptions below relate to the setting of
		\cite{BGS}. We assume that there are constants $D_f,D_g,M_f,M_k$ in $[0,1]$ and $\alpha$ in $(0,1]$ such that
		\begin{itemize}
			\item[-] $g\in \mathcal{C}^{2,\alpha}\brac{\R^n,\mathcal{Y}}$.
			\item[-] $f=f(y)$ and $g$ satisfy the growth assumption
			\\$\abs{f(y)}\leq C\brac{1+\abs{y}^{D_f}}$ and $\abs{g(x,y)+\nabla_xg(x,y)+\nabla^2_xg(x,y)}\leq C\brac{1+\abs{y}^{D_g}}.$
			\item[-]  $D_f$ and $D_g$ are related via $0\leq D_f+D_g<1$.
			\item[-]  $f(y)$ and $\nabla_x\phi(x,y) f(y)$ are respectively $M_f$ and $M_k$-H\"{o}lder continuous, where $\phi(x,y)$ is defined at \eqref{equation_Poisson_g}. Moreover, we have
			$\min\left\{{\frac{M_f}{2}+H},{\frac{M_k}{2}+H}\right\}>1$.
			\item[-] $h(\epsilon)^{-1}\epsilon^{-\frac{M_f}{2}}\to 0$ as $\epsilon\to 0$.
		\end{itemize}
	\end{customcondition}
	\begin{remark}
		Conditions \ref{condH2-A} and \ref{condH2-B} relate to the averaging
		results in \cite{HaiLi} and \cite{BGS} that state that the slow motion $X^{\epsilon}$ converges in
		probability, as $\epsilon$ goes to 0, to a deterministic limit $\bar{X}$ defined to be the
		solution to the integral equation
		\begin{align*}
			d\bar{X}_t = \bar{g}(\bar{X}_t)dt,\quad \bar{X}_0 = x_0.
		\end{align*}
	\end{remark}
	In addition, we need to assume uniqueness of a strong solution. Without having to refer to it again, this assumption is always in effect in this paper.
	\begin{customcondition}{H3}
		\label{condH3}
		The stochastic differential equation at \eqref{originalsystem} has a unique strong solution.
	\end{customcondition}
	\begin{remark}
		We direct readers to \cite{GueNua,MS,SLE} for existence and uniqueness of solutions to stochastic differential equations like \eqref{originalsystem}.
	\end{remark}

	Finally, define the operator $Q^{H}_{\bar{X}}$ by
	\begin{align}
		\label{operateurQ}
		Q^{H}_{\bar{X}}=\bar{f}\brac{\bar{X}}\dot{K}_H\brac{\bar{f}\brac{\bar{X}}\dot{K}_H}^*+\int_{\mathcal{Y}}{\nabla_y\phi\brac{\bar{X},y}
			\sigma\brac{y}}\brac{\nabla_y\phi\brac{\bar{X},y} \sigma\brac{y}}^\top \mu(dy),
	\end{align}
	where $\mu$ is the invariant measure defined in Remark
	\ref{remarkoninvmeasure}, $\dot{K}_H$ is the operator (related to the fractional Brownian
	motion) defined in \eqref{khdot} (see  Appendix \ref{SS:fBm_preliminaries}). Per the explanation in Section \ref{prooflaplaceupper}, both the domain and range of $Q^{H}_{\bar{X}}$ can be taken to be $L^2([0,1];\R^n)$. In fact, let $h\in L^2\brac{[0,1];\R^n}$. Then the operator $\bar{f}\brac{\bar{X}}\dot{K}_H\brac{\bar{f}\brac{\bar{X}}\dot{K}_H}^*$ admits the explicit representation
	\begin{align*}
		&\left[\bar{f}\brac{\bar{X}}\dot{K}_H\brac{\bar{f}\brac{\bar{X}}\dot{K}_H}^* h\right](t)=c_H^2\bar{f}\brac{\bar{X_t}}t^{H-1/2}\\&\qquad\qquad\qquad\qquad\int_0^t(t-z)^{H-3/2}z^{1-2H}\int_z^1 (s-z)^{H-3/2}s^{H-1/2}\bar{f}\brac{\bar{X_s}}^\top h(s)dsdz
	\end{align*}
	such that the constant $c_H$ equals $\brac{H(2H-1)/\beta(2-2H,H-1/2)}^{1/2}$, where
	$\beta(x,y)=\frac{\Gamma(x)\Gamma(y)}{\Gamma(x+y)}$ is the standard beta function.

	\begin{remark}
		In both this paper and the paper \cite{BGS}, the latter of which provides averaging results on multiscale models like \eqref{originalsystem}, one needs to bound terms that are Young integrals. However, each paper uses a different bounding technique which leads to different assumptions on \eqref{originalsystem}. For instance, the authors of \cite{BGS} use the maximal inequality in their Lemma 1 to bound
		\begin{align}
			\label{term_young_int_regular}
			\int_0^t f\brac{Y^\epsilon_s}dW^H_s,
		\end{align}
		an integral term which appears in \eqref{originalsystem}. Having the kernel $f(Y^\epsilon)$ independent from the driving process $W^H$ simplifies the application of the maximal inequality and yields the necessary bound on the integral \eqref{term_young_int_regular}. For this paper, we instead have to bound
		\begin{align}
			\label{term_young_int_control}
			\int_0^t f\brac{Y^\eu_s}dW^H_s,
		\end{align}
		an integral term which appears in \eqref{equation_controlled_XY}. In this case, the control process $w^\epsilon$ is dependent on $W^H$, implying the kernel $f\brac{Y^\eu}$ is dependent on $W^H$ as well. This lack of independence between the kernel and the driving process leads us to substitute the Young-Lo\'{e}ve inequality for the maximal inequality in order to bound \eqref{term_young_int_control}. The Young-Lo\'{e}ve inequality for Young integrals requires some kind of uniform H\"{o}lder continuity of the kernel, which explains why we impose certain uniform H\"{o}lder continuity condition on the coefficients of \eqref{originalsystem}, an assumption not made in \cite{BGS}.
		
	\end{remark}

	\subsection{Main results}
	\label{section_main_results}	
	The weak convergence approach to large deviations developed in
	\cite{DupEll} states that the large deviations principle for
	$\eta_t^{\epsilon}$ is equivalent to the Laplace principle which
	states that for any bounded continuous function $a \colon
	C\brac{[0,1];\R^n} \to \R$, there exists a rate function (also called
	action functional) $S^{H} \colon C\brac{[0,1];\R^n}\to \R$ that satisfies
	\begin{align*}
		\lim_{\epsilon\to\infty}-\frac{1}{h^2(\epsilon)}\ln \E{\exp\brac{-h^2(\epsilon)a(\eta^\epsilon)}}=\inf_{\Phi\in C\brac{[0,1];\R^n}}\{ S^{H}(\Phi)+a(\Phi)\}.
	\end{align*}	
	In this paper, we prove that the above Laplace principle holds and our
	main result, Theorem \ref{theorem_rate_function}, identifies that rate function $S^{H}(\Phi)$
	explicitly. The statement of this theorem is given below.

	\begin{theorem}
		\label{theorem_rate_function}
		Let Conditions \ref{condH1} and either \ref{condH2-A} or
		\ref{condH2-B} be satisfied. Moreover, assume that the operator $Q^{H}_{\bar{X}}$
		defined in \eqref{operateurQ} is invertible on $L^2([0,1];\R^n)$. Then, the process
		$\{X^\epsilon\colon \epsilon>0\}$ satisfies the moderate
		deviations principle, with the action functional
		$S^{H}(\Phi)$ given by
		\begin{align*}
			S^{H}(\Phi)= \int_0^1\brac{\dot{\Phi}_s-\nabla_x \bar{g} (\bar{X}_s)\Phi_s}^\top(Q^{H}_{\bar{X}_{s}})^{-1}\brac{\dot{\Phi}_s-\nabla_x \bar{g} (\bar{X_s})\Phi_s}ds  
		\end{align*}
		if $\Phi \in C\brac{[0,1];\R^n}$ is absolutely continuous, and $\infty$ otherwise.
	\end{theorem}

		
	
	\begin{remark}
		A sufficient condition for invertibility of $Q^{H}_{\bar{X}}$ on $L^2([0,1];\R^n)$ is
		that for all $x$ and non-zero $z\in \R^n$, we have
		\begin{align}
			\label{condition_suff_Q_inver}
			\inner{\int_{\mathcal{Y}}\nabla_y\phi\brac{x,y}\sigma\brac{y} \brac{\nabla_y\phi\brac{x,y}\sigma\brac{y}}^\top\mu(dy)z,z}>0.
		\end{align}
		The sufficiency of this condition is established in
		Lemma \ref{lemma_Q_invertible}. In most situations,
		Condition \eqref{condition_suff_Q_inver} proves easier
		to verify than the invertiblity of $Q^{H}_{\bar{X}}$ itself.
	\end{remark}

	\begin{remark}\label{R:BM_fBM_comparison}
		Let us now briefly compare the results in the $H=1/2$ and $H\neq 1/2$ case. As can be seen from the results of \cite{MorseSpiliopoulosMDP}, when $H=1/2$, i.e., when the slow motion in (\ref{originalsystem}) is driven by standard Brownian motion, the corresponding MDP is as in Theorem \ref{theorem_rate_function} but with $Q^{H}_{\bar{X}}$ defined in (\ref{operateurQ}) replaced by
		\begin{align}
			\label{operateurQ_BM}
			Q^{1/2}_{\bar{X}}=\int_{\mathcal{Y}} f(\bar{X},y)f(\bar{X},y)^\top \mu(dy)+\int_{\mathcal{Y}}{\nabla_y\phi\brac{\bar{X},y}
				\sigma\brac{y}}\brac{\nabla_y\phi\brac{\bar{X},y} \sigma\brac{y}}^\top \mu(dy).
		\end{align}
		
		It is interesting to note that the mapping $H\mapsto Q^{H}_{\bar{X}}$ is not, in general, continuous in $H$ at $H=1/2$. Indeed, if $H=1/2$, then the discussion of Appendix \ref{SS:fBm_preliminaries} shows that  $\dot{K}_{1/2}$ the operator  defined in \eqref{khdot}, will be the identity operator, so one would actually expect that
		\begin{align*}
			\lim_{H\rightarrow 1/2}Q^{H}_{\bar{X}}=\int_{\mathcal{Y}} f(\bar{X},y)\mu(dy)\int_{\mathcal{Y}}f(\bar{X},y)^\top \mu(dy)+\int_{\mathcal{Y}}{\nabla_y\phi\brac{\bar{X},y}
				\sigma\brac{y}}\brac{\nabla_y\phi\brac{\bar{X},y} \sigma\brac{y}}^\top \mu(dy).
		\end{align*}
		which is of course different from (\ref{operateurQ_BM}). Hence, we indeed have, unless of course $f(x,y)=f(x)$, that
		\begin{align*}
			Q^{1/2}_{\bar{X}}\neq \lim_{H\rightarrow 1/2}Q^{H}_{\bar{X}}.
		\end{align*}
		
		This result also immediately says that in general there is no continuity of the mapping $H\mapsto S^{H}$  at $H=1/2$.
		
		The lack of continuity does not come as a surprise. It is related to the fact that when averaging integrals with respect to fBm with $H\in(1/2,1)$, then one averages the integrand directly as opposed to the quadratic variation which is what happens when  $H=1/2$. We refer the interested reader to the recent papers \cite{BGS,HaiLi,LiSeiber2020} for further discussion on this.
		
	\end{remark}

	In certain circumstances, we can provide an explicit formula for $\brac{Q^{H}_{\bar{X}}}^{-1}$ as the following corollary of Theorem \ref{theorem_rate_function} shows. The proof will be presented in Section \ref{section_proof_cor_g(x)}.
	\begin{corollary}
		\label{corollary_g(x)}
		Let Conditions \ref{condH1} and either \ref{condH2-A} or \ref{condH2-B} be satisfied. Assume further that $1/2<H<3/4$, $g=g(x)$ and $f\brac{\bar{X}}$ is an invertible square matrix with a bounded inverse denoted by $L$. Then  $Q^{H}_{\bar{X}}$ is invertible and the process
		$\{X^\epsilon\colon \epsilon>0\}$ satisfies the moderate
		deviations principle as in Theorem \ref{theorem_rate_function}. In particular, given $\Psi\in L^2\brac{[0,1];\R^n}$, $\brac{Q^{H}_{\bar{X}}}^{-1}$ has the explicit form
		\begin{align*}
			\brac{Q^{H}_{\bar{X}}}^{-1}\Psi=c_H^{-2}\Gamma(H-1/2)^{-2}L^\top t^{1/2-H}D^{H-1/2}_{1^-}t^{2H-1}D^{H-1/2}_{0^+}t^{1/2-H}L\Psi
		\end{align*}
		or equivalently,
		\begin{align*}
			\brac{\brac{Q^{H}_{\bar{X}}}^{-1}\Psi}(t)&=c_H^{-2}\Gamma(H-1/2)^{-2}\Gamma(3/2-H)^{-2}L_t^\top t^{1/2-H}\\& \frac{d}{dt}\left[ \int_t^1 (z-t)^{1/2-H}z^{2H-1}\frac{d}{dz}\left[ \int_0^z (z-s)^{1/2-H}s^{1/2-H} L_s\Psi_sds\right]dz \right].
		\end{align*}
	\end{corollary}


	\section{The controlled processes}
	\label{section_proof_main}


	The proof of the Laplace principle is based on a variational
	formula established in \cite[Theorem 3.2]{Zha}, which can be
	regarded as an abstract Wiener space counterpart of the
	stochastic control representation from \cite[Theorem 3.1]{BouDup}
	for the classical Wiener space. Recall that $\mathcal{S}$
	denotes the Cameron-Martin
	space associated with the process $\left\{ (W^H_t,B_t)\colon t
	\in [0,1]\right\}$ defined in
	\eqref{def_Cameron_Martin_joint}. Let $a$ be a bounded Borel
	function on $C\brac{[0,1];\R^n}$. Then, the variational formula
	(applied to the framework of this paper) from \cite[Theorem
	3.2]{Zha} states that 	
	\begin{align}
		\label{equation_varrep}
		-\frac{1}{h^2(\epsilon)}\ln \E{\exp\brac{-h^2(\epsilon)a(\eta^\epsilon)}}&=\inf_{w^\epsilon\in \mathcal{S}}\E{\frac{1}{2}\norm{w^\epsilon}^2_{\mathcal{S}}+a\brac{\eta^\eu}}\nonumber\\
		&=\inf_{w^\epsilon = (K_H\hat{u}^{\epsilon},K_{1/2}\hat{v}^{\epsilon})\in \mathcal{S}}\E{\frac{1}{2}\int_0^1 \abs{\hat{u}^\epsilon_s}^2+\abs{\hat{v}^\epsilon_s}^2 ds+a\brac{\eta^\eu}},
	\end{align}
	where the controlled deviations process $\eta^\eu$ is defined by
	\begin{align}
		\label{def_control_eta}
		\eta^\eu_t=\frac{X^\eu_t-\bar{X}_t}{\sqrt{\epsilon}h\brac{\epsilon}}
	\end{align}
	and the controlled processes $X^\eu$ and $Y^\eu$ are defined by
	\begin{align}
		\label{equation_controlled_XY}
		\begin{split}
			X^\eu_t&=x_0 +\int_0^t
			g\brac{X^\eu_s,Y^\eu_s}+\sqrt{\epsilon}h(\epsilon)f\brac{X^\eu_s,Y^\eu_s}
			\dot{u}^\epsilon_s ds\\&\quad +\int_0^t\sqrt{\epsilon}f\brac{X^\eu_s,Y^\eu_s}d W^H_s\\
			Y^\eu_t&=y_0+\int_0^t\frac{1}{\epsilon}c\brac{Y^\eu_s}+\frac{h(\epsilon)}{\sqrt{\epsilon}}\sigma\brac{Y^\eu_s}\dot{v}^\epsilon_s
			ds +\int_0^t\frac{1}{\sqrt{\epsilon}}\sigma\brac{Y^\eu_s}dB_s.
		\end{split}
	\end{align}
	Note that, based on \eqref{def_control_eta} and \eqref{equation_controlled_XY}, we can rewrite $\eta^\eu$ in the form 	\begin{align}		
		\label{equation_eta}
		\eta^\eu_t
		&=\int_0^t\frac{1}{\sqrt{\epsilon}h(\epsilon)}\left[ g\brac{X^\eu_s,Y^\eu_s}-\bar{g}(\bar{X}_s) \right] ds+
		\int_0^t f\brac{X^\eu_s,Y^\eu_s} \dot{u}^\epsilon_s ds\nonumber\\&+\frac{1}{h(\epsilon)}\int_0^t f\brac{X^\eu_s,Y^\eu_s}d W^H_s.
	\end{align}
	Let  $\mathcal{U}=\R^m$ and $\mathcal{V}=\R^p$. These are the
	spaces in which the control processes $u^\epsilon$ and $v^\epsilon$
	take values in, respectively. Define
	$\theta\brac{x,\eta,y^{(1)},y^{(2)},u^{(1)},u^{(2)},v^{(1)},v^{(2)},s,r}\colon
	\R^n\times \R^n \times \mathcal{Y} \times \mathcal{Y} \times
	\mathcal{U}\times \mathcal{U}\times \mathcal{V}\times
	\mathcal{V}\times [0,1] \times [0,1]
	$ by
	\begin{align}
		\label{definitionoffunctiontheta}
		&\theta\brac{x,\eta,y^{(1)},y^{(2)},u^{(1)},u^{(2)},v^{(1)},v^{(2)},s,r}
		=\brac{\nabla_y\phi\brac{x,y^{(1)}}
			\sigma\brac{y^{(1)}}v^{(1)}+\nabla_x \bar{g}
			(x)\eta} \nonumber \\& \hspace{4cm} +c_H f(x,y^{(1)})(s-r)^{H-3/2}s^{H-1/2}r^{1/2-H}u^{(2)}\mathds{1}_{[0,s]}(r).
	\end{align}
	Condition \ref{condH1} and Theorem
	\ref{theorem_Poisson_equation} guarantee that the function $\theta$
	is bounded in $x$, affine in $\eta$, $u^{(2)}$ and $v^{(1)}$ and
	bounded polynomially in $\abs{y}$.
	
	Next, we introduce the occupation measure $P^{\epsilon}$. Let $A_1$,
	$A_2$, $B$ and $\Gamma$ be Borel sets of $\mathcal{U}$, $\mathcal{V}$,
	$\mathcal{Y}=\R^d$ and $[0,1]$, respectively. Let $(X^\eu,Y^\eu)$ solve
	\eqref{equation_controlled_XY}. Associate with
	$\brac{Y^\eu,\hat{u}^{\epsilon},\dot{v}^{\epsilon}}$ a family of occupation
	measures $P^{\epsilon}$ defined by
	\begin{align}
		\label{def_occumeasure}
		P^{\epsilon}\brac{A_1\times A_2 \times B\times \Gamma}=\int_\Gamma  \mathds{1}_{A_1}(\hat{u}^\epsilon_s)\mathds{1}_{A_2}(\dot{v}^\epsilon_s) \mathds{1}_{B}(Y^\eu_s)ds.
	\end{align}
	\begin{definition}
		\label{defviablepairs}
		Let
		$F\brac{x,\eta,y^{(1)},y^{(2)},u^{(1)},u^{(2)},v^{(1)},v^{(2)},s,r}
		\colon
		\R^n\times \R^n \times \mathcal{Y} \times \mathcal{Y} \times
		\mathcal{U}\times \mathcal{U}\times \mathcal{V}\times
		\mathcal{V}\times [0,1] \times [0,1]$ be a function that has at most
		polynomial growth in $\abs{y}$. Let $\mathcal{L}$ be a second order
		elliptic partial differential operator and denote
		its domain by $\mathcal{D}(\mathcal{L})$. A pair $(\psi,P) \in
		C\brac{[0,1];\R^n}\times
		\mathcal{P}(\mathcal{U}\times\mathcal{V}\times\mathcal{Y}\times[0,1])$
		is called a viable pair with respect to $(\theta,\mathcal{L})$ if
		\begin{itemize}
			\item[-] The function $\psi\in C([0,1];\mathbb{R}^{n})$ is absolutely continuous.
			\item[-] The measure $P$ is integrable in the sense that
			\begin{align*}
				\int_{\mathcal{U}\times \mathcal{V}\times \mathcal{Y}\times[0,1]}\left[ \abs{u}^2+\abs{v}^2+\abs{y}^2 \right] P(dudvdyds)<\infty.
			\end{align*}
			\item[-] For all $t \in [0,1]$,
			\begin{align*}
				&\psi_t=\int_{\mathcal{U}^2\times
					\mathcal{V}^2\times \mathcal{Y}^2\times
					[0,t]^2}F\brac{\bar{X}_s,\psi_s,y^{(1)},y^{(2)},u^{(1)},u^{(2)},v^{(1)},v^{(2)},s,r}
				\\
				& \qquad\qquad\qquad\qquad\qquad\qquad\qquad\qquad\qquad P\otimes
				P(du^{(1)}du^{(2)}dv^{(1)}dv^{(2)}dy^{(1)}dy^{(2)}dsdr).
			\end{align*}
			\item[-] For all $t \in [0,1]$, it holds that
			\begin{align}
				\label{decompdeP}
				P(dudvdydt)=\nu_{y,t}(dudv)\mu(dy)dt,
			\end{align}
			where  $\nu_{y,t}$ is a kernel on $\mathcal{U}\times
			\mathcal{V}$ dependent on $y\in \mathcal{Y}$ and $t\in
			[0,1]$, while $\mu$ is the unique invariant
			measure associated with the operator $\mathcal{L}$.
		\end{itemize}
		In order to indicate that the pair $(\psi,P)$ is viable with respect
		to $(F,\mathcal{L})$, we write $(\psi,P) \in \mathcal{V}(F,\mathcal{L})$.
	\end{definition}
	The controlled process \eqref{def_control_eta} and the definition of
	viable pairs (Definition \ref{defviablepairs}) will be used to prove the theorem below.
	\begin{theorem}
		\label{theobeforevaria}
		Let Conditions \ref{condH1} and either \ref{condH2-A} or
		\ref{condH2-B} be satisfied. Then, the process
		$\{X^\epsilon\colon \epsilon>0\}$ from \eqref{originalsystem} satisfies the moderate
		deviations principle, with the action functional
		$S^{H}(\Phi)$ given by
		\begin{align}
			\label{rate_function}
			S^{H}(\Phi)= \inf_{(\Phi,P)\in\mathcal{V}(\theta,\mathcal{L})}\left[ \frac{1}{2}  \int_{\mathcal{U}\times \mathcal{V}\times \mathcal{Y}\times[0,1]}\left[ \abs{u}^2+\abs{v}^2 \right]  P(dudvdyds)\right]
		\end{align}
		with the convention that the infimum over the empty set is $\infty$.
	\end{theorem}
	\begin{remark}
		As will be shown in the proof of Theorem \ref{theobeforevaria}, Theorem \ref{theorem_rate_function}
		follows directly from Theorem \ref{theobeforevaria}.
	\end{remark}
	
	\begin{remark}\label{R:CoreIdea}
		In this remark we discuss one of the key ideas that allows to naturally generalize the computations to the $H \neq 1/2$ case from the $H=1/2$ case. In the course of the proof, we will need to handle terms of the form $\int_0^t f\brac{X^\eu_s,Y^\eu_s} \dot{u}^\epsilon_s ds$, where $u^\epsilon$ is the control process introduced in the beginning of this section. Roughly speaking, if $H=1/2$ and $P^{\epsilon}$ is the occupational measure defined as in  (\ref{def_occumeasure}), then one has $\dot{u}^\epsilon=\hat{u}^\epsilon$ and thus
		\begin{align}
			\int_0^t f\brac{X^\eu_s,Y^\eu_s} \dot{u}^\epsilon_s ds&=\int_{\mathcal{U}\times\mathcal{V}\times\mathcal{Y}\times[0,t]}f(X^{\eu}_s,y)uP^{\epsilon}(d u dv dyds),\nonumber
		\end{align}
		and then after establishing tightness of $(X^\eu, P^{\epsilon})$ one can study  its limit. This approach does not work exactly like that in the case where the Hurst parameter $H\neq 1/2$. In order to generalize this idea for the case $H\neq 1/2$, we first notice that one can write that
		\[
		\frac{d}{ds} u^{\epsilon}_{s}=\frac{d}{ds} \left[K_{H}\hat{u}^\epsilon\right]_{s},
		\]
		where $K_{H}$ is the operator associated to fBm, see Appendix \ref{SS:fBm_preliminaries}.   With this observation at hand we then write
		\begin{align}
			&\int_0^t f\brac{X^\eu_s,Y^\eu_s} \dot{u}^\epsilon_s ds=\int_0^t f\brac{X^\eu_s,Y^\eu_s} \frac{d}{ds} \left[K_{H}\hat{u}^\epsilon\right]_{s} ds\nonumber\\
			&\quad=\int_0^t f\brac{X^\eu_s,Y^\eu_s}  \brac{c_Hs^{H-1/2}\int_0^s (s-r)^{H-3/2}r^{1/2-H}\hat{u}^{\epsilon}_rdr} ds\nonumber\\
			&\quad=c_H\int_{\mathcal{U}^2\times
				\mathcal{V}^2\times \mathcal{Y}^2\times
				[0,t]^2} f(X^\eu_s,y^{(1)})(s-r)^{H-3/2}s^{H-1/2}r^{1/2-H} u^{(2)} \mathds{1}_{[0,s]}(r)\nonumber\\
			& \qquad\qquad\qquad\qquad\qquad\qquad\qquad P^{\epsilon}\otimes
			P^{\epsilon}(du^{(1)}du^{(2)}dv^{(1)}dv^{(2)}dy^{(1)}dy^{(2)}dsdr),
		\end{align}
		which is what allows us then to take limits. The details are in Section \ref{proofoftheo2}.
	\end{remark}
	
	\section{Examples}\label{S:Examples}
	\subsection{Fractional financial model}
	In \cite[Chapter 4]{Shi99}, the author collects various empirical studies which observe persistence or long memory phenomena in financial data such as financial indexes and currency cross rates, among others. This motivates us, for the first example of this paper, to consider the multiscale volatility model
	\begin{align}
		\label{equation_fractional_vol}
		\begin{split}
			dX^\epsilon_t&=Y^\epsilon_tdt+\sqrt{\epsilon}\tau dW^H_t,\\
			dY^\epsilon_t&=\beta\brac{\theta-Y^\epsilon_t}dt+v\sqrt{Y_t}dB_t.
		\end{split}
	\end{align}
	We assume $\tau>0$ and $\beta,\theta,v$ are real constants such that $2\beta\theta\geq v^2$. $\brac{W^H,B}$ is an independent pair of one-dimensional fractional Brownian motion of Hurst parameter $H>1/2$ and one-dimensional Brownian motion. $X^\epsilon$ is a financial instrument with a perturbed fractional Brownian noise (in order to account for the long memory effect). The fast volatility process $Y^\epsilon$ follows the Cox-Ingersoll-Ross model of interest rate and the assumption $2\beta\theta\geq v^2$ ensures that $Y^\epsilon$ is strictly positive. It is also worth noting that adding fractional Brownian noise to financial models to simulate long memory has been an increasingly common practice in literature, see \cite{Che03,ComRen,ForZha17,HJL19,SotEsk,Shi99}.

	The stochastic differential equation of $Y^\epsilon$ in \eqref{equation_fractional_vol} has the Cox–Ingersoll–Ross process as its unique solution, which implies the process $X^\epsilon$ as an integral function of $Y^\epsilon$ plus a fractional Brownian noise term is well-defined. Moreover, in the context of the previous section, the invariant measure $\mu$ has the Gamma density (\cite[Section 33.4]{FPSS11})
	\begin{align*}
		\mu(y)=\frac{\brac{2\beta/v^2}^{2\beta\theta/v^2}}{\Gamma\brac{2\beta\theta/v^2}}y^{2\beta\theta/v^2-1}e^{-2\beta y/v^2},\text{ for $y\geq 0$.}
	\end{align*}
	Then according to \cite[Theorem 1]{BGS}, $X^\epsilon$ converges in probability in $C\brac{[0,1]}$ to
	\begin{align*}
		\bar{X}_t=\int_{\R}y\mu(dy)=\theta.
	\end{align*}
	The Poisson equation at \eqref{equation_Poisson_g} has an unique solution $\phi(y)$ due to Theorem \ref{theorem_Poisson_equation} and this solution satisfies $\phi'(y)=-\frac{1}{2\beta}$. This implies the operator $Q^H$ defined at \eqref{def_Q} is
	\begin{align*}
		Q^H=\tau^2\dot{K}_H \dot{K}_H^*+\brac{\frac{\tau}{2\beta}}^2,
	\end{align*}
	which is invertible since $\tau>0$ (see Lemma \ref{lemma_Q_invertible}). Here $\dot{K}_H \dot{K}_H^*$ is the operator
	\begin{align*}
		\left[\dot{K}_H \dot{K}_H^* h\right](t)&=c_H^2t^{H-1/2}\int_0^t(t-z)^{H-3/2}z^{1-2H}\int_z^1 (s-z)^{H-3/2}s^{H-1/2}h(s)dsdz.
	\end{align*}

	Let us now discuss moderate deviations of $X^\epsilon$. We have already established $Y^\epsilon$ has an invariant measure $\mu$, which makes Condition \ref{condH1} redundant for the model \eqref{equation_fractional_vol}.  Moreover, if we use the notation of Section \ref{section_intro} then the equation of $X^\epsilon$ at \eqref{equation_fractional_vol} has $f=\tau$ and $\brac{\frac{d}{dx}\phi(y)}f=0$. Therefore, based on the proofs of Lemma \ref{lemma_estimates_notyoung} and Lemma \ref{lemma_estimates_young}, Condition \ref{condH2-B} in this setting simplifies to $\abs{f(y)}\leq C\brac{1+\abs{y}^{D_f}}$ and $0\leq D_f< 1$, which is clearly satisfied for $f=\tau$. Then, as long as there exists $\beta\in[0,1]$ such that $\beta+H>1$ and $h(\epsilon)^{-1}\epsilon^{-\frac{\beta}{2}}\to 0$ as $\epsilon\to 0$, Theorem \ref{theorem_rate_function} asserts that for the moderate deviations process $\brac{X^\epsilon-\bar{X}}/h(\epsilon)\sqrt{\epsilon}$, its action functional, when finite, takes the form
	\begin{align*}
		S^H(\Phi)= \int_0^1\dot{\Phi}_s \brac{Q^H}^{-1}\dot{\Phi}_sds.
	\end{align*}
	
	\subsection{Fractional Langevin equation}
	For the second example, we consider the multiscale model
	\begin{align}
		\label{equation_fractional_langevin}
		\begin{split}
			dX^\epsilon_t&=\brac{-Q'\brac{Y^\epsilon_t}- V'\brac{X^\epsilon_t}}dt+\sqrt{\epsilon}\sqrt{2D} dW^H_t,\\
			dY^\epsilon_t&=-\frac{1}{\epsilon} Q'\brac{Y^\epsilon_t}dt+\frac{1}{\sqrt{\epsilon}}\sqrt{2D} dB_t.
		\end{split}
	\end{align}
	The equation of $X^\epsilon$ can be viewed as a rescaled Langevin equation with a fractional Brownian noise. A simpler version of this fractional Langevin equation that does not contain a fast process $Y^\epsilon$ was studied in \cite{AMP,CKM,GJR} among others. We assume
	that
	\begin{itemize}
		\item[-] $\mathcal{Y}$ is the one-dimensional unit torus.
		\item[-] There is a constant $C$ such that $
		\abs{Q'(y)}\leq C\brac{1+\abs{y}}
		$ and $\sup_{x\in\mathbb{R}}\sum_{k=1}^{3}\abs{ V^{(k)}(x)}\leq C$.
		\item[-] $Q'(y)$ is Lipschitz.
		\item[-] $V'''(x),Q'''(y)$ are continuous.
		\item[-] $D$ is a real-valued non-zero constant.
	\end{itemize}
	Our assumption implies that $Q'(y),V'(x)$ are Lipschitz and that $\abs{Q'(y)}+\abs{V'(x)}\leq C\brac{1+\abs{y}}$, so that there is a unique strong solution to \eqref{equation_fractional_langevin} based on \cite[Theorem 2.2]{GueNua}.
	
	Next, we consider averaging of $X^\epsilon$. Since $\mathcal{Y}$ is the unit torus, Condition 3 in \cite[Theorem 1]{BGS} is not needed for ergodicity of $Y^\epsilon$. Condition 1 in \cite[Theorem 1]{BGS} is met by our second and fourth assumptions for \eqref{equation_fractional_langevin} above. Thus, we conclude $X^\epsilon$ converges in probability on $C\brac{[0,1]}$ to
	\begin{align*}
		\bar{X}_t=\int_0^t\brac{-
			\int_{\mathcal{Y}}{Q'(y)}\mu(dy)-V'\brac{\bar{X}_s}}ds
	\end{align*}
	where, according to \cite{PavStu}, the invariant measure $\mu$ is the Gibbs measure
	\begin{align*}
		\mu(dy)=\frac{1}{Z}e^{-{Q(y)}/{D}},\qquad Z=\int_{\mathcal{Y}}e^{-{Q(y)}/{D}}dy.
	\end{align*}
	The Poisson equation at \eqref{equation_Poisson_g} becomes
	\begin{align*}
		-Q'(y)\phi'(y)+D \phi''(y)=\bar{Q'}-Q'(y), \qquad \int_\mathcal{Y}\phi(y)\mu(dy)=0
	\end{align*}
	such that $\bar{Q'}=\int_{\mathcal{Y}}{Q'(y)}\mu(dy)$. Its solution satisfies
	\begin{align*}
		\phi'(y)=\frac{\bar{Q'}}{D} e^{Q(y )/D}\int_0^y e^{-Q(\xi)/D}d\xi+M e^{Q(y )/D}+1
	\end{align*}
	where the constant $M$ is
	\begin{align*}
		M&=-\brac{\frac{\bar{Q'}}{D}\int_\mathcal{Y} e^{-Q(y )/D}\int_0^y e^{Q(\rho )/D}\int_0^\xi e^{-Q( \xi)/D}d\xi d\rho dy+\int_\mathcal{Y}y e^{-Q(y )/D}dy}\\
		&\hspace{18em}\brac{\int_\mathcal{Y}e^{-Q(y )/D}\int_0^y e^{Q(\xi)/D}d\xi dy}^{-1}.
	\end{align*}
	At this point, we can consider moderate deviations of $X^\epsilon$. In the notation of the previous section, we have $g(x,y)=-Q'(y)-V'(x)$, $c(y)=- Q'(y),f=\sigma=\sqrt{2D}$.  Since $\mathcal{Y}$ is the unit torus, the first recurrence assumption in Condition \ref{condH1} and $D_f+D_g<1$ in Condition \ref{condH2-B} are no longer needed. In addition, the fact that $\frac{d}{dx}\phi(y)=0$ makes redundant the assumption $M_k/2+H>1$ in Condition \ref{condH2-B}. Then the rest of Conditions \ref{condH1} and \ref{condH2-B} are satisfied by \eqref{equation_fractional_langevin}. In particular, we have $g(x,y)\in \mathcal{C}^{2,\alpha}(\R^d\times \mathcal{Y})$ since $V''(x), V'''(x)$ are bounded. Next, the operator $Q^H$ in \eqref{operateurQ} becomes
	\begin{align*}
		Q^H&=2D\dot{K}_H\dot{K}^*_H
		\\&+\frac{2D}{Z}\int_\mathcal{Y} e^{-Q(y )/D} \brac{\frac{\bar{Q'}}{D} e^{Q(y )/D}\int_0^y e^{-Q(\xi)/D}d\xi+M e^{Q(y )/D}+1}^2 dy
	\end{align*}
	where $\dot{K}_H \dot{K}_H^*$ is
	\begin{align*}
		\left[\dot{K}_H \dot{K}_H^* h\right](t)&=c_H^2t^{H-1/2}\\&\int_0^t(t-z)^{H-3/2}z^{1-2H}\int_z^1 (s-z)^{H-3/2}s^{H-1/2}h(s)dsdz.
	\end{align*}
	Thus, under the condition that $Q^H$ is invertible and there exists $\beta\in[0,1]$ such that $\beta+H>1$ and $h(\epsilon)^{-1}\epsilon^{-\frac{\beta}{2}}\to 0$ as $\epsilon\to 0$, Theorem \ref{theorem_rate_function} says for the moderate deviations process $\brac{X^\epsilon-\bar{X}}/h(\epsilon)\sqrt{\epsilon}$, its action functional, when finite, is
	\begin{align*}
		S^H(\Phi)= \int_0^1\brac{\dot{\Phi}_s+ V''\brac{\bar{X}_s}\Phi_s}\brac{Q^H}^{-1}\brac{\dot{\Phi}_s+ V''\brac{\bar{X}_s}\Phi_s} ds.
	\end{align*}

	\begin{remark}
		Here we compare the result above to the moderate deviations of $X^\epsilon$ in
		\begin{align}
			\label{equation_ito_langevin}
			\begin{split}
				dX^\epsilon_t&=\brac{-Q'\brac{Y^\epsilon_t}- V'\brac{X^\epsilon_t}}dt+\sqrt{\epsilon}\sqrt{2D} dW_t,\\
				dY^\epsilon_t&=-\frac{1}{\epsilon} Q'\brac{Y^\epsilon_t}dt+\frac{1}{\sqrt{\epsilon}}\sqrt{2D} dB_t.
			\end{split}
		\end{align}
		We assume $W$ is a Brownian motion independent from $B$ and $\mathcal{Y}$ is the one-dimensional unit torus. Under appropriate conditions, \cite[Theorem 2.1]{MorseSpiliopoulosMDP} says the moderate deviations action functional of \eqref{equation_ito_langevin}, when finite, is
		\begin{align*}
			S^{1/2}(\Phi)= \int_0^1\brac{\dot{\Phi}_s+ V''\brac{\bar{X}_s}\Phi_s}\brac{Q^{1/2}}^{-1}\brac{\dot{\Phi}_s+ V''\brac{\bar{X}_s}\Phi_s} ds.
		\end{align*}
		such that
		\begin{align*}
			Q^{1/2}=2D
			+\frac{2D}{Z}\int_\mathcal{Y} e^{-Q(y )/D} \brac{\frac{\bar{Q'}}{D} e^{Q(y )/D}\int_0^y e^{-Q(\xi)/D}d\xi+M e^{Q(y )/D}+1}^2 dy.
		\end{align*}
		Notice in this particular situation $f(x,y)=\sqrt{2D}$ (i.e. independent of $y$) and thus we have continuity of the mapping $H\mapsto S^{H}$  at $H=1/2$ (see Remark \ref{R:BM_fBM_comparison}).
	\end{remark}
	
	\section{Proof of Theorem \ref{theobeforevaria}}
	\label{proofoftheo2}
	The proof of Theorem \ref{theobeforevaria} will be divided into five
	subsections. In Subsections \ref{prooftightness} and \ref{proofviablepair}, we prove tightness and
	convergence of the pair $(\eta^{\eu},P^{\epsilon})$, respectively. In
	Subsection \ref{prooflaplacelower}, we prove the Laplace principle lower bound. In
	Subsection \ref{proofcompactness}, we prove that the level sets of $S(\cdot)$ are
	compact. Finally, in Subsection \ref{prooflaplaceupper}, we prove the Laplace principle
	upper bound and the representation formula of Theorem \ref{theorem_rate_function}. The main additional work that needs to be done due to the effect of the fBm is seen in the bounds that we need in order to prove tightness (see also Appendix \ref{A:AppendixB}) and in the proof of the upper bound  in Subsection \ref{prooflaplaceupper}.
	\subsection{Proof of tightness}
	\label{prooftightness}
	The main result of this section is the following proposition on
	tightness.
	\begin{proposition}
		\label{propontightness}
		Let Conditions \ref{condH1} and either \ref{condH2-A} or
		\ref{condH2-B} be satisfied. Consider any family
		$\left\{w^\epsilon\colon \epsilon >0  \right\}$ of controls in
		$\mathcal{S}$ satisfying, for some $N < \infty$,
		\begin{align*}
			\sup_{\epsilon>0}\norm{w^\epsilon}^2_{\mathcal{S}}=\sup_{\epsilon>0}\int_0^1 \abs{\hat{u}^\epsilon_s}^2+\abs{\hat{v}^\epsilon_s}^2 ds< N
		\end{align*}
		almost surely. Then, the family $\left\{ (\eta^{\eu},P^{\epsilon})\colon \epsilon >0
		\right\}$ is tight.

	\end{proposition}
	The proof of Proposition \ref{propontightness} will be divided into two parts which
	are the subject of Subsections \ref{tightnessofPepsilon} and
	\ref{tightnessofetacontrol}.
	\subsubsection{Tightness of $\left\{P^{\epsilon}\colon \epsilon >0
		\right\}$ in $\mathcal{P}(\mathcal{U}\times \mathcal{V}\times
		\mathcal{Y} \times [0,1])$}
	\label{tightnessofPepsilon}
	The argument for tightness is similar to the argument
	for tightness in \cite{HuSalSpi}. As a first step, we claim that
	\begin{align*}
		\Lambda(P)=\int_{\mathcal{U}\times \mathcal{V}\times
			\mathcal{Y}\times[0,1]}\left[ \abs{u}^2+\abs{v}^2+\abs{y}^{2} \right] P(dudvdydt).
	\end{align*}
	is a tightness function from $\mathcal{P}(\mathcal{U}\times \mathcal{V}\times \mathcal{Y}\times[0,1])$ to $\R\cup \{\infty\}$. Since $\Lambda$ is bounded from below, it is sufficient to show that for every $k\in \N$, the level sets
	\begin{align*}
		L_k=\{P\in \mathcal{P}(\mathcal{U}\times \mathcal{V}\times \mathcal{Y}\times[0,1]): \Lambda(P)\leq k  \}
	\end{align*}
	are relatively compact.
	For $\epsilon>0$, let $M$ be a positive constant large
	enough so that $k/M < \epsilon$, define $\lambda(u,v,y,t)=\abs{u}^2+\abs{v}^2+\abs{y}^{2}$ and
	\begin{align*}
		A^\epsilon=\{(u,v,y,t)\in \mathcal{U}\times
		\mathcal{V}\times \mathcal{Y}\times[0,1] \colon \abs{\lambda(u,v,y,t)} >M \}.
	\end{align*} By Chebyshev's inequality,
	\begin{align*}
		\sup_{P\in L_k} P(A^\epsilon) &\leq
		\frac{1}{M}\int_{\left\{(u,v,y,t) \in \mathcal{U}\times
			\mathcal{V}\times \mathcal{Y}\times[0,1]\colon
			\abs{\lambda(u,v,y,t)}\geq
			k\right\}}\abs{\lambda(u,v,y,t)}P(dudvdydt)
		\\
		& \leq \frac{\Lambda(P)}{M}<\frac{k}{M}< \epsilon.
	\end{align*}
	Therefore, we get
	\begin{equation*}
		\sup_{P\in L_k}P\brac{\brac{\mathcal{U}\times \mathcal{V}\times \mathcal{Y}\times[0,1]}\setminus A^\epsilon}>1-\epsilon.
	\end{equation*}
	Since $\brac{\mathcal{U}\times \mathcal{V}\times
		\mathcal{Y}\times[0,1]}\setminus A^\epsilon$ is also compact, this
	implies that $L_k$ is a tight set of measures and $\Lambda$ is a tightness function on $\mathcal{P}(\mathcal{U}\times \mathcal{V}\times \mathcal{Y}\times[0,1])$.
	
	For the second step, define $G:\mathcal{P}(\mathcal{P}(\mathcal{U}\times \mathcal{V}\times \mathcal{Y}\times[0,1]))\to \R\cup \{\infty\}$ by
	\begin{align*}
		G(\nu)=\int_{\mathcal{P}(\mathcal{U}\times \mathcal{V}\times \mathcal{Y}\times[0,1])} \Lambda(x)\nu(dx).
	\end{align*}	
	Then, according to \cite[Theorem A.3.17]{DupEll}, $G$
	is a tightness function on
	$\mathcal{P}(\mathcal{P}(\mathcal{U}\times
	\mathcal{V}\times \mathcal{Y}\times[0,1]))$. Moreover,
	the same theorem states that
	$\{P^{\epsilon}:\epsilon>0\}$ is a tight family in $\mathcal{P}(\mathcal{U}\times \mathcal{V}\times \mathcal{Y}\times[0,1])$ as long as
	\begin{align*}
		\sup_{\epsilon>0}G\brac{\mathcal{L}\brac{P^{\epsilon}}}<\infty,
	\end{align*}
	which is equivalent to
	\begin{align*}
		\sup_{\epsilon>0}\E{\Lambda\brac{P^{\epsilon}}}< \infty.
	\end{align*}
	As the above holds by Lemma \ref{lemma_boundedcontrol}
	and Lemma \ref{lemma_estimatefromMorKos2}, we get that
	indeed $\left\{P^{\epsilon}\colon \epsilon >0
	\right\}$ is tight in $\mathcal{P}(\mathcal{U}\times \mathcal{V}\times
	\mathcal{Y} \times [0,1])$.
	\subsubsection{Tightness of $\left\{ \eta^{\eu}\colon \epsilon >0
		\right\}$ on $C\brac{[0,1];\R^n}$}
	\label{tightnessofetacontrol}
	Let  $\omega_f(\delta)=\sup_{\abs{s-t}<\delta} \abs{f(s)-f(t)}$ be the modulus of continuity of a function $f$ on $C\brac{[0,1];\R^n}$. According to \cite[Theorem 7.3]{Bil}, the family $\{\eta^\eu:\epsilon>0\}$ is tight on $C\brac{[0,1];\R^n}$ if and only if
	\begin{itemize}
		\item[-] For each positive $\delta$, there
		exist an $a,\delta_0 >0$ such that $P\brac{\abs{\eta^\eu_0}\geq a}\leq \delta $ for $\epsilon\leq \delta_0$.
		\item[-]  For all $a>0$, $
		\lim_{\delta\to 0} \limsup_{\epsilon \to 0} P\brac{\omega_{\eta^\eu}(\delta)\geq a}=0.$
	\end{itemize}
	We only need to check the second condition above since
	the first condition is automatically true as
	$\eta^\eu_0=0$. Recall from \eqref{equation_eta} that $\eta^\eu$ is given by
	\begin{align*}		
		\eta^\eu_t
		&=\frac{1}{\sqrt{\epsilon}h(\epsilon)}\int_0^t\left[ g\brac{X^\eu_s,Y^\eu_s}-\bar{g}(\bar{X}_s) \right] ds+
		\int_0^t f\brac{X^\eu_s,Y^\eu_s} \dot{u}^\epsilon_s ds\nonumber\\&+\frac{1}{h(\epsilon)}\int_0^t f\brac{X^\eu_s,Y^\eu_s}d W^H_s.
	\end{align*}
	A combination of the Poisson equation stated in \eqref{equation_Poisson_g} and It\^{o}'s formula yields
	\begin{align*}
		&\int_0^t\frac{1}{\sqrt{\epsilon}h(\epsilon)}\left[ g\brac{X^\eu_s,Y^\eu_s}-\bar{g}\brac{X^\eu_s} \right]ds=\int_0^t \nabla_y\phi\brac{X^\eu_s,Y^\eu_s} \sigma\brac{Y^\eu_s}\dot{v}^\epsilon_s ds+R^\epsilon_1(t),
	\end{align*}
	where
	\begin{align}
		\label{equation_R1}
		R^\epsilon_1(t)&=-\frac{\sqrt{\epsilon}}{h(\epsilon)}\brac{\phi\brac{X^\eu_t,Y^\eu_t}-\phi\brac{x_0,y_0}}+\frac{\sqrt{\epsilon}}{h(\epsilon)}\int_0^t
		\nabla_x\phi\brac{X^\eu_s,Y^\eu_s}
		g\brac{X^\eu_s,Y^\eu_s}ds\nonumber\\
		&  \quad
		+\epsilon\int_0^t\nabla_x\phi\brac{X^\eu_s,Y^\eu_s}f\brac{X^\eu_s,Y^\eu_s}\dot{u}^\epsilon_sds\nonumber
		\\
		& \quad
		+\frac{1}{h(\epsilon)}\int_0^t\nabla_y\phi\brac{X^\eu_s,Y^\eu_s}\sigma\brac{Y^\eu_s}dB_s\nonumber\\&
		\quad +\frac{{\epsilon}}{h(\epsilon)}\int_0^t\nabla_x\phi\brac{X^\eu_s,Y^\eu_s}f\brac{X^\eu_s,Y^\eu_s}dW^H_s.
	\end{align}
	Therefore, we can rewrite $\eta^\eu$ as
	\begin{align}
		\label{equation_eta_new}
		\eta^\eu_t &=\int_0^t \nabla_y\phi\brac{X^\eu_s,Y^\eu_s} \sigma\brac{Y^\eu_s}\dot{v}^\epsilon_sds+\int_0^tf\brac{X^\eu_s,Y^\eu_s}\dot{u}^\epsilon_s ds\nonumber\\& +\frac{1}{h(\epsilon)}\int_0^t f\brac{X^\eu_s,Y^\eu_s}d W^H_s
		+\frac{1}{\sqrt{\epsilon}h(\epsilon)}
		\int_0^t\left[\bar{g}\brac{X^\eu_s}-\bar{g}\brac{\bar{X}_s}
		\right] ds+R^\epsilon_1(t)\nonumber \\
		&= D_1^{\epsilon}(t) + D_2^{\epsilon}(t)+D_3^{\epsilon}(t)+D_4^{\epsilon}(t)+R^\epsilon_1(t).
	\end{align}
	In combination with Markov's inequality, Lemma
	\ref{lemma_estimates_notyoung} implies tightness of
	$\left\{ D_1^{\epsilon} \colon \epsilon >0 \right\}$
	and $\left\{ D_2^{\epsilon} \colon \epsilon >0
	\right\}$. Lemma \ref{lemma_estimates_young} implies
	tightness of $\left\{ D_3^{\epsilon} \colon \epsilon
	>0 \right\}$ and Lemma \ref{lemma_estimate_eta}
	implies tightness of $\left\{ D_4^{\epsilon} \colon
	\epsilon >0 \right\}$. It remains to prove the
	tightness of $\left\{ R^\epsilon_1 \colon \epsilon >0
	\right\}$. The estimates at
	\eqref{estimate_Poissonsolution_h1},
	\eqref{estimate_Poissonsolution_h2} combined with
	Lemma \ref{lemma_estimatefromMorKos2} and the fact
	that $\frac{\sqrt{\epsilon}}{h(\epsilon)}\to 0$ imply
	that the first term in equation \eqref{equation_R1}
	converges to zero in probability, which implies
	tightness on $C\brac{[0,1];\R^n}$. Furthermore, tightness
	of the remaining integral terms in \eqref{equation_R1}
	is implied by Markov's inequality, Lemma
	\ref{lemma_estimates_notyoung} and Lemma
	\ref{lemma_estimates_young}. This shows that $\left\{ R^\epsilon_1 \colon \epsilon >0
	\right\}$ is tight and hence that $\left\{ \eta^{\eu}\colon \epsilon >0
	\right\}$ is indeed tight on $C\brac{[0,1];\R^n}$.
	\subsection{Proof of existence of a viable pair}
	\label{proofviablepair}
	In the previous subsection, we have proved that the family of
	processes $\left\{ (\eta^{\eu},P^{\epsilon})\colon \epsilon >0
	\right\}$ is tight (see Proposition \ref{propontightness}). It follows that for any subsequence of $\epsilon$
	converging to 0, there exists a subsubsequence of $\left\{ (\eta^{\eu},P^{\epsilon})\colon \epsilon >0
	\right\}$ which is convergent in distribution to some limit
	$(\bar{\eta},\bar{P})$. The goal of this subsection is to show that
	$(\bar{\eta},\bar{P})$ is a viable pair with respect to
	$(\theta,\mathcal{L})$ according to Definition \ref{defviablepairs} (where
	$\mathcal{L}$ is the generator defined in \eqref{def_generator_normalized_Y}).
	
	By the Skorokhod Representation Theorem, we may assume that
	$\eta^\eu$ converges to $\bar{\eta}$ almost surely along any
	subsequence. This will allow us to obtain an equation satisfied by
	$\bar{\eta}$ since we can study the almost sure limits of each
	individual summand in the representation of $\eta^{\eu}$ we had obtained
	in \eqref{equation_eta_new}. Recall that we had
	\begin{align}
		\label{recallrepetacontrol}
		\eta^\eu_t &=\int_0^t
		\nabla_y\phi\brac{X^\eu_s,Y^\eu_s}
		\sigma\brac{Y^\eu_s}\dot{v}^\epsilon_sds+\int_0^tf\brac{X^\eu_s,Y^\eu_s}\dot{u}^\epsilon_s
		ds \nonumber \\& +\frac{1}{h(\epsilon)}\int_0^t f\brac{X^\eu_s,Y^\eu_s}d W^H_s
		+\frac{1}{\sqrt{\epsilon}h(\epsilon)}
		\int_0^t\left[\bar{g}\brac{X^\eu_s}-\bar{g}\brac{\bar{X}_s}
		\right] ds+R^\epsilon_1(t).
	\end{align}
	Note that we can write the term before last as
	\begin{align}
		\label{equation_eta_g_part}
		\int_0^t\frac{1}{\sqrt{\epsilon}h(\epsilon)}\brac{\bar{g}\brac{X^\eu_s}-\bar{g}\brac{\bar{X}_s}}ds=\int_0^t \nabla_x \bar{g} (\bar{X}_s)\eta^\eu_sds+R^\epsilon_2(t),
	\end{align}
	where the remainder term $R^\epsilon_2(t)$ is given by
	\begin{align*}
		R^\epsilon_2(t)=\frac{1}{2}\int_0^t\nabla^2_x\bar{g}(\zeta_s)\eta^\eu_s\brac{X^\eu_s-\bar{X}_s}^2ds
	\end{align*}
	with $\zeta_s$ being a point in between $X^\eu_s$ and
	$\bar{X}_s$. Under either Condition \ref{condH2-A} or
	Condition \ref{condH2-B}, $\nabla^2_x\bar{g}(x)$ is bounded,
	so that we can write
	\begin{align}
		\label{bornepourr2}
		R^\epsilon_2(t)\leq \int_0^1 \abs{\eta^\eu_s} \abs{X^\eu_s-\bar{X}_s}ds.
	\end{align}
	Lemma \ref{lemma_limit_R2} assesses the convergence to zero of
	$R^\epsilon_2(t)$, and \cite[Lemma 3.2]{DupuisSpiliopoulos} addresses the
	convergence of all the other terms at play in
	\eqref{recallrepetacontrol} and \eqref{equation_eta_g_part} except for
	one, namely
	\begin{align*}
		A^{\eu}(t)=\int_0^tf\brac{X^\eu_s,Y^\eu_s} \dot{u}^\epsilon_s ds.
	\end{align*}
	In order to deal with this last term, let us introduce the term
	\begin{align*}
		B^{\eu}(t)=\int_0^tf\brac{\bar{X}_s,Y^\eu_s} \dot{u}^\epsilon_s ds.
	\end{align*}
	and prove that it has the same limit as $A^{\eu}(t)$. First, note
	that if we assume that Condition \ref{condH2-B} holds, the
	assumption that $f$ does not depend on $x$ implies that
	$A^{\eu}(t)=B^{\eu}(t)$ even before taking limits. If instead we
	assume that Condition \ref{condH2-A} holds, we can
	use the Lipschitz continuity of $f(x,y)$ to write
	\begin{align*}
		\abs{A^{\eu}(t)-B^{\eu}(t)}
		\leq & \int_0^t \abs{X^\eu_s-\bar{X}_s} \abs{\dot{u}^\epsilon_s} ds
		\leq \sup_{0\leq s\leq 1}\abs{X^\eu_s-\bar{X}_s}\int_0^t\abs{\dot{u}^\epsilon_s} ds.
	\end{align*}
	Proposition \ref{prop_dotKHinL2} and the
	fact that $X^\eu$ converges to $\bar{X}$ in probability then
	imply that
	\begin{align}
		\label{convergenceAversB}
		\abs{A^{\eu}(t)-B^{\eu}(t)}\to 0
	\end{align}
	almost surely as $\epsilon$ goes to 0.
	Therefore identifying the limit of $A^{\eu}(t)$ is the same as
	identifying the weak limit of $B^{\eu}(t)$. Using the definition
	of our occupation measures given by \eqref{def_occumeasure}
	and Lemma \ref{lemma_derivative_control}, we can rewrite $B^{\eu}(t)$ as
	\begin{align*}
		B^{\eu}(t)&= \int_0^tf\brac{\bar{X}_s,Y^\eu_s} \brac{c_Hs^{H-1/2}\int_0^s (s-r)^{H-3/2}r^{1/2-H}\hat{u}_rdr} ds
		\\&=c_H\int_{\mathcal{U}^2\times
			\mathcal{V}^2\times \mathcal{Y}^2\times
			[0,t]^2} f(\bar{X}_s,y)(s-r)^{H-3/2}s^{H-1/2}r^{1/2-H} u^{(2)} \mathds{1}_{[0,s]}(r)\\
		& \qquad\qquad\qquad\qquad\qquad\qquad P\otimes
		P(du^{(1)}du^{(2)}dv^{(1)}dv^{(2)}dy^{(1)}dy^{(2)}dsdr).
	\end{align*}
	In order to somewhat compactify notation, let us introduce the function
	\begin{align*}
		&k\brac{y^{(1)},y^{(2)},u^{(1)},u^{(2)},v^{(1)},v^{(2)},s,r}\\&=c_H f(\bar{X}_s,y)(s-r)^{H-3/2}s^{H-1/2}r^{1/2-H}u^{(2)}\mathds{1}_{[0,s]}(r)
	\end{align*}
	as well as, for $0<\zeta<s$, the sequence
	\begin{align*}
		&k^\zeta\brac{y^{(1)},y^{(2)},u^{(1)},u^{(2)},v^{(1)},v^{(2)},s,r}&\\&=c_Hk\brac{y^{(1)},y^{(2)},u^{(1)},u^{(2)},v^{(1)},v^{(2)},s,r} \mathds{1}_{[\zeta,s-\zeta]}(r).
	\end{align*}
	With these definitions at hand, we can state the following convergence
	lemma.
	\begin{lemma}
		\label{lemma_limit_C}
		Assume Conditions \ref{condH1} and either \ref{condH2-A}
		or \ref{condH2-B} hold. Then, one has that
		\begin{enumerate}
			\item[(i)]  $\displaystyle \abs{\int_{\mathcal{U}^2\times
					\mathcal{V}^2\times \mathcal{Y}^2\times
					[0,t]^2} k^\zeta dP^\epsilon\otimes dP^\epsilon- \int_{\mathcal{U}^2\times
					\mathcal{V}^2\times \mathcal{Y}^2\times
					[0,t]^2} k dP^\epsilon\otimes
				dP^\epsilon}\to 0$ a.s. as $\zeta
			\to 0$;
			\item[(ii)]  $\displaystyle     \abs{\int_{\mathcal{U}^2\times
					\mathcal{V}^2\times \mathcal{Y}^2\times
					[0,t]}^2 k^\zeta dP^\epsilon\otimes dP^\epsilon-\int_{\mathcal{U}^2\times
					\mathcal{V}^2\times \mathcal{Y}^2\times
					[0,t]}^2 k^\zeta d\bar{P}\otimes
				d\bar{P}}\to 0 $ a.s. as $\epsilon
			\to 0$;
			\item[(iii)]  $\displaystyle   \abs{\int_{\mathcal{U}^2\times
					\mathcal{V}^2\times \mathcal{Y}^2\times
					[0,t]^2} k^\zeta d\bar{P}\otimes d\bar{P}-\int_{\mathcal{U}^2\times
					\mathcal{V}^2\times \mathcal{Y}^2\times
					[0,t]^2} k d\bar{P}\otimes d\bar{P}}\to 0$ a.s. as $\zeta
			\to 0$;
		\end{enumerate}
	\end{lemma}
	
	\begin{proof}
		We first prove part $(i)$. Since $k^\zeta\to k$
		pointwise as $\zeta\to 0$, all we need to prove is that the function $k$ is integrable with respect to $P^\epsilon\otimes P^\epsilon$ on $\mathcal{U}^2\times
		\mathcal{V}^2\times \mathcal{Y}^2\times
		[0,1]^2$ as then, the Dominated Convergence
		Theorem applies and yields the desired limit. We have for some constant $C<\infty$
		\begin{align*}
			\\&\int_{\mathcal{U}^2\times
				\mathcal{V}^2\times \mathcal{Y}^2\times
				[0,1]^2}k dP^\epsilon\otimes dP^\epsilon \\&\leq c_H\int_0^1 \abs{f(\bar{X}_s,Y^\eu_s)} s^{H-1/2}\int_0^s(s-r)^{H-3/2}r^{1/2-H}\abs{\hat{u}^\epsilon_r}drds\\
			&\leq \int_0^1 \abs{f(\bar{X}_s,Y^\eu_s)} \left[\dot{K}_H \abs{\hat{u}^\epsilon}\right]_sds\\
			&\leq C \sqrt{\int_0^1 \abs{f(\bar{X}_s,Y^\eu_s)}^2 ds},
		\end{align*}
		where the second inequality follows by Lemma
		\ref{lemma_derivative_control} and the last inequality
		is a consequence of H\"{o}lder's inequality and
		Proposition \ref{prop_dotKHinL2}. Now, the boundedness
		of $f$ under Condition \ref{condH2-A} or the sublinear
		growth of $f$ under Condition \ref{condH2-B} together with Lemma \ref{lemma_estimatefromMorKos2} yield
		\begin{align*}
			\int_{\mathcal{U}^2\times
				\mathcal{V}^2\times \mathcal{Y}^2\times
				[0,1]^2}k dP^\epsilon\otimes dP^\epsilon
			\leq C \sqrt{\int_0^1 \abs{f(\bar{X}_s,Y^\eu_s)}^2 ds}
			< \infty.
		\end{align*}
		Part $(iii)$ is proven in the exact same way. Part
		$(ii)$ is a consequence of the weak convergence of
		$P^\epsilon$ to $\bar{P}$ and the uniform
		integrability of $\left\{ P^\epsilon\otimes
		P^\epsilon\colon \epsilon >0 \right\}$ (implied by the second
		point of Definition \ref{defviablepairs}). 
	\end{proof}
	The above results allow us to obtain an explicit representation of the
	limit points $(\bar{\eta},\bar{P})$, which is the object of the following proposition.
	\begin{proposition}
		\label{proposition_limit_eta}
		Let $(\bar{\eta},\bar{P})$ be a limit point of $\{(\eta^\eu,P^{\epsilon})\colon
		\epsilon>0\}$. Under Conditions \ref{condH1}
		and either \ref{condH2-A} or
		\ref{condH2-B}, it holds that
		\begin{align*}
			\bar{\eta}_t &=\int_{\mathcal{U}\times
				\mathcal{V}\times
				\mathcal{Y}\times [0,t]} \left[   \nabla_y\phi\brac{\bar{X}_s,y} \sigma\brac{y}v+\nabla_x \bar{g} (\bar{X}_s)\bar{\eta}_s \right] d\bar{P}\nonumber
			\\ &+c_H\int_{\mathcal{U}^2\times
				\mathcal{V}^2\times \mathcal{Y}^2\times
				[0,t]^2}f(\bar{X}_s,y^{(1)})(s-r)^{H-3/2}s^{H-1/2}r^{1/2-H}u^{(2)}\mathds{1}_{[0,s]}(r)d\bar{P}\otimes d\bar{P}.
		\end{align*}
	\end{proposition}
	\begin{proof}
		As pointed out earlier, we can consider the limiting
		behavior of each individual summands in the
		representation \eqref{recallrepetacontrol} of $\eta^\eu$. First, under Conditions \ref{condH1} and either \ref{condH2-A} or
		\ref{condH2-B}, Lemma \ref{lemma_limit_C} and
		\eqref{convergenceAversB} guarantee that
		\begin{align*}
			&\lim_{\epsilon\to 0}
			\int_0^tf\brac{X^\eu_s,Y^\eu_s} \dot{u}^\epsilon_s
			ds=c_H\int_{\mathcal{U}^2\times
				\mathcal{V}^2\times \mathcal{Y}^2\times
				[0,t]^2}
			f(\bar{X}_s,y)(s-r)^{H-3/2}s^{H-1/2}r^{1/2-H}u^{(2)}\\
			& \qquad\qquad\qquad\qquad\qquad\qquad\qquad\qquad\qquad\qquad\qquad\qquad\qquad\qquad\qquad \qquad\mathds{1}_{[0,s]}(r) d\bar{P}\otimes d\bar{P}.
		\end{align*}
		Next, we consider the Young integral terms. Under Conditions \ref{condH1} and \ref{condH2-A}, part $(i)$ of Lemma \ref{lemma_estimates_young} implies that, as $\epsilon\to 0$,
		\begin{align*}
			\E{\sup_{0\leq t\leq 1}\abs{\frac{1}{h(\epsilon)}\int_0^t f\brac{X^\eu_s,Y^\eu_s}dW^H_s }}
			\leq Ch(\epsilon)^{-1}{\epsilon^{-\frac{\beta}{2}}}\to 0,
		\end{align*}
		and
		\begin{align*}
			\E{ \abs{\sup_{t\in [0,1]} \frac{\epsilon}{h(\epsilon)}\int_0^t\nabla_x\phi\brac{X^\eu_s,Y^\eu_s}f\brac{X^\eu_s,Y^\eu_s}d{W}^H_s
			}}&\leq Ch(\epsilon)^{-1}\epsilon^{\frac{1}{2}}\to 0.
		\end{align*}
		Likewise, under Conditions \ref{condH1} and \ref{condH2-B}, part $(ii)$ of Lemma \ref{lemma_estimates_young} implies that as $\epsilon\to 0$,
		\begin{align*}
			\E{\sup_{t\in [0,1]}\abs{\frac{1}{h(\epsilon)}\int_0^t f\brac{Y^\eu_s}dW^H_s}}
			\leq Ch(\epsilon)^{-1}\epsilon^{-\frac{M_f}{2}}\to 0,
		\end{align*}
		and
		\begin{align*}
			\E{\sup_{t\in[0,1]}\abs{\frac{\epsilon}{h(\epsilon)}\int_0^t \nabla_x\phi(X^\eu_s,Y^\eu_s)f(Y^\eu_s)dW^H_s}}\leq Ch(\epsilon)^{-1}\epsilon^{1-\frac{M_k}{2}}\to 0.
		\end{align*}
		Consequently, under Conditions \ref{condH1} and either \ref{condH2-A} or
		\ref{condH2-B}, we get
		\begin{align*}
			\lim_{\epsilon \to 0}\frac{1}{h(\epsilon)}\int_0^t f\brac{Y^\eu_s}dW^H_s=0
		\end{align*}
		and
		\begin{align*}
			\lim_{\epsilon \to 0}\frac{\epsilon}{h(\epsilon)}\int_0^t \nabla_x\phi(X^\eu_s,Y^\eu_s)f(Y^\eu_s)dW^H_s=0.
		\end{align*}
		For the limits of the remaining terms in the representation
		\eqref{recallrepetacontrol} of $\eta^\eu$, we refer to
		\cite[Lemma 3.2]{DupuisSpiliopoulos} in which these terms have
		already been addressed.
	\end{proof}
	The following proposition asserts that the invariant measure of
	$Y^\epsilon$ and the Lebesgue measure are among the marginals of $\bar{P}$.
	\begin{proposition}
		\label{proposition_occu_decomp}
		Recall that $\mu$ denotes the unique invariant measure
		associated with the generator $\mathcal{L}$ defined in \eqref{def_generator_normalized_Y}. Under Condition \ref{condH1}, we have the decomposition
		\begin{align*}
			\bar{P}(dudvdydt)=\nu_{y,t}(dudv)\mu(dy)dt,
		\end{align*}
		where $\nu_{y,t}$ is a kernel on $\mathcal{U}\times
		\mathcal{V}$ dependent on $y\in \mathcal{Y}$ and $t\in
		[0,1]$.
	\end{proposition}
	\begin{proof}
		Let $F$ be an element of a dense subset of $\mathcal{C}^2(\mathcal{Y})$ that consists of bounded functions with bounded first and second derivatives. By It\^{o}'s formula, we have
		\begin{align}
			\label{equation_generator_of_Y}
			&\int_{\mathcal{U}\times \mathcal{V}\times
				\mathcal{Y}\times[0,t]}\mathcal{L}
			F(y)dP^\epsilon=\epsilon
			\brac{F(Y^\eu_t)-F(y_0)}\nonumber\\& \qquad\qquad -\sqrt{\epsilon}\int_0^t \brac{\nabla_yF}^\top(Y^\eu_s)\sigma(Y^\eu_s)dB_s-\sqrt{\epsilon}h(\epsilon) \int_0^t \brac{\nabla_yF}^\top(Y^\eu_s) \sigma(Y^\eu_s)\dot{v}^\epsilon_sds
		\end{align}
		Let us consider each individual term on the right-hand
		side of the above equation. The first term converges
		to $0$ given that $F$ is bounded. For the second term,
		an application of the Burkh\"{o}lder-Davis-Gundy inequality yields
		\begin{align*}
			\sqrt{\epsilon}\E{\abs{\int_0^t \brac{\nabla_yF}^\top(Y^\eu_s)\sigma(Y^\eu_s)dB_s}}\leq C\sqrt{\epsilon}\sqrt{\E{ \int_0^1 \abs{\sigma(Y^\eu_s)}^2 ds}},
		\end{align*}
		which converges to zero due to the boundedness of $\sigma(y)\sigma^\top(y)$ in Condition \ref{condH1}. Similarly, by Lemma \ref{lemma_boundedcontrol}, we have
		\begin{align*}
			&\abs{\sqrt{\epsilon}h(\epsilon) \int_0^t
				\brac{\nabla_yF}^\top(Y^\eu_s)
				\sigma(Y^\eu_s)\dot{v}^\epsilon_sds}\\&
			\qquad\qquad\qquad \leq \sqrt{\epsilon}h(\epsilon)
			\sqrt{\int_0^t \abs{\brac{\nabla_yF}^\top(Y^\eu_s)
					\sigma(Y^\eu_s)\sigma^\top(Y^\eu_s)\nabla_yF(Y^\eu_s)}ds\int_0^t\abs{\dot{v}^\epsilon_s}^2ds}\\
			& \qquad\qquad\qquad \leq C\sqrt{\epsilon}h(\epsilon).
		\end{align*}
		Hence, \eqref{equation_generator_of_Y} becomes
		\begin{align}
			\label{equation_characterize_invariant}
			\int_{\mathcal{U}\times \mathcal{V}\times \mathcal{Y}\times[0,t]}\mathcal{L} F(y)d\bar{P}=0.
		\end{align}
		Moreover, it is immediate to see that
		$P^\epsilon\brac{\mathcal{U}\times \mathcal{V}\times
			\mathcal{Y}\times[0,t]}=t$, which implies that the
		last marginal of $\bar{P}$ is the Lebesgue measure. In
		other words, $\bar{P}$ is of the form
		$\bar{P}(dudvdydt)=\nu_{t,y}(dudv)m(dy)dt$. Moreover,
		since $\mathcal{L}$ is independent of the control
		$(u,v)$, \eqref{equation_characterize_invariant}
		implies that
		\begin{align*}
			\int_{ \mathcal{Y}}\mathcal{L} F(y)m(dy)=0,
		\end{align*}
		which implies that $m(dy)$ is the unique invariant measure $\mu(dy)$ associated with $\mathcal{L}$.
	\end{proof}
	The next proposition asserts that the pair $(\bar{\eta},\bar{P})$ is
	indeed a viable pair with respect to $(\theta,\mathcal{L})$, which was
	what this subsection was aimed at proving.
	\begin{proposition}
		\label{prop_viable}
		The pair $\brac{\bar{\eta},\bar{P}}$ is a viable pair
		with respect to $\brac{\theta,\mathcal{L}}$, where $\theta$ is the
		function defined in \eqref{definitionoffunctiontheta}
		and $\mathcal{L}$ is the generator defined in
		\eqref{def_generator_normalized_Y}.
	\end{proposition}
	\begin{proof}
		Lemmas \ref{lemma_boundedcontrol} and \ref{lemma_estimatefromMorKos2}
		together with Fatou's lemma ensure that $\bar{P}$ satisfies the first
		property in Definition \ref{defviablepairs}. The following two properties in
		Definition \ref{defviablepairs} have been established in Proposition
		\ref{proposition_limit_eta} and Proposition
		\ref{proposition_occu_decomp}.
	\end{proof}
	\subsection{Proof of the Laplace principle lower bound}
	\label{prooflaplacelower}
	The Laplace principle lower bound can be immediately derived from
	Fatou's lemma and Proposition \ref{proposition_limit_eta},
	which is shown in the following proposition.
	\begin{proposition}
		\label{proposition_laplace_lower}
		Assume Conditions \ref{condH1} and
		either \ref{condH2-A} or \ref{condH2-B} are
		satisfied. Then, for all bounded and continuous
		mappings $a \colon C\brac{[0,1];\R^n}\to \R$, the following Laplace principle lower bound holds.
		\begin{align*}
			\liminf_{\epsilon\to
				0}-\frac{1}{h^2(\epsilon)}\ln
			\E{\exp\brac{-h^2(\epsilon)a(\eta^\epsilon)}}\geq
			\inf_{\Phi\in C\brac{[0,1];\R^n}} S^{H}(\Phi)+a(\Phi),
		\end{align*}
		where the rate function $S^{H}$ is defined at \eqref{rate_function}.
	\end{proposition}
	\begin{proof}
		We can write
		\begin{align*}
			\liminf_{\epsilon\to
				0}-\frac{1}{h^2(\epsilon)}\ln
			\E{\exp\brac{-h^2(\epsilon)a(\eta^\epsilon)}}&\geq
			\liminf_{\epsilon\to
				0}\E{\frac{1}{2}\int_0^1
				\left[\abs{\hat{u}^\epsilon_s}^2+\abs{\hat{v}^\epsilon_s}^2  \right] ds+a\brac{\eta^\eu}}-\delta\\
			&=  \liminf_{\epsilon\to 0} \E{ \frac{1}{2}\int_{\mathcal{U}\times \mathcal{V}\times \mathcal{Y}\times[0,1]}\left[\abs{u}^2+\abs{v}^2  \right] dP^\epsilon+a\brac{\eta^\eu}}\\
			&\geq \E{ \frac{1}{2}\int_{\mathcal{U}\times \mathcal{V}\times \mathcal{Y}\times[0,1]}\left[ \abs{u}^2+\abs{v}^2 \right] d\bar{P}+a\brac{\bar{\eta}}}\\
			&\geq \inf_{\Phi\in C\brac{[0,1];\R^n}} S^{H}(\Phi)+a(\Phi).
		\end{align*}
		The first inequality comes from the variational
		formula \eqref{equation_varrep}. The second line is a
		direct application of the definition of the occupation measure
		$P^\epsilon$. The third line follows from Fatou's lemma and the convergence result in
		Proposition \ref{proposition_limit_eta}. Finally, the
		last line is a consequence of Proposition
		\ref{prop_viable}.
	\end{proof}
	\subsection{Proof of the compactness of the level sets of $S^{H}(\cdot)$}
	\label{proofcompactness}
	We need to show that, for each $k \in \R$, the level sets of
	$S^{H}$ given by
	\begin{align*}
		L_k=\{\Phi\in C\brac{[0,1];\R^n}:S^{H}(\Phi)\leq k \},\quad k < \infty.
	\end{align*}
	are compact subsets of $C\brac{[0,1];\R^n}$, which indicates that
	$S^{H}$ is a good rate function. We will actually show that, for
	any $k\in \R$, $L_k$ is relatively compact and closed. We
	start with relative compactness, which the following lemma addresses.
	\begin{lemma}
		\label{lemma_precompactness_level_set}
		Let $\{\brac{\Phi_n,P_n}\colon n\in\N\}$ be a sequence
		such that for every $n\in\N$, $\brac{\Phi_n,P_n} \in
		\mathcal{V}\brac{\theta,\mathcal{L}}$ and $\Phi_n\in
		L_k$. Assuming Conditions \ref{condH1} and
		either \ref{condH2-A} or \ref{condH2-B} are
		satisfied, this sequence is relatively compact on $C\brac{[0,1];\R^n}$.
	\end{lemma}
	\begin{proof}
		We can show that the family $\{P_n\colon n\in\N\}$ is relatively
		compact in the same way as in Subsection \ref{tightnessofPepsilon}
		where we proved the tightness of $\left\{P^{\epsilon}\colon \epsilon >0
		\right\}$. To show the relative compactness of
		$\{\Phi_n\colon n\in\N\}$, it is sufficient to verify that
		\begin{align*}
			\lim_{\delta\to 0}\sup_{\Phi\in L_k}\omega_\Phi(\delta)=\lim_{\delta\to 0}\sup_{\Phi\in L_k}\sup_{\abs{t-r}\leq \delta}\abs{\Phi(t)-\Phi(r)}=0.
		\end{align*}
		By Proposition \ref{prop_control_problems_equi}, the fact that $\brac{\Phi_n,P_n}\in \mathcal{V}\brac{\theta,\mathcal{L}}$ implies there exists a pair of ordinary controls $\brac{u,v}\in L^2\brac{\mathcal{Y}^2\times [0,1]^2;\R^m\times \R^p}$ such that
		\begin{align*}
			\Phi_t=\int_{\mathcal{Y}\times[0,t]}
			\nabla_y\phi\brac{\bar{X}_s,y}
			\sigma\brac{y}v(s,y)\mu(dy)ds &+\int_0^t\nabla_x \bar{g} (\bar{X}_s)\Phi_sds\nonumber\\
			&+\int_{\mathcal{Y}\times[0,t]}f(\bar{X}_s,y)\brac{\dot{K}_Hu}(s,y)\mu(dy)ds.
		\end{align*}
		Then,
		\begin{align*}
			\Phi_t-\Phi_r&= \int_{\mathcal{Y}\times[r,t]}
			\nabla_y\phi\brac{\bar{X}_s,y}
			\sigma\brac{y}v(s,y)\mu(dy)ds +\int_0^t\nabla_x \bar{g} (\bar{X}_s)\Phi_sds\nonumber\\
			&\quad +\int_{\mathcal{Y}\times[r,t]}f(\bar{X}_s,y)\brac{\dot{K}_Hu}(s,y)\mu(dy)ds\\
			&=A_1+A_2+A_3.
		\end{align*}
		The term $A_1$ can be estimated by
		\begin{align*}
			\abs{A_1}&\leq \sqrt{\int_{\mathcal{Y}\times[r,t]}\abs{\nabla_y\phi\brac{\bar{X}_s,y} \sigma\brac{y}}^2\mu(dy)ds} \sqrt{\int_{\mathcal{Y}\times[0,1]}\abs{v(s,y)}^2\mu(dy)ds }\\
			&\leq C\sqrt{\int_{\mathcal{Y}\times[r,t]} \abs{y}^{2D_g}\mu(dy)ds }\\
			&\leq C\sqrt{\abs{t-r}}.
		\end{align*}
		Similarly, we have
		\begin{align*}
			\abs{A_2}\leq C\int_r^t\abs{\Phi_s }ds\leq C\abs{t-r},
		\end{align*}
		which is immediate as $\Phi\in C\brac{[0,1];\R^n}$ is bounded. For the final term $A_3$, we apply Proposition \ref{prop_dotKHinL2} to get
		\begin{align*}
			\abs{A_3}&\leq  \sqrt{\int_{\mathcal{Y}\times[r,t]} \abs{y}^{2D_f}\mu(dy)ds}\sqrt{\int_{\mathcal{Y}\times[0,1]} \abs{\brac{\dot{K}_Hu}(s,y)}^2 \mu(dy)ds}\\
			&\leq C \sqrt{\abs{t-r}}.
		\end{align*}
		Combining the previous estimates leads to
		\begin{align*}
			\abs{\Phi_t-\Phi_r}\leq C\brac{\abs{t-r}+\sqrt{\abs{t-r}}},
		\end{align*}
		which completes the proof.
	\end{proof}	
	The next step is to prove that the limit of a sequence of viable
	pairs is a viable pair. This is the object of the next lemma.
	\begin{lemma}
		\label{lemma_limitviablepair}
		Let $\{\brac{\Phi_n,P_n}\colon n\in\N\}$ be a sequence
		such that for every $n \in \N$, $\brac{\Phi_n,P_n} \in
		\mathcal{V}\brac{\theta,\mathcal{L}}$ and $\Phi_n\in
		L_k$. Furthermore, assume that the sequence
		$\{\brac{\Phi_n,P_n}\colon n\in\N\}$ converges to a limit $\brac{\Phi,P}$. Assuming Conditions \ref{condH1} and
		either \ref{condH2-A} or \ref{condH2-B} are
		satisfied, we have $\brac{\Phi,P} \in \mathcal{V}\brac{\theta,\mathcal{L}}$.
	\end{lemma}
	\begin{proof}
		Using Fatou's lemma, we can write
		\begin{align*}
			\int_{\mathcal{U}\times \mathcal{V}\times \mathcal{Y}\times[0,1]}\left[  \abs{u}^2+\abs{v}^2 \right] dP\leq \liminf_{n\to\infty}\int_{\mathcal{U}\times \mathcal{V}\times \mathcal{Y}\times[0,1]}\left[  \abs{u}^2+\abs{v}^2 \right] dP_n\leq k,
		\end{align*}
		so that the second criterion in Definition
		\ref{defviablepairs} is satisfied. The third and
		fourth criteria can be proved in a similar but simpler
		manner as in the proofs of Propositions \ref{proposition_limit_eta} and \ref{proposition_occu_decomp}.
	\end{proof}
	The final step is to prove that the map $S^{H}$ is lower semicontinuous,
	which is done in the lemma below.
	\begin{lemma}
		\label{lemma_lower_semicontinuity}
		Assume Conditions \ref{condH1} and
		either \ref{condH2-A} or \ref{condH2-B} hold. Then,
		$S^{H}(\Phi)$ is lower semicontinuous, which is equivalent to the statement that the level sets of $S^{H}$ are closed in $C\brac{[0,1];\R^n}$.
	\end{lemma}
	\begin{proof}
		Let $\Phi_n$ converge to $\Phi$ in $C\brac{[0,1];\R^n}$. We will show \begin{align*}
			\liminf_{n\to\infty}S^{H}(\Phi_n)\geq S^{H}(\Phi).
		\end{align*}
		When $S^{H}(\Phi_n)<\infty$, there exists $P_n$ such that $\brac{\Phi_n,P_n}\in\mathcal{V}\brac{\theta,\mathcal{L}}$ and
		\begin{align*}
			S^{H}(\Phi_n)\geq \frac{1}{2}\int_{\mathcal{U}\times \mathcal{V}\times \mathcal{Y}\times[0,1]}\left[   \abs{u}^2+\abs{v}^2\right] dP_n-\frac{1}{n}.
		\end{align*}
		By Lemma \ref{lemma_precompactness_level_set}, we can consider a subsequence along which $\brac{\Phi_n,P_n}$ converges to $\brac{\Phi,P}$. Moreover, Lemma \ref{lemma_limitviablepair} guarantees $\brac{\Phi,P}\in \mathcal{V}\brac{\theta,\mathcal{L}}$. Consequently,
		\begin{align*}
			\liminf_{n\to\infty}S^{H}(\Phi_n)=  &\liminf_{n\to\infty}\frac{1}{2}\int_{\mathcal{U}\times \mathcal{V}\times \mathcal{Y}\times[0,1]}\left[ \abs{u}^2+\abs{v}^2 \right]  dP_n-\frac{1}{n}\\
			&\geq\frac{1}{2}\int_{\mathcal{U}\times \mathcal{V}\times \mathcal{Y}\times[0,1]}\left[  \abs{u}^2+\abs{v}^2 \right] dP\\
			&\geq\frac{1}{2}\inf_{(\Phi,P)\in\mathcal{V}}\int_{\mathcal{U}\times \mathcal{V}\times \mathcal{Y}\times[0,1]}\left[\abs{u}^2+\abs{v}^2  \right]  dP=S^{H}(\Phi),
		\end{align*}
		which concludes the proof.
	\end{proof}
	We can now combine the preceding results in order to state the
	following proposition, which was the object of this subsection.
	\begin{proposition}
		\label{proposition_compactness_level_set}
		Assume Conditions \ref{condH1} and
		either \ref{condH2-A} or \ref{condH2-B} are
		satisfied. Then, for every $k\in \R$, $L_k$ is a compact subset of $C\brac{[0,1];\R^n}$.
	\end{proposition}


	\subsection{Proof of the Laplace principle upper bound and
		representation formula}
	\label{prooflaplaceupper}
	We begin by introducing an alternate representation of $S^{H}(\Phi)$. By
	the definition of viable pairs (see Definition \ref{defviablepairs})
	and that of $S^{H}(\Phi)$ (see \eqref{rate_function}), we can write, for
	any $\Phi\in C\brac{[0,1];\R^n}$,
	\begin{align*}
		S^{H}(\Phi)&=
		\inf_{(\Phi,P)\in\mathcal{V}(\theta,\mathcal{L})}\left[ \frac{1}{2}
		\int_{\mathcal{U}\times \mathcal{V}\times
			\mathcal{Y}\times[0,1]}\left[ \abs{u}^2+\abs{v}^2 \right]
		P(dudvdyds)\right]\\
		&= L^{r}\brac{\Phi,\bar{X}},
	\end{align*}
	where
	\begin{align*}
		L^{r}\brac{\Phi,\bar{X}}= \inf_{P\in \mathcal{A}_{\Phi}^r} \frac{1}{2}\int_{\mathcal{U}\times \mathcal{V}\times \mathcal{Y}\times [0,1]} \left[ \abs{u}^2+\abs{v}^2 \right] P(dudvdy dt).
	\end{align*}
	The set $\mathcal{A}_{\Phi}^r$ consists of elements $P\in
	\mathcal{P}(\mathcal{U}\times \mathcal{V}\times
	\mathcal{Y}\times [0,1])$ for which the decomposition
	\eqref{decompdeP} holds and  such that
	\begin{align*}
		\int_{\mathcal{U}\times \mathcal{V}\times \mathcal{Y}\times[0,1]}\left[\abs{u}^2+\abs{v}^2+\abs{y}^2  \right]  P(dudvdyds)<\infty
	\end{align*}
	and (recalling the definition of the function $\theta$ given
	in \eqref{definitionoffunctiontheta})
	\begin{align*}
		&\int_{\mathcal{U}^2\times \mathcal{V}^2\times
			\mathcal{Y}^2\times[0,t]^2}\theta\brac{\bar{X}_s,\Phi_{s},y^{(1)},y^{(2)},u^{(1)},u^{(2)},v^{(1)},v^{(2)},s,r}dP\otimes
		dP = \Phi_t.
	\end{align*}
	Now, for any $\Phi\in C\brac{[0,1];\R^n}$, let us define
	\begin{align}
		\label{ordinarycontrolpb}
		L^{o}\brac{\Phi,\bar{X}}= \inf_{w\in \mathcal{A}^o_{\Phi}} \frac{1}{2}\int_{\mathcal{Y}\times [0,1]} \abs{w(t,y)}^2 \mu(dy) dt,
	\end{align}		
	where the set $\mathcal{A}^o_{\Phi}$ consists of elements
	$w=\brac{u,v}\colon \mathcal{Y}^2\times [0,1]^2\to \R^{m+p}$ of
	$\mathcal{S}$ such that, for any $t\in[0,1]$,
	\begin{align}
		\label{equation_ordinarycontrol}
		\int_{\mathcal{Y}\times[0,t]}\left[\nabla_y\phi\brac{\bar{X}_s,y}
		\sigma\brac{y}v(s,y)+\nabla_x \bar{g}
		(\bar{X}_s)\Phi_s +\bar{f}(\bar{X}_s)\brac{\dot{K}_Hu}(s,y) \right]
		\mu(dy)ds = \Phi_t
	\end{align}
	and
	\begin{align*}\int_{ \mathcal{Y}\times
			[0,1]}\left[\abs{u(t,y)}^2+\abs{v(t,y)}^2  \right] \mu(dy)dt<\infty.
	\end{align*}
	
	Our claim is that one actually has that
	$L^{r}\brac{\Phi,\bar{X}}=L^{o}\brac{\Phi,\bar{X}}$,
	which will provide us with the representation of $S^{H}(\Phi)$ we need to
	derive the upper bound of the Laplace principle. The equivalence
	between these two control systems is the object of the following proposition.
	\begin{proposition}
		\label{prop_control_problems_equi}
		Under Conditions \ref{condH1} and
		either \ref{condH2-A} or \ref{condH2-B}, it holds that $L^r\brac{\Phi,\bar{X}}=L^o\brac{\Phi,\bar{X}}$.
	\end{proposition}
	\begin{proof}
		Let us first show
		$L^r\brac{\Phi,\bar{X}}\geq
		L^o\brac{\Phi,\bar{X}}$. Choose any $P\in
		\mathcal{A}_{\Phi}^r$. Then, by definition of $\mathcal{A}_{\Phi}^r$, the
		decomposition
		$P(dudvdydt)=\nu_{t,y}(dudv)\mu(dy)dt$ holds. This
		allows us to define an element $w=\brac{u,v}$ with
		\begin{align}
			\label{eq:ordinarycontrol}
			u(t,y)=\int_{\mathcal{U}\times\mathcal{V}}
			v\nu_{t,y}(dudv)\quad \mbox{and}\quad v(t,y)=\int_{\mathcal{U}\times\mathcal{V}} v\nu_{t,y}(dudv).
		\end{align}
		We claim that $w\in \mathcal{A}^o_{\Phi}$. Jensen's inequality and the decomposition of $P$ imply
		\begin{align}
			\label{estimate_costuv_less_costP}
			&\int_{ \mathcal{Y}\times
				[0,1]}\left[\abs{u(t,y)}^2+\abs{v(t,y)}^2  \right] \mu(dy)dt=\int_{\mathcal{Y}\times
				[0,1]}\left[ \brac{\int_{\mathcal{U}\times\mathcal{V}}
				z_1\nu_{t,y}(dudv)}^2\right.\nonumber \\
			&\left. \qquad\qquad\qquad\qquad\qquad\qquad\qquad\quad +\brac{\int_{\mathcal{U}\times\mathcal{V}} z_2\nu_{t,y}(dudv)}^2+\brac{\int_{\mathcal{U}\times\mathcal{V}} v\nu_{t,y}(dudv)}^2\right]\mu(dy)dt\nonumber\\
			&\qquad\qquad\qquad\qquad\qquad\qquad\qquad\leq
			\int_{\mathcal{U}\times \mathcal{V}\times
				\mathcal{Y}\times [0,1]}\left[ \abs{u}^2+ \abs{v}^2 \right] \nu_{t,y}(dudv)\mu(dy)dt\nonumber\\&\qquad\qquad\qquad\qquad\qquad\qquad\qquad= \int_{\mathcal{U}\times \mathcal{V}\times \mathcal{Y}\times [0,1]}\left[ \abs{u}^2+ \abs{v}^2 \right] P(dudvdydt)<\infty.
		\end{align}
		Hence, the last property in the definition of
		$\mathcal{A}^o_{\Phi}$ is satisfied and based on
		\eqref{eq:ordinarycontrol}, so is the first one. This
		shows that $w^{(1)}(t,y)\in
		\mathcal{A}^o_{\Phi}$. Furthermore,
		\eqref{estimate_costuv_less_costP} yields
		\begin{align*}
			L^r\brac{\Phi,\bar{X}}&=\inf_{P\in \mathcal{A}_{\Phi}^r} \frac{1}{2}\int_{\mathcal{U}\times \mathcal{V}\times \mathcal{Y}\times [0,1]}\left[ \abs{u}^2+\abs{v}^2 \right]  P(dudvdydt)\\
			&\geq \inf_{P\in \mathcal{A}_{\Phi}^r}\frac{1}{2}\int_{\mathcal{Y}\times [0,1]}\left[\abs{u(t,y)}^2+\abs{v(t,y)}^2  \right]  \mu(dy) dt\\
			&\geq \inf_{w\in \mathcal{A}^o_{\Phi}} \frac{1}{2}\int_{\mathcal{Y}\times [0,1]} \abs{w(t,y)}^2 \mu(dy) dt=L^o\brac{\Phi,\bar{X}}.
		\end{align*}
		It remains to prove that $L^r\brac{\Phi,\bar{X}}\leq
		L^o\brac{\Phi,\bar{X}}$. To this end, choose any
		$w=(u,v)\in \mathcal{A}^o_{\Phi}$ and construct a measure $P\in
		\mathcal{A}_{\Phi}^r$ according to
		\begin{align*}
			P(dudvdydt)=\delta_{u(t,y)}(du)\delta_{v(t,y) }(dv)\mu(dy)dt.
		\end{align*}
		Checking that $P$ satisfies all the needed
		properties to belong to $\mathcal{A}_{\Phi}^r$ is
		similar to what was done above. We hence have a set
		$\mathcal{B}^r\subseteq \mathcal{A}_{\Phi}^r$ that
		corresponds to the set  $\mathcal{A}^o_{\Phi}$, from
		which we deduce that
		\begin{align*}
			L^o\brac{\Phi,\bar{X}}&=\inf_{(u,v)\in \mathcal{A}^o_{\Phi}} \frac{1}{2}\int_{\mathcal{Y}\times [0,1]}\left[ \abs{u(t,y)}^2+\abs{v(t,y)}^2 \right]  \mu(dy)dt\\
			&\geq \inf_{P\in \mathcal{A}_{\Phi}^r} \frac{1}{2}\int_{\mathcal{U}\times \mathcal{V}\times \mathcal{Y}\times [0,1]}\left[ \abs{u}^2+\abs{v}^2 \right]  P(dudvdy dt)=L^r\brac{\Phi,\bar{X}},
		\end{align*}
		which concludes the proof.
	\end{proof}
	The next step is to derive an explicit expression of
	$L^o\brac{\Phi,\bar{X}}$. The statement of this expression
	requires us to introduce some linear maps. For a given $x\in
	C\brac{[0,1];\R^n}$, let $\pi_x:L^2\brac{\mathcal{Y}\times [0,1];\R^m}\to L^2\brac{\mathcal{Y}\times [0,1];\R^n}$ and $\rho_x:L^2\brac{\mathcal{Y}\times [0,1];\R^p}\to L^2\brac{\mathcal{Y}\times [0,1];\R^n}$ be two operators defined by
	\begin{align*}
		\pi_x u(t)= \int_{\mathcal{Y}}\bar{f}(x)\dot{K}_Hu(t,y)\mu(dy),\qquad
		\rho_x v(t)=\int_{\mathcal{Y}} \nabla_y\phi\brac{x,y} \sigma\brac{y}v(t,y)\mu(dy).
	\end{align*}
	Under either Condition \ref{condH2-A} or \ref{condH2-B}, $\bar{f}(x)$ is bounded. This fact and Proposition \ref{prop_dotKHinL2} yield
	\begin{align*}
		\norm{\pi_x u}_{L^2\brac{[0,1];\R^n}}\leq C\sqrt{\int_{\mathcal{Y}\times [0,1]}\abs{\dot{K}_Hu(t,y)}^2\mu(dy)dt}\leq C\norm{u}_{L^2\brac{\mathcal{Y}\times [0,1];\R^m}},
	\end{align*}
	so that $\pi_x$ is bounded. The operator $\rho_x$ is also bounded via Lemma
	\ref{lemma_estimatefromMorKos2} and the estimates
	\eqref{estimate_Poissonsolution_h1}, \eqref{estimate_Poissonsolution_h2}. Therefore, the Hilbert
	adjoints $\pi_x^*$ and $\rho_x^*$ are well-defined and given
	by $\pi_x^* h=\dot{K}_H^*\brac{\bar{f}}^\top(x) h$ and
	$\rho_x^* h=\sigma\brac{\cdot}^\top\brac{\nabla_y\phi}^\top\brac{x,\cdot}
	h$, respectively. It follows from these facts that
	\begin{align*}
		\Sigma_x(u,v)=\pi_x u+\rho_x v
	\end{align*}
	is also a bounded operator. Thus, its
	Hilbert adjoint $\Sigma_x^*$ exists and is given by
	$\Sigma_x^*h=\brac{\pi_x^*h,\rho_x^*h}$. With these
	definitions at hand, let us finally define the operator $Q^{H}_x$ from
	$L^2\brac{[0,1];\R^n}$ to itself by
	\begin{align}
		\label{def_Q}
		Q^{H}_x&=\Sigma_x\Sigma_x^*=\bar{f}(x)\dot{K}_H\brac{\bar{f}(x)\dot{K}_H}^*+\int_{\mathcal{Y}}^{}\brac{\nabla_y\phi\brac{x,y}
			\sigma\brac{y}}\brac{\nabla_y\phi\brac{x,y} \sigma\brac{y}}^\top \mu(dy).
	\end{align}
	such that for $h\in L^2\brac{[0,1];\R^n}$,
	\begin{align*}
		&\left[\bar{f}\brac{\bar{X}}\dot{K}_H\brac{\bar{f}\brac{\bar{X}}\dot{K}_H}^* h\right](t)=c_H^2\bar{f}\brac{\bar{X_t}}t^{H-1/2}\\&\qquad\qquad\qquad\qquad\int_0^t(t-z)^{H-3/2}z^{1-2H}\int_z^1 (s-z)^{H-3/2}s^{H-1/2}\bar{f}\brac{\bar{X_s}}^\top h(s)dsdz.
	\end{align*}
	We are now ready to present the explicit expression of
	$L^o\brac{\Phi,\bar{X}}$, which is the object of the next proposition.
	\begin{proposition}
		\label{proposition_optimalcontrol}
		Assume Conditions \ref{condH1} and either
		\ref{condH2-A} or \ref{condH2-B}, and further that the operator $Q^{H}_{\bar{X}}$ is invertible. Then the ordinary control
		problem \eqref{ordinarycontrolpb} has a finite minimum
		cost if and only if $\Phi$ is absolutely continuous and $\dot{\Phi}$ (defined a.e.) is square integrable. In this case, the solution is given by
		\begin{align*}
			L^o\brac{\Phi,\bar{X}}= \int_0^1\brac{\dot{\Phi}_s-\nabla_x \bar{g} (\bar{X}_s)\Phi_s}^\top(Q^{H}_{\bar{X}_{s}})^{-1}\brac{\dot{\Phi}_s-\nabla_x \bar{g} (\bar{X_s})\Phi_s}ds
		\end{align*}
		and is achieved for the optimal control
		\begin{align*}
			\brac{\bar{u},\bar{v}}=\brac{\pi^* (Q^{H}_{\bar{X}})^{-1}\brac{\dot{\Phi}-\nabla_x \bar{g} \brac{\bar{X}}\Phi},\rho^*(Q^{H}_{\bar{X}})^{-1}\brac{\dot{\Phi}-\nabla_x \bar{g} \brac{\bar{X}}\Phi}}.
		\end{align*}
	\end{proposition}
	\begin{proof}
		
		In one direction, let us assume \eqref{ordinarycontrolpb} has a finite minimum cost then Equation \eqref{equation_ordinarycontrol} is satisfied for some $(u,v)\in A^o_\Phi$. Hence, $\Phi$ is absolutely continuous. Furthermore, by Cauchy-Schwarz inequality,
		\begin{align}
			\label{estimate_norm_dotphi}
			\norm{\dot{\Phi}}_{L^2([0,1];\R^n)}&= \int_0^1 	\abs{\int_{\mathcal{Y}}\left[\nabla_y\phi\brac{\bar{X}_s,y}
				\sigma\brac{y}v(s,y)+\nabla_x \bar{g}
				(\bar{X}_s)\Phi_s +\bar{f}(\bar{X}_s)\brac{\dot{K}_Hu}(s,y) \right]
				\mu(dy)}^2ds\nonumber\\
			&\leq \int_{\mathcal{Y}\times [0,1]} \abs{\nabla_y\phi\brac{\bar{X}_s,y}
				\sigma\brac{y}v(s,y)}^2+\abs{\nabla_x \bar{g}
				(\bar{X}_s)\Phi_s}^2+\abs{\bar{f}(\bar{X}_s)\brac{\dot{K}_Hu}(s,y)}^2\mu(dy)ds
		\end{align}
		To see that $\dot{\Phi}$ is square integrable, we study the right hand side of this inequality. $\nabla_y\phi\brac{\bar{X}_s,y}
		\sigma\brac{y}v(s,y)\in L^2(\mathcal{Y}\times [0,1];\R^n)$ due to Lemma \ref{lemma_estimatefromMorKos2}, Remark \ref{remark_poisson_equation} and boundedness of $\sigma(y)\sigma(y)^T$ in Condition \ref{condH1}. Next, we have $u\in L^2(\mathcal{Y}\times [0,1];\R^m)$, $v\in L^2(\mathcal{Y}\times [0,1];\R^p)$ and it follows from Proposition \ref{prop_dotKHinL2} that $\dot{K}_Hu\in L^2(\mathcal{Y}\times [0,1];\R^m)$. This along with the fact that $\bar{f}(x),\nabla_x\bar{g}(x)$ is bounded under either Condition \ref{condH2-A} or \ref{condH2-B}, and the fact that $\Phi$ is bounded since its derivative exists a.e. on $[0,1]$, imply the remaining quantities on the right side of \eqref{estimate_norm_dotphi} are in $L^2(\mathcal{Y}\times [0,1];\R^n)$. Therefore, we can conclude $\dot{\Phi}$ is square integrable.
		
		In the other direction, let us assume that $\Phi$ is absolutely continuous and that the  a.e. defined derivative is square integrable. Then $\dot{\Phi}-\nabla_x \bar{g} \brac{\bar{X}}\Phi$
		is also square integrable. Then, we can construct a control
		$(\bar{u},\bar{v})$ as in the statement of the proposition so
		that the set $A^o_\Phi$ associated with the ordinary control
		problem \eqref{ordinarycontrolpb} is non-empty and therefore, the minimum cost $L^o$ is finite. This settles the first claim of this proposition.
		

		Next, let us derive an explicit formula for the
		minimum cost. When $\Phi$ is absolutely continuous and its a.e. defined derivative is square integrable, we have
		\begin{align*}
			\norm{(\bar{u},\bar{v})}^2_{L^2\brac{\mathcal{Y}\times
					[0,1];\R^m\times \R^p}}&= \norm{\bar{u}}_{L^2\brac{\mathcal{Y}\times
					[0,1];\R^m}}^2+\norm{\bar{v}}_{L^2\brac{\mathcal{Y}\times
					[0,1];\R^p}}^2
			\\&=
			\inner{(Q^{H}_{\bar{X}})^{-1}\brac{\dot{\Phi}-\nabla_x
					\bar{g} \brac{\bar{X}}\Phi},\pi\pi^*
				(Q^{H}_{\bar{X}})^{-1}\brac{\dot{\Phi}-\nabla_x
					\bar{g} \brac{\bar{X}}\Phi}}_{L^2\brac{[0,1];\R^n}}
			\\&\quad +\inner{(Q^{H}_{\bar{X}})^{-1}\brac{\dot{\Phi}-\nabla_x \bar{g} (X)\Phi},\rho\rho^* (Q^{H}_{\bar{X}})^{-1}\brac{\dot{\Phi}-\nabla_x \bar{g} (X)\Phi}}_{L^2\brac{[0,1];\R^n}}
			\\&=\inner{(Q^{H}_{\bar{X}})^{-1}\brac{\dot{\Phi}-\nabla_x \bar{g} ({\bar{X}})\Phi},Q^{H}_{\bar{X}} (Q^{H}_{\bar{X}})^{-1}\brac{\dot{\Phi}-\nabla_x \bar{g} ({\bar{X}})\Phi}}_{L^2\brac{[0,1];\R^n}}\\
			&=\inner{\dot{\Phi}-\nabla_x \bar{g} ({\bar{X}})\Phi,(Q^{H}_{\bar{X}})^{-1}\brac{\dot{\Phi}-\nabla_x \bar{g} ({\bar{X}})\Phi}}_{L^2\brac{[0,1];\R^n}}.
		\end{align*}
		Since we also know that $(\bar{u},\bar{v})\in
		A^o_\Phi$, this implies
		\begin{align*}
			L^o\brac{\Phi,\bar{X},\Phi}\leq \inner{\dot{\Phi}-\nabla_x \bar{g} \brac{\bar{X}}\Phi,(Q^{H}_{\bar{X}})^{-1}\brac{\dot{\Phi}-\nabla_x \bar{g} \brac{\bar{X}}\Phi}}_{L^2\brac{[0,1];\R^n}}.
		\end{align*}
		Furthermore, by Lemma \ref{lemma_KostSalins} and the
		fact that ${\dot{\Phi}-\nabla_x \bar{g}
			(X)\eta}=\Sigma_X(u,v)$, we can write
		\begin{align*}
			L^o\brac{\Phi,\bar{X}}&\geq\inf_{(u,v)\in \mathcal{A}^o_\Phi}\norm{\Sigma_{\bar{X}}^*(Q^{H}_{\bar{X}})^{-1}\Sigma\brac{u,v}}^2_{L^2\brac{\mathcal{Y}\times[0,1];\R^m\times \R^p}}\\
			&=\norm{\Sigma_{\bar{X}}^*(Q^{H}_{\bar{X}})^{-1}\brac{\dot{\Phi}-\nabla_x \bar{g} \brac{\bar{X}}\Phi}}^2_{L^2\brac{\mathcal{Y}\times[0,1];\R^m\times \R^p}}\\
			&=\inner{(Q^{H}_{\bar{X}})^{-1}\brac{\dot{\Phi}-\nabla_x \bar{g} \brac{\bar{X}}\Phi},\Sigma_{\bar{X}}\Sigma_{\bar{X}}^*(Q^{H}_{\bar{X}})^{-1}\brac{\dot{\Phi}-\nabla_x \bar{g} \brac{\bar{X}}\Phi}}_{L^2\brac{[0,1];\R^n}}\\&=\inner{\dot{\Phi}-\nabla_x \bar{g} \brac{\bar{X}}\Phi,(Q^{H}_{\bar{X}})^{-1}\brac{\dot{\Phi}-\nabla_x \bar{g} \brac{\bar{X}}\Phi}}_{L^2\brac{[0,1];\R^n}}.
		\end{align*}
		Thus, the minimum cost of the ordinary control problem
		\eqref{ordinarycontrolpb} is the quantity in the last line.
	\end{proof}
	We are now ready to prove the Laplace principle upper bound, which is
	the object of the next proposition.
	\begin{proposition}
		\label{proposition_laplace_upper}
		Assume Conditions \ref{condH1} and either
		\ref{condH2-A} or \ref{condH2-B}, and further that the operator $Q^{H}_{\bar{X}}$ is invertible. Then the following Laplace principle upper bound holds.
		\begin{align*}
			\limsup_{\epsilon\to 0}-\frac{1}{h^2(\epsilon)}\ln \E{\exp\brac{-h^2(\epsilon)a(\eta^\epsilon)}}\leq \inf_{\Phi\in C\brac{[0,1];\R^n}} S^{H}(\Phi)+a(\Phi), \end{align*}
		where the function $S$ is the one defined at \eqref{rate_function}.
	\end{proposition}
	\begin{proof}
		We can assume, without loss of generality, that $\inf_{\Phi\in C\brac{[0,1];\R^n}}S^{H}(\Phi)<\infty$, so that for any $\zeta>0$, there exists an element $\Phi_0\in C\brac{[0,1];\R^n}$ for which
		\begin{align*}
			S^{H}(\Phi_0)+h(\Phi_0)\leq   \inf_{\Phi\in C\brac{[0,1];\R^n}} \brac{S^{H}(\Phi)+h(\Phi)}+\zeta.
		\end{align*}
		Let us also define
		\begin{align*}
			\widetilde{w}(\phi,x,y,\eta)&=\brac{\widetilde{u} (\phi,x,y,\eta),\widetilde{v} (\phi,x,y,\eta)}=\brac{\pi_x^* Q_x^{-1}\brac{\dot{\phi}-\nabla_x \bar{g} (x)\eta },\rho_x^*Q_x^{-1}\brac{\dot{\phi}-\nabla_x \bar{g} (x)\eta }}
		\end{align*}
		and
		\begin{align*}
			w_0=\brac{\widetilde{u} \brac{\Phi_0,X^\eu,Y^\eu,\eta^\eu},\widetilde{v}\brac{\Phi_0,X^\eu,Y^\eu,\eta^\eu}}.
		\end{align*}
		We can then substitute $w_0$ into the control variable
		of equation \eqref{equation_eta_new} and take the
		limit of $\eta^{\epsilon,w_0}$ as $\epsilon\to
		0$. This procedure is the same as the one that was carried out in Proposition \ref{proposition_limit_eta} and after which we obtained
		\begin{align}
			\label{eta_is_phi}
			\bar{\eta}_t&=\int_0^t \left[
			\rho\brac{\rho^*(Q^{H}_{\bar{X}})^{-1}\brac{\dot{\Phi_0}(s)-\nabla_x
					\bar{g} (\bar{X}_s)\bar{\eta}_s
			}}+\pi\brac{\pi^*
				(Q^{H}_{\bar{X}})^{-1}\brac{\dot{\Phi_0}(s)-\nabla_x
					\bar{g} (\bar{X}_s)\bar{\eta}_s
			}} \right]ds\nonumber \\
			&\quad +\int_0^t\nabla_x \bar{g} (\bar{X}_s)\bar{\eta}_s ds\nonumber \\
			&=\int_0^t
			Q^{H}_{\bar{X}}(Q^{H}_{\bar{X}})^{-1}\brac{\dot{\Phi_0}(s)-\nabla_x
				\bar{g} (\bar{X}_s)\bar{\eta}_s }ds+\int_0^t \nabla_x \bar{g} (\bar{X}_s)\bar{\eta}_s ds\nonumber
			\\
			&=\Phi_0(t).
		\end{align}
		In addition, we have
		\begin{align}
			\label{prelimit_cost}
			&\lim_{\epsilon\to 0} \E{\frac{1}{2}\int_0^1 \abs{\widetilde{u}(\Phi_0(s),X^\eu_s,Y^\eu_s,\eta^\eu_s)}^2+\abs{\widetilde{v}(\Phi_0(s),X^\eu_s,Y^\eu_s,\eta^\eu_s)}^2 ds}\nonumber \\
			&\qquad\qquad\qquad\qquad\qquad = \E{\frac{1}{2}\int_{\mathcal{Y}\times [0,1]} \abs{\widetilde{u}(\Phi_0(s),\bar{X}_s,y,\bar{\eta}_s)}^2+\abs{\widetilde{v}(\Phi_0(s),\bar{X}_s,y,\bar{\eta}_s)}^2\mu(dy) ds}\nonumber \\
			&\qquad\qquad\qquad\qquad\qquad =S^{H}(\Phi_0),
		\end{align}
		where the last equality is a consequence of Propositions \ref{prop_control_problems_equi}, \ref{proposition_optimalcontrol} and (\ref{eta_is_phi}). Therefore,
		\begin{align*}
			&\limsup_{\epsilon\to 0}-\frac{1}{h^2(\epsilon)}\ln \E{\exp\brac{-h^2(\epsilon)a(\eta^\epsilon)}}=\limsup_{\epsilon\to 0}\inf_{w^\epsilon\in \mathcal{S}}\E{\frac{1}{2}\int_0^1 \abs{\hat{u}^\epsilon_s}^2+\abs{\hat{v}^\epsilon_s}^2 ds+a\brac{\eta^\eu}}\\
			&\qquad\qquad\qquad \leq \limsup_{\epsilon\to
				0}\mathbb{E}\bigg[\frac{1}{2}\int_0^1
			\abs{\widetilde{u}\brac{\Phi_0(s),X^\eu_s,Y^\eu_s,\eta^\eu_s}}^2\\&\quad \qquad\qquad\qquad
			+\abs{\widetilde{v}\brac{\Phi_0(s),X^\eu_s,Y^\eu_s,\eta^\eu_s}}^2 ds+a\brac{\eta^{\epsilon,w_0}_s}\bigg]\\
			&\qquad\qquad\qquad = \E{\frac{1}{2}\int_{\mathcal{Y}\times [0,1]} \abs{\widetilde{u}(\Phi_0(s),\bar{X}_s,y,\bar{\eta}_s)}^2+\abs{\widetilde{v}(\Phi_0(s),\bar{X}_s,y,\bar{\eta}_s)}^2\mu(dy) ds+a\brac{\bar{\eta}_s}}\\
			&\qquad\qquad\qquad =S^{H}(\Phi_0)+a(\Phi_0)
			\leq \inf_{\Phi\in C\brac{[0,1];\R^n}} \brac{S^{H}(\Phi)+h(\Phi)}+\zeta,
		\end{align*}
		where the first equality is due to the variational
		formula \eqref{equation_varrep}, the second line is
		due to the choice of a particular control, and the
		last two equalities are consequences of
		\eqref{eta_is_phi} and \eqref{prelimit_cost}. Finally,
		the fact that $\zeta$ can be chosen arbitrarily yields
		the desired Laplace principle upper bound.
	\end{proof}
	
	\section{Proof of Corollary \ref{corollary_g(x)}}
	\label{section_proof_cor_g(x)}
	First, observe that given any $\Psi\in L^2\brac{[0,1];\R^n}$, the quantities
	\begin{align*}
		D_1=t^{H-1/2}\bar{f}\brac{\bar{X}}^\top\Psi,\qquad D_2=t^{1-2H}I^{H-1/2}_{1^-}t^{H-1/2}\bar{f}\brac{\bar{X}}^\top\Psi
	\end{align*}
	are in $L^2\brac{[0,1];\R^n}$. $D_1$ is square-integrable because $H>1/2$ and $\bar{f}$ is bounded under either Condition \ref{condH2-A} or \ref{condH2-B}. Regarding $D_2$, notice that the assumptions of Lemma \ref{lemma_Samko_bounded} are satisfied with $p=2,\alpha=H-1/2,\beta=1-2H$ and $\gamma=H-1/2$ for values of $H$ in the range $(1/2,3/4)$. Then the operator $t^{1-2H}I^{H-1/2}_{1^-}t^{H-1/2}$ is bounded on $L^2\brac{[0,1];\R^n}$, which implies $D_2$ is square-integrable.

	Next, $g=g(x)$ implies $\nabla_y\phi(x,y)=0$, where $\phi(x,y)$ is defined in \eqref{equation_Poisson_g}. Thus,
	\begin{align*}
		Q^{H}_{\bar{X}}=\bar{f}\brac{\bar{X}}\dot{K}_H \dot{K}_H^*\bar{f}\brac{\bar{X}}^\top.
	\end{align*}
	We also know
	\begin{align*}
		\dot{K}_H=c_H\Gamma(H-1/2)t^{H-1/2}I^{H-1/2}_{0^+}t^{1/2-H},\qquad \dot{K}_H^*=c_H\Gamma(H-1/2) t^{1/2-H}I^{H-1/2}_{1^-}t^{H-1/2}
	\end{align*}
	so that
	\begin{align*}
		Q^{H}_{\bar{X}}=c_H^2\Gamma(H-1/2)^2 \bar{f}\brac{\bar{X}}t^{H-1/2} I^{H-1/2}_{0^+}t^{1-2H}I^{H-1/2}_{1^-}t^{H-1/2}\bar{f}\brac{\bar{X}}^\top.
	\end{align*}
	Recalling that $Lf\brac{\bar{X}}=f\brac{\bar{X}}L=I$, at this point we want to show
	\begin{align}
		\label{def_inverse_Q}
		W=c_H^{-2}\Gamma(H-1/2)^{-2}L^\top t^{1/2-H}D^{H-1/2}_{1^-}t^{2H-1}D^{H-1/2}_{0^+}t^{1/2-H}L
	\end{align}
	is the left inverse of $Q^{H}_{\bar{X}}$. \cite[Lemma 2.4]{KST} says given any $h\in L^2\brac{[0,1];\R^n}$, we have
	\begin{align*}
		D^{H-1/2}_{0^+} I^{H-1/2}_{0^+} h=h,\qquad D^{H-1/2}_{1^-} I^{H-1/2}_{1^-} h=h.
	\end{align*}
	This combined with the fact that $D_1,D_2\in L^2\brac{[0,1];\R^n}$ implies for any $\Psi\in L^2\brac{[0,1];\R^n}$,
	\begin{align*}
		W Q^{H}_{\bar{X}}\Psi&=\brac{c_H^{-2}\Gamma(H-1/2)^{-2}L^\top t^{1/2-H}D^{H-1/2}_{1^-}t^{2H-1}D^{H-1/2}_{0^+}t^{1/2-H}L}\\&\qquad\brac{c_H^2\Gamma(H-1/2)^2 \bar{f}\brac{\bar{X}}t^{H-1/2} I^{H-1/2}_{0^+}t^{1-2H}I^{H-1/2}_{1^-}t^{H-1/2}\bar{f}\brac{\bar{X}}^\top\Psi}\\
		&=L^\top t^{1/2-H}D^{H-1/2}_{1^-}t^{2H-1}\brac{D^{H-1/2}_{0^+} I^{H-1/2}_{0^+}}t^{1-2H}I^{H-1/2}_{1^-}t^{H-1/2}\bar{f}\brac{\bar{X}}^\top\Psi\\
		&=L^\top t^{1/2-H}\brac{D^{H-1/2}_{1^-}I^{H-1/2}_{1^-}}t^{H-1/2}\bar{f}\brac{\bar{X}}^\top\Psi=\Psi.
	\end{align*}
	Therefore, $W$ is the left inverse of $Q^{H}_{\bar{X}}$ and $\ker Q^{H}_{\bar{X}}=\{0\}$. Moreover, we know $Q^{H}_{\bar{X}}$ is self-adjoint, hence
	\begin{align*}
		\operatorname{Im} Q^{H}_{\bar{X}}=\left[\ker \brac{Q^{H}_{\bar{X}}}^*\right]^\bot=\left[\ker Q^{H}_{\bar{X}}\right]^\bot=\{0\}^\bot=L^2\brac{[0,1];\R^n}.
	\end{align*}
	It follows that $Q^{H}_{\bar{X}}$ is bijective. It is also bounded on $L^2\brac{[0,1];\R^n}$ via Proposition
	\ref{prop_dotKHinL2},
	so we can conclude it has a bounded inverse by the inverse mapping theorem.
	
	Finally, the inverse of $Q^{H}_{\bar{X}}$ must coincide with the left inverse $W$ at \eqref{def_inverse_Q} and by using the formula for fractional derivatives in Appendix A, we get the second equation for $\brac{Q^{H}_{\bar{X}}}^{-1}$ in the statement of this lemma.
	
	\section{Conclusions and future work}\label{S:Conclusions}
	
	In this paper, we established the moderate deviations principle for slow-fast systems of the form (\ref{originalsystem}) where the slow component is driven by fractional Brownian motion. There are many interesting potential directions for future work on this topic.
	
	In this paper, the fast motion is driven by standard Brownian motion and is independent of the slow component. This was done in order to focus on the effect of fBm on the tail behavior of the slow component. If the fast component was driven by fBm as well, then one would first need to understand the proper ergodic behavior of the fast process, an issue still not fully resolved, see though \cite{LiSeiber2020} for some preliminary results in special cases. Feedback from the slow process into the fast process would also mean interaction of the ergodic behavior of the fast process with the fBm driving the slow process, see \cite{HaiLi} for partial preliminary results in this direction.
	
	Another interesting direction would be to include ``unbounded homogenization" terms in the slow component as done for similar systems driven by standard Brownian motion, see \cite{spiliopoulos2014fluctuation}.
	
	Lastly, establishing the MDP opens the door to the construction of provably-efficient accelerated Monte Carlo methods, like importance sampling, for the estimation of rare event probabilities. See \cite{MorKos2} for related work in the case where $H=1/2$.
	
	We plan to explore these avenues in future works on this topic.\\

\textbf{Funding\quad} Solesne Bourguin was partially supported by the Simons
Foundation Award 635136. Konstantinos Spiliopoulos was partially
supported by the National Science Foundation (DMS 1550918, DMS
2107856) and Simons Foundation Award 672441.
\\~\\
\textbf{Data Availability\quad} Data sharing is not applicable to this article as no datasets were generated or analyzed
during the current study.
\\~\\~\\
{\LARGE\textbf{Declarations}}
\\~\\~
\textbf{Conflict of interest\quad} The authors declare that they have no conflict of interest.

	\appendix
	
	\section{Fractional Brownian motion and pathwise stochastic integration}\label{A:Appendix}
	
	\subsection{Fractional Brownian motion: definition and main properties}
	
	A one-dimensional fractional Brownian motion (fBm) is a centered Gaussian process $W^H
	= \left\{ W^H_t \colon t \in [0,1] \right\}
	\subset L^2(\Omega)$, characterized by its
	covariance function
	\begin{equation*}
		R_H(t,s) = E (W^H_t W^H_s) = \frac{1}{2} \left( s^{2H} + t^{2H} -
		\left\vert t-s \right\rvert^{2H} \right).
	\end{equation*}
	It is straightforward to verify that increments of fBm are stationary. The parameter $H \in (0,1)$ is usually referred to as the Hurst exponent, Hurst parameter, or Hurst index.
	
	By Kolmogorov's continuity criterion, such a process admits a modification with continuous sample paths, and we always choose to work with such. In this case one may show in fact that almost every sample path is locally H\"older continuous of any order strictly less than $H$. It is this sense in which it is often said that the value of $H$ determines the regularity of the sample paths.
	
	Note that when $H = \frac{1}{2}$, the covariance function is $R_{\frac{1}{2}}(t,s) = t \wedge s$. Thus, one sees that $W^{\frac{1}{2}}$ is a standard Brownian motion, and in particular that its disjoint increments are independent. In contrast to this, when $H \neq \frac{1}{2}$, nontrivial increments are not independent. In particular, when $H > \frac{1}{2}$, the process exhibits long-range dependence.
	
	Note moreover that when $H \neq \frac{1}{2}$, the fractional Brownian motion is not a semimartingale, and the usual It\^o calculus therefore does not apply.
	
	Another noteworthy property of fractional Brownian motion is that it
	is self-similar in the sense that, for any constant $a >0$, the
	processes $\left\{ W^H_t \colon t \in [0,1]\right\}$ and $\left\{ a^{-H}
	W^H_{at}\colon t \in [0,1]\right\}$ have the same distribution.

        Finally, an $n$-dimensional fractional Brownian
          motion is a random vector where the components are
          independent one-dimensional fractional Brownian motions with
          the same Hurst parameter $H\in (0,1)$.
	
	The self-similarity and long-memory properties of the fractional
	Brownian motion make it an interesting and suitable input noise in
	many models in various fields such as analysis of financial time series,
	hydrology, and telecommunications. However, in order to develop
	interesting models based on fractional Brownian motion, one needs an
	integration theory with respect to it, which we present in the next subsection.

	\subsection{Pathwise stochastic integration with respect to fractional Brownian
		motion}\label{Eq:StochasticIntergation_prelim}
	~\\
	Stochastic integrals with respect to fractional Brownian motion can
	be understood, when $H \geq 1/2$, as generalized Stieltjes integral as
	introduced in the work of Z\"ahle \cite{Zahmain}. Let $f \in L^1([a,b])$
	and $\alpha > 0$. The left-sided and right-sided fractional
	Rienmann-Liouville integrals of $f$ of order $\alpha$ are defined for
	almost all $x \in [a,b]$ by
	\begin{align*}
		I_{a^+}^{\alpha}f(x) = \frac{1}{\Gamma(\alpha)}\int_a^x (x-y)^{\alpha-1}f(y)dy
	\end{align*}
	and
	\begin{align*}
		I_{b^-}^{\alpha}f(x) = \frac{1}{\Gamma(\alpha)}\int_x^b (y-x)^{\alpha-1}f(y)dy
	\end{align*}
	respectively, where $\Gamma(\alpha)$ is the Euler gamma function. This naturally leads to the definition of the function spaces
 \begin{align*}
     I_{a^+}^{\alpha}(L^p([a,b]))=\{g=I_{a^+}^{\alpha}(f):f\in L^p([a,b]) \}
 \end{align*}
 and
 \begin{align*}
    I_{b^-}^{\alpha}(L^p([a,b]))=\{g=I_{b^-}^{\alpha}(f):f\in L^p([a,b]) \}.
 \end{align*}

The following integration by parts formula holds
	\begin{align*}
		\int_a^b I_{a^+}^{\alpha}f(x) g(x)dx=\int_a^b f(x) I_{b^-}^{\alpha}g(x)dx
	\end{align*}
	for $f\in L^p([a,b]),g\in L^q([a,b])$ such that $1/p+1/q\leq 1+\alpha$.
	
	For $0 < \alpha <1$, we can define the fractional derivatives
	\begin{align*}
		D^\alpha_{a^+}f(x)=\frac{d}{dx}I^{1-\alpha}_{a^+} f(x)=\frac{1}{\Gamma(1-\alpha)}\frac{d}{dx}\int_a^x (x-t)^{-\alpha}f(t)dt
	\end{align*}
	and
	\begin{align*}
		D^\alpha_{b^-}f(x)=\frac{d}{dx}I^{1-\alpha}_{b^-} f(x)=\frac{1}{\Gamma(1-\alpha)}\frac{d}{dx}\int_x^b (t-x)^{-\alpha}f(t)dt
	\end{align*}
	as long as the right hand sides are well-defined. Furthermore, if $f \in I_{a^+}^{\alpha}\brac{L^p([a,b])}$ (respectively $f \in
	I_{b^-}^{\alpha}(L^p([a,b]))$ and $0 < \alpha <1$ then the previous fractional derivatives admit the Weyl representation
	\begin{align*}
		D_{a^+}^{\alpha}f(x) = \frac{1}{\Gamma(1-\alpha)} \left(
		\frac{f(x)}{(x-a)^{\alpha}} + \alpha \int_a^x \frac{f(x) - f(y)}{(x-y)^{\alpha-1}}dy \right)\mathds{1}_{(a,b)}(x)
	\end{align*}
	and
	\begin{align*}
		D_{b^-}^{\alpha}f(x) = \frac{1}{\Gamma(1-\alpha)} \left(
		\frac{f(x)}{(b-x)^{\alpha}} + \alpha \int_x^b \frac{f(x) - f(y)}{(y-x)^{\alpha-1}}dy \right)\mathds{1}_{(a,b)}(x),
	\end{align*}
	respectively, for almost all $x \in [a,b]$. There is also the integration by parts formula
	\begin{align*}
		\int_a^b D_{a^+}^{\alpha}f(x) g(x)dx=\int_a^b f(x) D_{b^-}^{\alpha}g(x)dx
	\end{align*}
	for $f \in I_{a^+}^{\alpha}(L^p([a,b])),g \in I_{b^-}^{\alpha}(L^q([a,b])) $ such that $1/p+1/q\leq 1+\alpha$.

	The upcoming lemma contains a useful technical result in \cite{KMS}.
	\begin{lemma}
		\label{lemma_Samko_bounded}
		Let $p\geq 1$ and $b>0$. Then the operator $t^\beta I_{0^+}^\alpha t^\gamma$ is bounded in $L^p([0,b])$ if $\alpha>0, \alpha+\beta+\gamma=0$ and $(\gamma+1)p>1$. Meanwhile, the operator $t^\beta I_{1^-}^\alpha t^\gamma$ is bounded in $L^p([0,b])$ if $\alpha>0, \alpha+\beta+\gamma=0$ and $(\alpha+\gamma)p<1$.
	\end{lemma}
	\begin{proof}
		This is a consequence of \cite[(5.45') and (5.46')]{KMS}.
	\end{proof}
	We refer to \cite{KMS} for more detailed properties of fractional operators.
	
	Next, let $f(a+) = \lim_{x \searrow 0} f(a + x)$ and $g(b-) =
	\lim_{x \searrow 0} g(b-x)$ and define
	\begin{align*}
		& f_{a^+}(x) = (f(x) - f(a+))\mathds{1}_{(a,b)}(x) \\
		& g_{b^-}(x) = (g(x) - g(b-))\mathds{1}_{(a,b)}(x).
	\end{align*}
	We recall from \cite{Zahmain} the definition of generalized Stieltjes
	fractional integrals with respect to irregular functions (in the sense
	of which we view the stochastic integrals with respect to fractional
	Brownian motion appearing in this paper).
	\begin{definition}[Generalized Stieltjes integral]
		\label{def_gen_stie}
		Suppose that $f$ and $g$ are functions such that $f(a+)$, $g(a+)$
		and $g(b-)$ exist, $f_{a^+} \in I_{a^+}^{\alpha}(L^p([a,b]))$ and $g_{b^-}
		\in I_{b^-}^{1-\alpha}(L^p([a,b]))$ for some $p,q \geq 1$, $1/p + 1/q \leq
		1$, $0 < \alpha <1$. Then the integral of $f$ with respect to $g$ is
		defined by
		\begin{align*}
			\int_a^b f dg = (-1)^{\alpha} \int_a^b
			D_{a^+}^{\alpha}f(x)D_{b^-}^{1-\alpha}g_{b^-}(x)dx + f(a+)\left(
			g(b-) - g(a+) \right).
		\end{align*}
	\end{definition}
	\begin{remark}
		If $\alpha p <1$, under the assumptions of the preceding definition,
		we have that $f \in I_{a^+}^{\alpha}(L^p([a,b]))$ and we can write
		\begin{align*}
			\int_a^b f dg = (-1)^{\alpha} \int_a^b
			D_{a^+}^{\alpha}f_{a^+}(x)D_{b^-}^{1-\alpha}g_{b^-}(x)dx.
		\end{align*}
	\end{remark}
	In \cite{Zahmain}, it was further shown that if $f$ and $g$ are respectively $\lambda$ and $\mu$-H\"{o}lder continuous such that $\lambda+\mu>1$, then the conditions for the generalized Stieltjes integral $\int_a^b f dg$ are satisfied for $p=q=\infty$ and $\alpha<\lambda,1-\alpha<\mu$. In particular, this class of generalized Stieltjes integrals with H\"{o}lder continuous $f,g$ coincides with the class of Riemann–Stieltjes integrals studied in \cite{You} by Young. We note here that any Young integrals appearing in this paper are constructed from H\"{o}lder continuous paths of a fractional Brownian motion. Further details are given in \cite[Section 5.1]{Zahmain}.

	\subsection{The Cameron-Martin space of fractional Brownian motion}\label{SS:fBm_preliminaries}
	Consider the deterministic kernel
	\begin{align*}
		K_H(t,s)=c_Hs^{1/2-H}\brac{\int_s^t (u-s)^{H-3/2}u^{H-1/2}du}\mathds{1}_{\{t>s\}}
	\end{align*}
	for which $c_H=\brac{H(2H-1)/\beta(2-2H,H-1/2)}^{1/2}$. Slightly abusing notation, we also write $K_H$ for the integral operator
	\begin{align*}
		K_Hg (s)=\int_0^s K_H(s,r)g(r)dr.
	\end{align*}	
	For $H \geq 1/2$, the operator $K_H$ can
	be represented as
	\begin{align*}
		K_Hg=c_H\Gamma(H-1/2)I^1_{0^+}t^{H-1/2}I^{H-1/2}_{0^+}t^{1/2-H}g.
	\end{align*}
	Additionally, we denote by $\dot{K}_H$ the ``derivation'' of the
	operator $K_H$, i.e.,
	\begin{align}
		\label{khdot}
		\dot{K}_Hg = c_H\Gamma(H-1/2)t^{H-1/2}I^{H-1/2}_{0^+}t^{1/2-H}g.
	\end{align}
	The Cameron-Martin space $\mathcal{H}_H$ associated with $W^H$ is
	\begin{align*}
		\mathcal{H}_H=\{K_H\hat{g}\colon \hat{g}\in L^2\brac{[0,1];\R^m}\},
	\end{align*}
	equipped with the inner product
	$\inner{g,f}_{\mathcal{H}_H}=\inner{\hat{g},\hat{f}}_{L^2\brac{[0,1];\R^m}}$. Note
	that later on, we will alternate between $\hat{g}$ and
	$K_H^{-1}g$, which are equivalent ways of writing the same
	quantity.
	
	In this paper, the noise process we consider in the slow-fast
	systems we study is of the form
	\begin{align*}\left\{\brac{W^H(t),B(t)}\colon t\in [0,1]\right\},\end{align*}
	where $B$ is a $m$-dimensional standard Brownian motion, $W^H$ is a $p$-dimensional fractional Brownian motion of Hurst parameter $H$ and they are independent. We will hence
	need to work with the Cameron-Martin space associated with the
	proces $(W^H,B)$, which, based on the previous description, is
	defined to be the space $\mathcal{S}$ given by
	\begin{align}
		\label{def_Cameron_Martin_joint}
		\left\{\brac{K_H\hat{g}_1,K_{1/2}\hat{g}_2}\colon \brac{\hat{g}_1,\hat{g}_2}\in L^2\brac{[0,1];\R^{m+p}}\right\}.
	\end{align}
	As a Cameron-Martin space, $\mathcal{S}$ is a Hilbert space
	equipped with the inner product given by
	\begin{align*}
		\inner{\brac{g_1,g_2},\brac{f_1,f_2}}_{\mathcal{S}}=\inner{g_1,f_1}_{\mathcal{H}_H}+\inner{g_2,f_2}_{\mathcal{H}_{1/2}}.
	\end{align*}
	Let us now state an important fact regarding the differentiability of
	elements in $\mathcal{H}_H$ when $H> 1/2$ which we will need
	throughout the paper.
	\begin{lemma}
		\label{lemma_derivative_control}
		If $H> 1/2$ and $u\in \mathcal{H}_H$ such that $u=K_H \hat{u},\hat{u}\in L^2([0,1];\R^n)$, then we have
		\begin{align*}
			\dot{u}(t)=\dot{K}_H\hat{u}(t)&=c_H\Gamma(H-1/2)t^{H-1/2}I^{H-1/2}_{0^+}t^{1/2-H}\hat{u}(t)\\
			&=c_H t^{H-1/2}\int_0^t (t-s)^{H-3/2}s^{1/2-H}\hat{u}_sds,
		\end{align*}
		such that $c_H=\brac{H(2H-1)/\beta(2-2H,H-1/2)}^{1/2}$.
		
		Meanwhile, if $H=1/2$ and $u\in \mathcal{H}_{1/2}$, then  $\dot{u}_t=\hat{u}_t$.
	\end{lemma}
	\begin{proof}
		This is a direct consequence of formula \eqref{khdot}.
	\end{proof}

	The following is another important property of the operator $\dot{K}_H$.
	\begin{proposition}
		\label{prop_dotKHinL2}
		The map $\dot{K}_H$ as described in Lemma \ref{lemma_derivative_control} is a bounded operator in $L^2([0,1];\R^n)$.
	\end{proposition}	
	\begin{proof}
		The assumptions of Lemma \ref{lemma_Samko_bounded} are satisfied
		for $p=2,\alpha=H-1/2,\beta=0$ and $\gamma=1/2-H$, hence the operator $I^{H-1/2}_{0^+}t^{1/2-H}$ is bounded in $L^2([0,1];\R^n)$.
		Since $\dot{K}_H=t^{H-1/2}I^{H-1/2}_{0^+}t^{1/2-H}$ based on Lemma \ref{lemma_derivative_control}, this implies
		\begin{align*}
			\norm{\dot{K}_H
				f}_{L^2([0,1];\R^n)}=\norm{t^{H-1/2}I^{H-1/2}_{0^+}t^{1/2-H}f}_{L^2([0,1];\R^n)}
			&\leq \norm{I^{H-1/2}_{0^+}t^{1/2-H}f}_{L^2([0,1];\R^n)}\\&\leq C\norm{f}_{L^2([0,1];\R^n)}.
		\end{align*}
	\end{proof}
	For more details about fractional Brownian motion, we refer the reader
	to the monographs \cite{biagini_stochastic_2008, Nua}.

	\subsection{Results related to Young integrals}
	The two results presented here provide us with a way of bounding Young
	integrals and with a version of change of variable formula for differential equations that contain Young integrals, respectively.
	\begin{lemma}[Young-Lo\'{e}ve's inequality]
		\label{lemma_Young}
		Let $f$ and $g$ be respectively $\alpha$ and $\beta$-H\"{o}lder continuous, such that $\alpha+\beta>1$. Then a.s one has
		\begin{align*}
			\abs{\int_r^t f_sdg_s-f_r(g_t-g_r)}\leq C \abs{f}_\alpha \abs{g}_\beta \abs{t-r}^{\alpha+\beta}.
		\end{align*}
		Moreover, assume $f$ is bounded then
		\begin{align*}
			\abs{\int_r^t f_sdg_s}\leq C\abs{f}_\alpha \abs{g}_\beta \abs{t-r}^{\alpha+\beta}+\abs{f}_\infty \abs{t-r}^\beta
			\leq C\abs{f}_\alpha \abs{g}_\beta \abs{t-r}^\beta.
		\end{align*}
	\end{lemma}
	\begin{proof}
		Refer to \cite[Proposition 6.4]{FV}.
	\end{proof}

	\begin{theorem}
		\label{Itoform}
		For $i=1,\ldots,m$, let $0<\alpha_i<1/2$, $f^i\in I^{\alpha_i}_{0+}(L^2([0,b]))$ be bounded and
		$g^i_{b-}\in I^{1-\alpha_i}_{b-}(L^2([0,b]))$, where the function $g^i_{b-}$ is defined below Lemma \ref{lemma_Samko_bounded}. Moreover, assume $h=(h^1,\ldots,h^m)$ such that
		\begin{align*}
			h^i_t=h^i_0+\int_0^t f^i_sdg^i_s.
		\end{align*}
		Then for any $\mathcal{C}^1$ mapping $F:\R^m\times \R\to \R^n$ such that $\frac{\partial F}{\partial x_i}\in \mathcal{C}^1,i=1,\ldots ,m$ and $r\leq t\leq T$, it holds that
		\begin{align*}
			F(h_t,t)-F(h_r,r)=\sum_{i=1}^m\int_r^t \frac{\partial F}{\partial x_i}(h_s,s)f^i_sdg^i_s +\int_r^t \frac{\partial F}{\partial s}(h_s,s) ds.
		\end{align*}
		In particular, this change of variable formula applies to the special case when $f_i$ and $g_i$ are respectively $\lambda_i$ and $\mu_i$-H\"{o}lder continuous such that $\lambda_i+\mu_i>1,i=1,\ldots,m$.
	\end{theorem}	
	
	\begin{proof}
		For the change of variable formula in the general case, we refer to \cite[Theorem 5.2]{Zah}.
		
		Now, let us consider the special case and assume there is a constant $C$ such that $\abs{f_i}_{\lambda_i}$,$\abs{g_i}_{\mu_i}<C$ and $\lambda_i+\mu_i>1$ for $1\leq i\leq m$. Then one can choose $\alpha_i$ in the interval $(0,1/2)$ such that $\lambda_i>\alpha_i$ and $\mu_i>1-\alpha_i$ for $1\leq i\leq m$. Based on the previous fact, Lemmas 13.2 and 13.2' in \cite{KMS} imply respectively that \begin{align*}
			f^i\in I^{\alpha_i}_{0+}(L^2([0,b])),\qquad g^i_{b-}\in I^{1-\alpha_i}_{b-}(L^2([0,b])).
		\end{align*}
		Moreover, $f^i$ as H\"{o}lder continuous functions on $[0,b]$ are necessarily bounded. Consequently, the general change of variable formula covers this particular case.
	\end{proof}

	
	\section{Regularity results and other technical lemmas}\label{A:AppendixB}
	This appendix gathers results related to Poisson equations as well as the technical lemmas required for the analysis
	of the control problems.
	\subsection{Results related to Poisson equations}
	The following theorem is a consequence of \cite[Theorem 2]{ParVer1} and \cite[Theorem 3]{ParVer2} for solutions of Poisson equations. Let $\mathcal{L}$ be the infinitesimal generator defined in (\ref{def_generator_normalized_Y}).
	\begin{theorem}
		\label{theorem_Poisson_equation}
		Recall $\mathcal{C}^{2,\zeta}(\R^n\times \mathcal{Y})$ for
		some $\zeta >0$ is the function space defined at the beginning of Section \ref{section_prelim}. Let $h \in \mathcal{C}^{2,\zeta}(\R^n\times \mathcal{Y})$ such that
		\begin{align*}
			\int_{\mathcal{Y}} h(x,y)\mu(dy)=0
		\end{align*}
		and that for some positive constants $K$ and $D_h$,
		\begin{align*}
			\abs{h(x,y)} +\abs{\nabla_x h(x,y)}+ \abs{\nabla^2_x h(x,y)}\leq K\brac{1+\abs{y}^{D_h}}
		\end{align*}
		uniformly with respect to $x$.
		Then, there is a unique solution to
		\begin{align*}
			\mathcal{L} u(x,y)=-h(x,y),\quad \int_\mathcal{Y} u(x,y) \mu(dy)=0.
		\end{align*}
		Moreover, $u(\cdot,y)\in \mathcal{C}^2$, $\nabla^2_x u\in
		\mathcal{C}(\R^n\times \mathcal{Y})$ and there exists a positive
		constant $M$ such that
		\begin{align*}
			\abs{u(x,y)}+\abs{\nabla_y u(x,y)}+\abs{\nabla_x u(x,y)}+\abs{\nabla^2_x u(x,y)}+\abs{\nabla_y\nabla_x u(x,y)}\leq M(1+\abs{y}^{D_h}).
		\end{align*}
	\end{theorem}

	\begin{remark}
		\label{remark_poisson_equation}
		Consider the Poisson equation in
		\eqref{equation_Poisson_g}. Under Conditions
		\ref{condH1} and \ref{condH2-A}, Theorem
		\ref{theorem_Poisson_equation} states that there
		exists a positive constant $C$ such that, uniformly,
		\begin{align}
			\label{estimate_Poissonsolution_h1}
			\abs{\phi(x,y)}+\abs{\nabla_y \phi(x,y)}+\abs{\nabla_x \phi(x,y)}+\abs{\nabla^2_x \phi(x,y)}+\abs{\nabla_y\nabla_x \phi(x,y)}< C.
		\end{align}
		On the other hand, under Conditions \ref{condH1} and \ref{condH2-B}, Theorem \ref{theorem_Poisson_equation} states that there
		exists a positive constant $C$ such that, uniformly
		with respect to $x$,
		\begin{align}
			\label{estimate_Poissonsolution_h2}
			\abs{\phi(x,y)}+\abs{\nabla_y \phi(x,y)}+\abs{\nabla_x \phi(x,y)}+\abs{\nabla^2_x \phi(x,y)}+\abs{\nabla_y\nabla_x \phi(x,y)}< C\brac{1+\abs{y}^{D_g}}.
		\end{align}
	\end{remark}

	\subsection{Ancillary results related to the control problems}
	This subsection gathers all technical results related to the
	study of the control problems appearing throughout the paper.

	\begin{lemma}[Lemma 5.2 in \cite{HuSalSpi}]
		\label{lemma_KostSalins}
		Let $H,H'$ be Hilbert spaces and $a:H \to H'$ be a bounded linear operator. Moreover, let $q=aa^*$ and $q^{-1}$ be the inverse of $q$. Then for any $u\in H$,
		\begin{align*}
			\norm{a^*q^{-1}a u}_{H}\leq \norm{u}_{H}.
		\end{align*}
	\end{lemma}
	\begin{lemma}
		\label{lemma_Q_invertible}
		Assume that for all $x$ and non-zero $z\in \R^n$,
		\begin{align*}
			\inner{\int_{\mathcal{Y}}\nabla_y\phi\brac{x,y}\sigma\brac{y} \brac{\nabla_y\phi\brac{x,y}\sigma\brac{y}}^\top\mu(dy)z,z}>0.
		\end{align*}
		Then, the operator $Q^{H}_x$ defined in \eqref{def_Q} is invertible and its inverse $(Q^{H}_x)^{-1}$ is a bounded in $L^2\brac{[0,1];\R^n}$.
	\end{lemma}
	\begin{proof}
		Using the operators $\pi,\pi^*,\rho,\rho^*$ defined in Section \ref{prooflaplaceupper}, we have $Q^{H}_xh=\brac{\pi\pi^*+\rho\rho^*}h$ with
		\begin{align*}\rho\rho^*h(t)=\int_{\mathcal{Y}} {\nabla_y\phi\brac{x,y} \sigma\brac{y}}\brac{\nabla_y\phi\brac{x,y} \sigma\brac{y}}^\top h(t,y)\mu(dy).\end{align*}
		Furthermore, $\pi\pi^*,\rho\rho^*$ are positive and
		self-adjoint operators, which means that $Q^{H}_x$ is also
		positive and self-adjoint. In addition, the fact that
		$Q^{H}_{x}\geq \rho\rho^*$ and Condition \ref{condH2-A} or
		\ref{condH2-B} imply that $(Q^{H}_{x})^2\geq \brac{\rho\rho^*}^2>0$. This leads to
		\begin{align*}
			\inf_{\norm{h}_{L^2\brac{[0,1];\R^n}}=1}\norm{Q^{H}_xh}_{L^2\brac{[0,1];\R^n}}&=\inf_{\norm{h}_{L^2\brac{[0,1];\R^n}}=1}\inner{(Q^{H}_x)^2 h,h}_{L^2\brac{[0,1];\R^n}}\\&\geq
			\inf_{\norm{h}_{L^2\brac{[0,1];\R^n}}=1}\inner{\brac{\rho\rho^*}^2h,h}_{L^2\brac{[0,1];\R^n}}>0,
		\end{align*}
		so that $Q^{H}_x$ is bounded from below and $\ker Q^{H}_x=\{0\}$. This combined
		with self-adjointness implies
		\begin{align*}
			\operatorname{Im} Q^{H}_{x}=\left[\ker \brac{Q^{H}_{x}}^*\right]^\bot=\left[\ker Q^{H}_{x}\right]^\bot=\{0\}^\bot=L^2\brac{[0,1];\R^n}.
		\end{align*}		
		It follows that $Q^{H}_{\bar{X}}$ is bijective. The operator $Q^{H}_{\bar{X}}$ is also bounded in $L^2\brac{[0,1];\R^n}$ via Proposition
		\ref{prop_dotKHinL2},
		so we can conclude it has a bounded inverse by the inverse mapping theorem.
	\end{proof}
	\begin{lemma}
		\label{lemma_boundedcontrol}
		It can be assumed that there exists a finite constant
		$N$ such that, almost surely, the control process $w^{\epsilon}$
		appearing in the variational representation \eqref{equation_varrep} satisfies
		\begin{align*}
			\sup_{\epsilon>0}\norm{w^\epsilon}^2_{\mathcal{S}}\leq N.
		\end{align*}
	\end{lemma}
	\begin{proof}
		This is an immediate consequence of \cite[Theorem 3.2]{Zha}. 
	\end{proof}

	\begin{lemma}
		\label{lemma_estimatefromMorKos2}
		Assume $w^\epsilon\in \mathcal{S}$ is a control such that
		\begin{align*}
			\sup_{\epsilon>0}\norm{w^\epsilon}^2_{\mathcal{S}}=\sup_{\epsilon>0}\int_0^1 \abs{\hat{u}^\epsilon_s}^2+\abs{\hat{v}^\epsilon_s}^2 ds< N
		\end{align*}
		for some finite constant $N$. Then, under Condition
		\ref{condH1}, it holds that for $\epsilon_0>0$ small enough,
		\begin{align*}
			\sup_{\epsilon<\epsilon_0}\E{\int_0^1\abs{Y^\eu_s}^2ds}<C
		\end{align*}
		for some constant $C>0$, which further implies that
		\begin{align*}
			\E{\sup_{t\in [0,1]}\abs{Y^\eu_t}}\leq \frac{C}{\sqrt{\epsilon}}.
		\end{align*}
	\end{lemma}
	\begin{proof}
		The first estimate was proven in \cite[Lemma
		3.1]{MorKos2}. For the second estimate, the dissipative property of the drift coefficient of $Y^\eu$ and It\^{o}'s formula yield
		\begin{align*}
			Y^\eu_t=e^{-\frac{1}{\epsilon}\Gamma t}y_0+\int_0^t \frac{1}{\epsilon} e^{-\frac{1}{\epsilon}(t-s)}\zeta(Y^\eu)ds&+\int_0^t\frac{h(\epsilon)}{\sqrt{\epsilon}}e^{-\frac{1}{\epsilon}(t-s)}\sigma\brac{Y^\eu_s}\dot{v}^\epsilon_s ds\\&+\int_0^t\frac{1}{\sqrt{\epsilon}}e^{-\frac{1}{\epsilon}(t-s)}\sigma\brac{Y^\eu_s}dB_s.
		\end{align*}
		We then apply the Burkh\"{o}lder-Davis-Gundy
		inequality to the It\^{o} integral term and
		H\"{o}lder's inequality to the Riemann integral terms to get
		\begin{align*}
			\E{\sup_{t\in [0,1]}\abs{Y^\eu_t}}
			&\leq  \sup_{t\in [0,1]}e^{-\frac{1}{\epsilon}\Gamma t}y_0+\frac{1}{\epsilon}\sqrt{\int_0^t e^{-\frac{2}{\epsilon}\Gamma(t-s)} ds}\sqrt{\E{\int_0^1 \abs{\zeta(Y^\eu)}^2 ds }}\\
			&\quad +\frac{h(\epsilon)}{\sqrt{\epsilon}}\sqrt{\int_0^t e^{-\frac{2}{\epsilon}\Gamma(t-s)} ds\int_0^1 \abs{\dot{v}^\epsilon_s}^2 ds }+\frac{1}{\sqrt{\epsilon}}\sqrt{\E{\int_0^1 e^{-\frac{2}{\epsilon}\Gamma(t-s)}\abs{\sigma(Y^\eu_s)}^2 ds}}.
		\end{align*}
		Since $\sigma(y)\sigma^T(y)$ is bounded and $\zeta(y)$ is sublinear,
		the first estimate of this lemma can be applied to the
		expression $\E{\int_0^1 \abs{\zeta(Y^\eu)}^2 ds
		}$. Then, the simple fact that $\int_r^t
		e^{-\frac{2}{\epsilon}\Gamma(t-s)}ds\leq
		\epsilon\int_0^\infty e^{-2\Gamma
			s}ds=\frac{\epsilon}{2\Gamma}$ implies that
		\begin{align*}
			\E{\sup_{t\in [0,1]}\abs{Y^\eu_t}}\leq C\brac{\frac{1}{\sqrt{\epsilon}}+h(\epsilon)}\leq \frac{C}{\sqrt{\epsilon}}.
		\end{align*}
	\end{proof}
	\begin{lemma}
		\label{lemma_estimates_notyoung}
		Assume $w^\epsilon\in \mathcal{S}$ is a control such that
		\begin{align*}
			\sup_{\epsilon>0}\norm{w^\epsilon}^2_{\mathcal{S}}=\sup_{\epsilon>0}\int_0^1 \abs{\hat{u}^\epsilon_s}^2+\abs{\hat{v}^\epsilon_s}^2 ds< N
		\end{align*}
		for some finite constant $N$.
		\begin{enumerate}
			\item[(i)] Under Conditions \ref{condH1} and \ref{condH2-A}, there exist
			constants $C$ that change from line to line such that
			\begin{align*}
				&\E{\sup_{\substack{0\leq r,t\leq 1 \\ \abs{r-t}<\rho}}  \abs{ \int_r^t \nabla_x\phi\brac{X^\eu_s,Y^\eu_s} g\brac{X^\eu_s,Y^\eu_s }ds}}\leq C\rho,\\
				&\E{\sup_{\substack{0\leq r,t\leq 1 \\
							\abs{r-t}<\rho}}
					\abs{\int_r^t\nabla_y\phi\brac{X^\eu_s,Y^\eu_s}\sigma\brac{Y^\eu_s}dB_s
					}^2}\leq C\rho\\
				&\E{\sup_{\substack{0\leq r,t\leq 1 \\ \abs{r-t}<\rho}}  \abs{ \int_r^t \nabla_y\phi\brac{X^\eu_s,Y^\eu_s} \sigma\brac{Y^\eu_s}\dot{v}^\epsilon_sds}^2}\leq C\rho\\
				&\E{\sup_{\substack{0\leq r,t\leq 1 \\ \abs{r-t}<\rho}}  \abs{ \int_r^t \nabla_x\phi\brac{X^\eu_s,Y^\eu_s} f\brac{Y^\eu_s}\dot{u}^\epsilon_sds}^2}\leq C\rho,\\
				&\E{\sup_{\substack{0\leq r,t\leq 1 \\ \abs{r-t}<\rho}}\abs{\int_0^t f\brac{X^\eu_s,Y^\eu_s}\dot{u}^\epsilon_sds}^2}\leq C\rho.
			\end{align*}
			\item[(ii)] Under Conditions \ref{condH1} and \ref{condH2-B}, there exist
			constants $C$ that change from line to line such that for any $q$ in $\Big(1,\frac{1}{D_f+D_g}\Big]$, we have
			\begin{align*}
				&\E{\sup_{\substack{0\leq r,t\leq 1 \\ \abs{r-t}<\rho}}  \abs{ \int_r^t \nabla_x\phi\brac{X^\eu_s,Y^\eu_s} g\brac{X^\eu_s,Y^\eu_s }ds}^q}\leq C\rho^{q-1},\\
				&\E{\sup_{\substack{0\leq r,t\leq 1 \\ \abs{r-t}<\rho}}  \abs{\int_r^t\nabla_y\phi\brac{X^\eu_s,Y^\eu_s}\sigma\brac{Y^\eu_s}dB_s }^{2q}}\leq C\rho^{q-1}\\
				&\E{\sup_{\substack{0\leq r,t\leq 1 \\ \abs{r-t}<\rho}}  \abs{ \int_r^t \nabla_y\phi\brac{X^\eu_s,Y^\eu_s} \sigma\brac{Y^\eu_s}\dot{v}^\epsilon_sds}^{2q}}\leq C\rho^{q-1},\\
				&\E{\sup_{\substack{0\leq r,t\leq 1 \\ \abs{r-t}<\rho}}  \abs{ \int_r^t \nabla_x\phi\brac{X^\eu_s,Y^\eu_s} f\brac{Y^\eu_s}\dot{u}^\epsilon_sds}^{2q}}\leq C\rho^{q-1},\\
				&\E{\sup_{\substack{0\leq r,t\leq 1 \\ \abs{r-t}<\rho}}\abs{\int_0^t f\brac{Y^\eu_s}\dot{u}^\epsilon_sds}^{2q}}\leq C\rho^{q-1}.
			\end{align*}
		\end{enumerate}
	\end{lemma}
	\begin{proof}
		We start with part $(i)$. The first estimate is straightforward
		due to the boundedness of $\nabla_x\phi(x,y)$ stated
		in \eqref{estimate_Poissonsolution_h1} and the
		boundedness of $g(x,y)$ guaranteed by Condition \ref{condH2-A}. For the second
		estimate, we assume that $0\leq r\leq t\leq 1$ and apply the Burkh\"{o}lder-Davis-Gundy inequality to obtain
		\begin{align*}
			&\E{\sup_{\substack{0\leq r\leq t\leq 1 \\
						\abs{r-t}<\rho}}  \abs{
					\int_r^t\nabla_y\phi\brac{X^\eu_s,Y^\eu_s}\sigma\brac{Y^\eu_s}dB_s}^2}\\
			&\qquad\qquad\qquad\qquad\qquad\qquad\qquad \leq \E{  \brac{\int_r^{r+\rho}\abs{\nabla_y\phi\brac{X^\eu_s,Y^\eu_s} \sigma\brac{Y^\eu_s}}^2 ds}}\leq C\rho.
		\end{align*} 	
		For the third estimate, we can write
		\begin{align*}
			&\E{\sup_{\substack{0\leq r,t\leq 1 \\ \abs{r-t}<\rho}}  \abs{ \int_r^t \nabla_y\phi\brac{X^\eu_s,Y^\eu_s} \sigma\brac{Y^\eu_s}\dot{v}^\epsilon_sds}^2}\\
			&\qquad\qquad\qquad\qquad\qquad \leq \E{\sup_{\substack{0\leq r,t\leq 1 \\ \abs{r-t}<\rho}}  \brac{\int_r^t \abs{\nabla_y\phi\brac{X^\eu_s,Y^\eu_s} \sigma\brac{Y^\eu_s}}^2 ds\int_0^1 \abs{\dot{v}^\epsilon_s}^2ds}}\\
			&\qquad\qquad\qquad\qquad\qquad\leq C\E{\sup_{\substack{0\leq r,t\leq 1 \\ \abs{r-t}<\rho}}  \brac{\int_r^t \abs{\nabla_y\phi\brac{X^\eu_s,Y^\eu_s} \sigma\brac{Y^\eu_s}}^2 ds}}
			\leq C \rho.
		\end{align*}
		The last inequality in part $(i)$ is a consequence of
		the boundedness of $\sigma(y)\sigma^\top(y)$ in Condition
		\ref{condH1} and the boundedness of $\nabla_x\phi(x,y)$ stated in \eqref{estimate_Poissonsolution_h1} (requiring Condition \ref{condH2-A}). Finally, the two remaining estimates of part $(i)$ are derived similarly to the previous one.
		
		We continue with part $(ii)$.  For the first inequality, the sublinear
		growth of $\nabla_x\phi(x,y)$ in $y$ stated at
		\eqref{estimate_Poissonsolution_h2} (requiring Condition \ref{condH2-B}) and the sublinear growth of $g(x,y)$ in $y$ from Condition \ref{condH2-B} imply for any $q$ in $\Big(1,\frac{1}{D_g}\Big]$,
		\begin{align*}
			\E{\sup_{\substack{0\leq r,t\leq 1 \\ \abs{r-t}<\rho}}  \abs{ \int_r^t \nabla_x\phi\brac{X^\eu_s,Y^\eu_s} g\brac{X^\eu_s,Y^\eu_s }ds}^q}\leq C\E{\brac{\sup_{\substack{0\leq r,t\leq 1 \\ \abs{r-t}<\rho}}  { \int_r^t {1+\abs{Y^\eu_s}^{2D_g}} ds}}^q}&\\
			\leq C\rho^{q-1}\E{\sup_{\substack{0\leq r,t\leq 1 \\ \abs{r-t}<\rho}}  { \int_r^t \brac{1+\abs{Y^\eu_s}^{2qD_g}} ds}}\leq C\rho^{q-1},&
		\end{align*}
		where the last inequality is due to Lemma
		\ref{lemma_estimatefromMorKos2}. For the second
		estimate, assume that $0\leq r\leq t\leq 1$. Then, the
		Burkh\"{o}lder-Davis-Gundy inequality combined with
		the sublinear growth of $\nabla_y\phi(x,y)$ in $y$
		(requiring Condition \ref{condH2-B}) and the boundedness of $\sigma(y)\sigma^\top(y)$ in
		Condition \ref{condH1} imply that for any $q$ in $\Big(1,\frac{1}{D_g}\Big]$,
		\begin{align*}
			&\E{\sup_{\substack{0\leq r\leq t\leq 1 \\
						\abs{r-t}<\rho}}  \abs{
					\int_r^t\nabla_y\phi\brac{X^\eu_s,Y^\eu_s}\sigma\brac{Y^\eu_s}dB_s}^{2q}}\\
			&\qquad\qquad\qquad\qquad\qquad\qquad\qquad\qquad\qquad \leq \E{  \brac{\int_r^{r+\rho}\abs{\nabla_y\phi\brac{X^\eu_s,Y^\eu_s} \sigma\brac{Y^\eu_s}}^2 ds}^{q}}\\
			&\qquad\qquad\qquad\qquad\qquad\qquad\qquad\qquad\qquad
			\leq C\rho^{q-1}\E{\sup_{\substack{0\leq r,t\leq 1 \\ \abs{r-t}<\rho}}  { \int_r^t {1+\abs{Y^\eu_s}^{2qD_g}} ds}}\leq C\rho^{q-1}.
		\end{align*}
		
		The arguments for the three remaining estimates of
		part $(ii)$ are similar, so we will handle one case
		only. The sublinear growth of $\nabla_x\phi(x,y)$ in $y$ stated at \eqref{estimate_Poissonsolution_h2} (requiring Condition \ref{condH2-B}) and sublinear growth of $f(y)$ in $y$ in
		Condition \ref{condH2-B} imply that for any $q$ in $\Big(1,\frac{1}{D_f+D_g}\Big]$,
		\begin{align*}
			&\E{\sup_{\substack{0\leq r,t\leq 1 \\ \abs{r-t}<\rho}}  \abs{ \int_r^t \nabla_x\phi\brac{X^\eu_s,Y^\eu_s} f\brac{Y^\eu_s}\dot{u}^\epsilon_sds}^{2q}}\\
			&\qquad\qquad\qquad\qquad\qquad \leq \E{\sup_{\substack{0\leq r,t\leq 1 \\ \abs{r-t}<\rho}}  { \brac{\int_r^t \abs{\nabla_x\phi\brac{X^\eu_s,Y^\eu_s} f\brac{Y^\eu_s}}^2ds}^{q} \brac{\int_0^1\abs{\dot{u}^\epsilon_s}^2 ds}^{q}}}\\
			&\qquad\qquad\qquad\qquad\qquad \leq C\rho^{q-1}\E{\sup_{\substack{0\leq r,t\leq 1 \\ \abs{r-t}<\rho}}  { \int_r^t {1+\abs{Y^\eu_s}^{2q(D_f+D_g)}} ds}}\leq C\rho^{q-1},
		\end{align*}
		where the last inequality is once again obtained using
		Lemma \ref{lemma_estimatefromMorKos2}.
	\end{proof}
	\begin{lemma}
		\label{lemma_Holder_Y}
		Assume $w^\epsilon\in \mathcal{S}$ is a control such that
		\begin{align*}
			\sup_{\epsilon>0}\norm{w^\epsilon}^2_{\mathcal{S}}=\sup_{\epsilon>0}\int_0^1
			\left[ \abs{\hat{u}^\epsilon_s}^2+\abs{\hat{v}^\epsilon_s}^2 \right] ds< N
		\end{align*}
		for some finite constant $N$. Under Condition \ref{condH1}, for $0<\alpha\leq 1/2$, we have the almost sure H\"{o}lder estimate
		\begin{align*}{\abs{Y^\eu}_{\alpha}}\leq \frac{C}{\sqrt{\epsilon}}.\end{align*}
	\end{lemma}
	\begin{proof}
		Without loss of generality, let us assume $t>r$. The dissipative property of the drift coefficient of $Y^\eu$ and It\^{o}'s formula yield
		\begin{align*}
			Y^\eu_t=e^{-\frac{1}{\epsilon}\Gamma (t-r)}Y^\eu_r+\int_r^t \frac{1}{\epsilon} e^{-\frac{1}{\epsilon}(t-s)}\zeta(Y^\eu)ds&+\int_r^t\frac{h(\epsilon)}{\sqrt{\epsilon}}e^{-\frac{1}{\epsilon}(t-s)}\sigma\brac{Y^\eu_s}\dot{v}^\epsilon_s ds\\&+\int_r^t\frac{1}{\sqrt{\epsilon}}e^{-\frac{1}{\epsilon}(t-s)}\sigma\brac{Y^\eu_s}dB_s.
		\end{align*}
		Now, by subtracting $Y^\eu_r$ from both sides and
		applying H\"{o}lder's inequality along with the Burkh\"{o}lder-Davis-Gundy inequality, we get
		\begin{align}
			\label{estimate_Holder}
			\E{\abs{Y^\eu_t-Y^\eu_r}}&\leq  \abs{e^{-\frac{1}{\epsilon}\Gamma (t-r)}-1} \E{\abs{Y^\eu_r}}+\frac{1}{\epsilon}\sqrt{\int_r^t e^{-\frac{2}{\epsilon}\Gamma(t-s)} ds}\E{\sqrt{\int_0^1 \abs{\zeta(Y^\eu)}^2 ds }}\nonumber\\
			&\quad + \frac{h(\epsilon)}{\sqrt{\epsilon}}\sqrt{\int_r^t e^{-\frac{2}{\epsilon}\Gamma(t-s)} ds}\E{\sqrt{\int_0^1 \abs{\dot{v}^\epsilon_s}^2 ds }}\nonumber\\
			&\quad +\frac{1}{\sqrt{\epsilon}}\E{\sqrt{\int_r^t e^{-\frac{2}{\epsilon}\Gamma(t-s)}\abs{\sigma(Y^\eu_s)\sigma^\top(Y^\eu_s)}ds}}.
		\end{align}
		To bound the first term on the right-hand side, we
		combine the second estimate in Lemma
		\ref{lemma_estimatefromMorKos2} and the fact that
		$e^{-\frac{1}{\epsilon}\Gamma
			(t-r)}-1=\frac{1}{\epsilon}\int_r^t
		e^{-\frac{1}{\epsilon}\Gamma(t-s)}ds\leq
		\abs{t-r}$. For the second term, note that $\int_r^t
		e^{-\frac{2}{\epsilon}\Gamma(t-s)}ds=C\epsilon
		\abs{t-r}.$ Moreover, the sublinearity of $\zeta(y)$
		and the first estimate in Lemma
		\ref{lemma_estimatefromMorKos2} yield a finite bound on the
		expression $\E{\sqrt{\int_0^1 \abs{\zeta(Y^\eu)}^2 ds
		}}.$ The third term on the right-hand side of
		\eqref{estimate_Holder} can be treated similarly with
		the help of Lemma
		\ref{lemma_boundedcontrol}. Regarding the last term, recall that $\sigma(y)\sigma^T(y)$ is bounded in Condition \ref{condH1}. Thus, we have
		\begin{align*}
			\E{\sup_{0\leq r,t\leq 1} \abs{Y^\eu_t-Y^\eu_r}}\leq C\frac{1}{\sqrt{\epsilon}}\abs{t-r}^{1/2}.
		\end{align*}
		The Kolmogorov Continuity Theorem then yields the almost sure H\"{o}lder contintuity of $Y^\eu$.
	\end{proof}
	\begin{lemma}
		\label{lemma_Holder_X}
		Assume $w^\epsilon\in \mathcal{S}$ is a control such that
		\begin{align*}
			\sup_{\epsilon>0}\norm{w^\epsilon}^2_{\mathcal{S}}=\sup_{\epsilon>0}\int_0^1
			\left[ \abs{\hat{u}^\epsilon_s}^2+\abs{\hat{v}^\epsilon_s}^2 \right] ds< N
		\end{align*}
		for some finite constant $N$. Under Conditions \ref{condH1} and \ref{condH2-A} or \ref{condH2-B}, there exists a constant C and $\epsilon_0$ small enough such that for $0<\beta\leq 1/2$,
		\begin{align*}
			\E{\sup_{\epsilon<\epsilon_0}\abs{X^\eu}_{\beta}}<C.
		\end{align*}
	\end{lemma}
	\begin{proof} We begin by proving the result under Conditions
		\ref{condH1} and \ref{condH2-A}. According to Condition \ref{condH2-A}, $f(x,y)$ is Lipschitz-continuous and bounded, so that $f(x,y)$ is also $\gamma$-H\"{o}lder continuous for $0<\gamma\leq 1$. This further implies
		\begin{align*}
			\E{\abs{f\brac{X^\eu_t,Y^\eu_t}-f\brac{X^\eu_r,Y^\eu_r}}}
			&\leq
			\E{\abs{f\brac{X^\eu_t,Y^\eu_t}-f\brac{X^\eu_r,Y^\eu_t}}}\\
			&\quad +\E{\abs{f\brac{X^\eu_r,Y^\eu_t}-\brac{X^\eu_r,Y^\eu_r} }}\\
			&\leq C\E{\abs{X^\eu_t-X^\eu_r}}+\E{\abs{Y^\eu_t-Y^\eu_r}^{\gamma}}\\
			&\leq C\brac{\E{\abs{X^\eu_t-X^\eu_r}}+\epsilon^{-\frac{\gamma}{2}}\abs{t-r}^\gamma},
		\end{align*}
		and hence that for $0<\gamma\leq 1$,
		\begin{align*}
			\E{\abs{f\brac{X^\eu,Y^\eu}}_\gamma}\leq C\brac{\E{\abs{X^\eu}_\gamma}+\epsilon^{-\frac{\gamma}{2}}}.
		\end{align*}
		This last estimate, together with the Young-Lo\'{e}ve inequality in Lemma \ref{lemma_Young}
		imply that for $1-H<\beta\leq 1$,
		\begin{align}
			\label{estimate_young_h1}
			\E{\abs{\int_r^t f\brac{X^\eu_s,Y^\eu_s}dW^H_s }}&\leq  C\E{\abs{f\brac{X^\eu,Y^\eu}}_{\beta} }\abs{t-r}^H\nonumber\\
			&\leq C\E{{\abs{X^\eu}_\beta+\epsilon^{-\frac{\beta}{2}}}}\abs{t-r}^H.
		\end{align}
		Meanwhile, a similar estimate to the one stated in
		part $(i)$ of Lemma \ref{lemma_estimates_notyoung}
		states that
		\begin{align*}
			\E{\abs{\int_0^t f\brac{X^\eu_s,Y^\eu_s}\dot{u}^\epsilon_sds}}&\leq C\abs{r-t}^{\frac{1}{2}}.
		\end{align*}
		Moreover, boundedness of $g(x,y)$ in Condition \ref{condH2-A} yields
		\begin{align*}
			\E{\abs{\int_r^t g\brac{X^\eu_s,Y^\eu_s}ds} }\leq C\abs{t-r}.
		\end{align*}
		Thus, using the estimate $\E{\abs{Y^\eu}_{\alpha}}\leq
		C\frac{1}{\sqrt{\epsilon}},\alpha\leq \frac{1}{2}$ in
		Lemma \ref{lemma_Holder_Y} (which requires Condition
		\ref{condH1}), we can deduce that, for $1-H<\beta\leq 1/2$,
		\begin{align*}
			\E{\abs{X^\eu_t-X^\eu_r}}\leq C\brac{\E{\sqrt{\epsilon}\abs{X^\eu}_\beta} \abs{t-r}^H+\abs{t-r}^{\frac{1}{2}}+\abs{t-r}},
		\end{align*}
		and consequently,
		\begin{align*}
			\E{\abs{X^\eu}_\beta}\leq C\brac{\E{\sqrt{\epsilon}\abs{X^\eu}_\beta} \abs{t-r}^{H-\beta}+\abs{t-r}^{\frac{1}{2}-\beta}+\abs{t-r}^{1-\beta}}.
		\end{align*}
		Now, by choosing $\epsilon_0$ small enough, we get
		$\E{\abs{X^\eu}_\beta}\leq C$ for some constant
		$C$. Since for $0<\beta_1\leq \beta_2\leq 1$,
		$\beta_2$-H\"{o}lder continuity of $X^\eu$ implies
		$\beta_1$-H\"{o}lder continuity, the conclusion
		follows.
		
		We now present a proof of the claim under Conditions
		\ref{condH1} and \ref{condH2-B}. Under Condition
		\ref{condH2-B}, $f(y)$ is $M_f$-H\"{o}lder continuous while
		$Y^\eu$ is $\frac{1}{2}$-H\"{o}lder continuous by Lemma
		\ref{lemma_Holder_Y}, so that
		\begin{align*}
			\E{\abs{f\brac{Y^\eu_t}-f\brac{Y^\eu_r}}}
			\leq C\abs{Y^\eu_t-Y^\eu_r}^{M_f}
			\leq C\epsilon^{-\frac{M_f}{2}}\abs{t-r}^{\frac{M_f}{2}}
		\end{align*}
		or equivalently
		\begin{align*}
			\E{\abs{f\brac{Y^\eu}}_{\frac{M_f}{2}}}  \leq C\epsilon^{-\frac{M_f}{2}}.
		\end{align*}
		Then, the Young-Lo\'{e}ve inequality in Lemma
		\ref{lemma_Young} implies that,  for $1-\frac{M_f}{2}<K\leq H$,
		\begin{align}
			\label{estimate_young_h2}
			\E{\abs{\int_r^t f\brac{Y^\eu_s}dW^H_s }} &\leq  {{\abs{Y^\eu}}_{\frac{M_f}{2}}}\abs{t-r}^{\frac{M_f}{2}+K}\abs{W^H}_K+\abs{f(Y^\eu_r)}\abs{W^H_t-W^H_r}\nonumber\\
			&\leq C\brac{{\abs{Y^\eu}}_{\frac{M_f}{2}}\abs{t-r}^{\frac{M_f}{2}+K}+\E{\sup_{t\in[0,1]} \abs{Y^\eu_t}^{D_f}} \abs{t-r}^K}\nonumber\\
			&\leq C\brac{\epsilon^{-\frac{M_f}{2}}+\epsilon^{-\frac{1}{2}}}\abs{t-r}^K
			\leq C\epsilon^{-\frac{1}{2}}\abs{t-r}^K,
		\end{align}
		where the first inequality is obtained by Condition
		\ref{condH2-B} and the last inequality is a
		consequence of the estimate
		$\E{\abs{Y^\eu}_{\alpha}}\leq
		C\frac{1}{\sqrt{\epsilon}},\alpha\leq \frac{1}{2}$ in
		Lemma \ref{lemma_Holder_Y} (which requires Condition
		\ref{condH1}). Moreover, similar calculations to those
		performed in the proof of part $(ii)$ of Lemma
		\ref{lemma_estimates_notyoung} yield that, for any $q$ in $\bigg(1,\frac{1}{D_f}\bigg]$,
		\begin{align*}
			\E{\abs{\int_r^t f\brac{Y^\eu_s}\dot{u}^\epsilon_sds}}&\leq C\abs{t-r}^{\frac{1}{2}-\frac{1}{2q}},
		\end{align*}
		as well as
		\begin{align*}
			\E{\abs{\int_r^t g\brac{X^\eu_s,Y^\eu_s}ds} }\leq C\abs{t-r}^{\frac{1}{2}}.
		\end{align*}
		Consequently, we have
		\begin{align*}
			\E{\abs{X^\eu_t-X^\eu_r}}\leq C\brac{\abs{t-r}^K+\sqrt{\epsilon}h(\epsilon)\abs{t-r}^{\frac{1}{2}-\frac{1}{2q}}+\abs{t-r}^{\frac{1}{2}}}.
		\end{align*}
		By choosing $\epsilon_0$ small enough and noting that $K>1-\frac{M_f}{2}\geq \frac{1}{2}$, we arrive at
		\begin{align*}
			\E{\sup_{\epsilon<\epsilon_0}\abs{X^\eu}_\beta}<C
		\end{align*}
		for $0\leq \beta\leq \frac{1}{2}$.
	\end{proof}
	\begin{lemma}
		\label{lemma_estimates_young}
		Assume $w^\epsilon\in \mathcal{S}$ is a control such that
		\begin{align*}
			\sup_{\epsilon>0}\norm{w^\epsilon}^2_{\mathcal{S}}=\sup_{\epsilon>0}\int_0^1
			\left[ \abs{\hat{u}^\epsilon_s}^2+\abs{\hat{v}^\epsilon_s}^2 \right] ds< N
		\end{align*}
		for some finite constant $N$. Then, the following two
		assertions hold.		
		\begin{enumerate}
			\item[(i)] Under Conditions \ref{condH1} and \ref{condH2-A}, there exists a constant $C$ such that for any $\beta$ in $(1-H,1]$,
			\begin{align*}
				\E{\sup_{0\leq t\leq 1}\abs{\int_0^t f\brac{X^\eu_s,Y^\eu_s}dW^H_s }}
				\leq C{\epsilon^{-\frac{\beta}{2}}}
			\end{align*}
			and
			\begin{align*}
				\E{ \abs{\sup_{t\in [0,1]} \int_0^t\nabla_x\phi\brac{X^\eu_s,Y^\eu_s}f\brac{X^\eu_s,Y^\eu_s}d{W}^H_s
				}}&\leq C\epsilon^{-\frac{1}{2}}.
			\end{align*}
			
			\item[(ii)] Under Conditions \ref{condH1} and \ref{condH2-B}, there exists a constant $C$ such that
			\begin{align*}
				\E{\sup_{t\in [0,1]}\abs{\int_0^t f\brac{Y^\eu_s}dW^H_s}}
				\leq C{\epsilon^{-\frac{M_f}{2}}}
			\end{align*}
			and
			\begin{align*}
				\E{\sup_{t\in[0,1]}\abs{\int_0^t \nabla_x\phi(X^\eu_s,Y^\eu_s)f(Y^\eu_s)dW^H_s}}\leq C\epsilon^{-\frac{M_k}{2}}.
			\end{align*}
		\end{enumerate}
		
	\end{lemma}
	\begin{proof}
		We begin by proving part $(i)$. The first estimate is immediate based
		on Lemma \ref{lemma_Holder_X} and the estimate at
		\eqref{estimate_young_h1}. Regarding the second estimate, the
		inequality at \eqref{estimate_Poissonsolution_h1} and Condition \ref{condH2-A} imply that $\nabla_x\phi(x,y)f(x,y)$ is Lipschitz continuous. Hence, by the Young-Lo\'{e}ve inequality for $1-H<\beta\leq 1$, we have
		\begin{align*}
			&\E{\sup_{0\leq t\leq 1}
				\abs{\int_0^t\nabla_x\phi\brac{X^\eu_s,Y^\eu_s}f\brac{X^\eu_s,Y^\eu_s}d{W}^H_s}}\\
			& \qquad\qquad\qquad\qquad\qquad\qquad\qquad\qquad\qquad \leq  \E{\abs{\nabla_x\phi\brac{X^\eu,Y^\eu}f(X^\eu,Y^\eu)}_\beta}
			\\
			&\qquad\qquad\qquad\qquad\qquad\qquad\qquad\qquad\qquad \leq \abs{\nabla_x\phi(x,y)f(x,y)}_{\operatorname{Lip}}\E{\abs{X^\eu}_\beta+\abs{Y^\eu}_\beta}\\
			&\qquad\qquad\qquad\qquad\qquad\qquad\qquad\qquad\qquad \leq C\brac{1+\epsilon^{-\frac{1}{2}}},
		\end{align*}
		where the last inequality is a consequence of Lemmas \ref{lemma_Holder_Y} and \ref{lemma_Holder_X}.
		
		We now proceed to the proof of part $(ii)$. For the first
		estimate, we perform a similar calculation to the
		one that was done at \eqref{estimate_young_h2} (this
		requires Conditions \ref{condH1} and \ref{condH2-B}) and get
		\begin{align*}
			\E{\sup_{t\in [0,1]}\abs{\int_0^t f\brac{Y^\eu_s}dW^H_s}}
			\leq  C\brac{{\abs{Y^\eu}}_{\frac{M_f}{2}}+\E{(y_0)^{D_f}} }
			\leq C{\epsilon^{-\frac{M_f}{2}}}.
		\end{align*}
		Next, under Conditions \ref{condH1} and
		\ref{condH2-B}, the $M_k$-H\"{o}lder
		continuity of $\nabla_x\phi(x,y)f(x)$ together with the
		estimates in Lemmas \ref{lemma_Holder_Y} and \ref{lemma_Holder_X} yield
		\begin{align*}
			&\E{\abs{\nabla_x\phi(X^\eu_r,Y^\eu_r)f(Y^\eu_r)-\nabla_x\phi(X^\eu_t,Y^\eu_t)f(Y^\eu_t)}}\\
			&\qquad\qquad\qquad\qquad\qquad\qquad\qquad\qquad\qquad\leq \E{\abs{X^\eu_r-X^\eu_t}^{M_k}}+\E{\abs{Y^\eu_r-Y^\eu_t}^{M_k}}\\
			&\qquad\qquad\qquad\qquad\qquad\qquad\qquad\qquad\qquad\leq C\brac{1+\epsilon^{-\frac{M_k}{2}}}\abs{r-t}^{\frac{M_k}{2}},
		\end{align*}
		so that
		\begin{align*}
			\E{\abs{\nabla_x\phi(X^\eu,Y^\eu)f(Y^\eu)}_{\frac{M_k}{2}}}\leq C\epsilon^{-\frac{M_k}{2}}.
		\end{align*}
		Therefore, as $\frac{M_k}{2}+H>1$ in Condition \ref{condH2-B}, we can apply the Young-Lo\'{e}ve inequality to obtain
		\begin{align*}
			\E{\sup_{t\in[0,1]}\abs{\int_0^t
					\nabla_x\phi(X^\eu_s,Y^\eu_s)f(Y^\eu_s)dW^H_s}}&\leq
			C\bigg( \abs{W^H}_H
			\E{\abs{\nabla_x\phi(X^\eu,Y^\eu)f(Y^\eu)}_{\frac{M_k}{2}}}\\
			&\quad +{\E{\abs{\nabla_x\phi(x_0,y_0)f(y_0)}}}\bigg)\\
			&\leq C\epsilon^{-\frac{M_k}{2}}.
		\end{align*}
	\end{proof}
	\begin{lemma}
		\label{lemma_estimate_eta}
		Assume $w^\epsilon\in \mathcal{S}$ is a control such that
		\begin{align*}
			\sup_{\epsilon>0}\norm{w^\epsilon}^2_{\mathcal{S}}=\sup_{\epsilon>0}\int_0^1
			\left[ \abs{\hat{u}^\epsilon_s}^2+\abs{\hat{v}^\epsilon_s}^2 \right] ds< N
		\end{align*}
		for some finite constant $N$. Under Conditions
		\ref{condH1} and either \ref{condH2-A}
		or \ref{condH2-B}, there exists a constant $C$ such that
		\begin{align*}
			\E{\sup_{0\leq t\leq 1} \abs{\eta^\eu_t}^2} < C.
		\end{align*}
		Furthermore, this implies for any $\rho>0$,
		\begin{align*}
			\E{\sup_{\substack{0\leq r,t\leq 1 \\ \abs{r-t}<\rho}}\abs{\int_{r}^{t} \frac{1}{\sqrt{\epsilon}h(\epsilon)}\brac{\bar{g}(X^\eu_s)-\bar{g}(\bar{X}_s)}ds}}\leq C\rho.
		\end{align*}
	\end{lemma}
	\begin{proof}
		Under Condition \ref{condH2-A} or \ref{condH2-B}, $\nabla_x\bar{g}(x)$
		is bounded. This fact, combined with equation
		\eqref{equation_eta_g_part} and the fact that $X^\eu$ converges to
		$\bar{X}$ in probability, implies that there exists some constant $C$
		such that
		\begin{align}
			\label{estimate_barg}
			\E{\sup_{0\leq t\leq 1}\abs{\int_{0}^{t} \frac{1}{\sqrt{\epsilon}h(\epsilon)}\brac{\bar{g}(X^\eu_s)-\bar{g}(\bar{X}_s)}ds}^2}\leq C\int_0^1 \E{\sup_{0\leq r\leq s}\abs{\eta^\eu_r}^2}ds.
		\end{align}
		In addition, based on equation \eqref{equation_eta_new}, we have
		\begin{align}
			\label{estimate_eta_middlestep}
			\E{\sup_{0\leq t\leq 1}\abs{\eta^\eu_t}^2} &\leq  C\bigg(\E{\sup_{0\leq t\leq 1}\abs{\int_0^t \nabla_y\phi\brac{X^\eu_s,Y^\eu_s} \sigma\brac{Y^\eu_s}\dot{v}^\epsilon_sds}^2}\nonumber
			\\&\quad + \E{\sup_{0\leq t\leq 1}\abs{\int_{0}^{t}  \frac{1}{\sqrt{\epsilon}h(\epsilon)}\brac{\bar{g}(X^\eu_s)-\bar{g}(\bar{X}_s)}ds}^2}\nonumber
			\\&\quad +\E{\sup_{0\leq t\leq 1}\abs{\frac{1}{h(\epsilon)}\int_0^tf\brac{X^\eu_s,Y^\eu_s}d{W}^H_s}^2}\nonumber
			\\&\quad +\E{\sup_{0\leq t\leq 1}\abs{\int_0^tf\brac{X^\eu_s,Y^\eu_s}\dot{u}^\epsilon_sds}^2}
			+\E{\sup_{0\leq t\leq 1}\abs{R^\epsilon_2(t)}^2}\bigg),
		\end{align}
		with
		\begin{align*}
			\E{\sup_{0\leq t\leq 1}\abs{R^\epsilon_2(t)}^2}&\leq C\bigg(\E{\sup_{0\leq t\leq 1}\abs{\frac{\sqrt{\epsilon}}{h(\epsilon)}\brac{\phi\brac{X^\eu_t,Y^\eu_t}-\phi(x_0,y_0)}}^2}\\
			&\quad +\E{\sup_{0\leq t\leq 1}\abs{\frac{\sqrt{\epsilon}}{h(\epsilon)}\int_0^t \nabla_x\phi\brac{X^\eu_s,Y^\eu_s} g\brac{X^\eu_s,Y^\eu_s}ds}^2}\\
			&\quad +\E{\sup_{0\leq t\leq 1}\abs{\epsilon\int_0^t\nabla_x\phi\brac{X^\eu_s,Y^\eu_s}f\brac{X^\eu_s,Y^\eu_s}\dot{u}^\epsilon_s ds}^2}\\
			&\quad +\E{\sup_{0\leq t\leq 1}\abs{\frac{{\epsilon}}{h(\epsilon)}\int_0^t\nabla_x\phi\brac{X^\eu_s,Y^\eu_s}f\brac{X^\eu_s,Y^\eu_s}d{W}^H_s}^2}\\
			&\quad +\E{\sup_{0\leq t\leq 1}\abs{\frac{1}{h(\epsilon)}\int_0^t\nabla_y\phi\brac{X^\eu_s,Y^\eu_s}\sigma\brac{Y^\eu_s}dB_s}^2}\bigg).
		\end{align*}
		We will estimate the terms on the right-hand side of
		\eqref{estimate_eta_middlestep}, starting with those which contain
		Young integrals. Condition \ref{condH2-A} guarantees that there exists
		some $\beta$ in $[0,1]$ such that $\beta+H>1$ and
		$h(\epsilon)^{-1}\epsilon^{-\frac{\beta}{2}}\to 0$ as $\epsilon\to 0$,
		so that part $(i)$ of Lemma \ref{lemma_estimates_young} (which
		requires Conditions \ref{condH1} and \ref{condH2-A}) yields
		\begin{align*}
			\E{\sup_{0\leq t\leq 1}\abs{\frac{1}{h(\epsilon)}\int_0^tf\brac{X^\eu_s,Y^\eu_s}d{W}^H_s}^2}\leq Ch(\epsilon)^{-2}\epsilon^{-\beta}\to 0.
		\end{align*}
		Part $(i)$ of Lemma \ref{lemma_estimates_young} also
		implies that
		\begin{align*}
			\E{ \sup_{t\in [0,1]} \abs{\frac{{\epsilon}}{h(\epsilon)}\int_0^t\nabla_x\phi\brac{X^\eu_s,Y^\eu_s}f\brac{X^\eu_s,Y^\eu_s}d{W}^H_s
				}^2}&\leq C\frac{1}{h(\epsilon)^2}\to 0.
		\end{align*}
		Meanwhile, under Condition \ref{condH2-B}, we use part
		$(ii)$ of Lemma \ref{lemma_estimates_young} to get, as $\epsilon\to 0$,
		\begin{align*}
			\E{\sup_{0\leq t\leq 1}\abs{\frac{1}{h(\epsilon)}\int_0^tf\brac{Y^\eu_s}d{W}^H_s}^2}\leq Ch(\epsilon)^{-2}\epsilon^{-M_f}\to 0
		\end{align*}
		and
		\begin{align*}
			\E{ \sup_{t\in [0,1]} \abs{\frac{{\epsilon}}{h(\epsilon)}\int_0^t\nabla_x\phi\brac{X^\eu_s,Y^\eu_s}f\brac{X^\eu_s,Y^\eu_s}d{W}^H_s
				}^2}&\leq Ch(\epsilon)^{-2}\epsilon^{2-M_k}\to 0.
		\end{align*}
		The remaining terms on the right-hand side of
		\eqref{estimate_eta_middlestep}, except the
		term \begin{align*}\E{\sup_{0\leq t\leq 1}\abs{\int_{0}^{t}
					\frac{1}{\sqrt{\epsilon}h(\epsilon)}\brac{\bar{g}(X^\eu_s)-\bar{g}(\bar{X}_s)}ds}^2},\end{align*}
		are bounded by using Lemmas \ref{lemma_estimatefromMorKos2} and \ref{lemma_estimates_notyoung} (which require Conditions \ref{condH1} and \ref{condH2-B}). Thus, it follows from the estimates at \eqref{estimate_barg} and \eqref{estimate_eta_middlestep} that
		\begin{align*}
			\E{\sup_{0\leq t\leq 1} \abs{\eta^\eu_t}^2}\leq C_1+C_2\int_0^1\E{\sup_{0\leq r\leq s} \abs{\eta^\eu_r}^2}ds
			.\end{align*}
		An application of Gronwall's inequality then yields the
		first claim of our statement, which is
		\begin{align*}
			\E{\sup_{0\leq t\leq 1} \abs{\eta^\eu_t}^2}\leq C.
		\end{align*}
		For the second claim, we proceed similarly to the derivation of the
		estimate at \eqref{estimate_barg}. Then for $\rho>0$,
		\begin{align*}
			\E{\sup_{\substack{0\leq r,t\leq 1 \\ \abs{r-t}<\rho}}\abs{\int_{r}^{t} \frac{1}{\sqrt{\epsilon}h(\epsilon)}\brac{\bar{g}(X^\eu_s)-\bar{g}(\bar{X}_s)}ds}}&\leq C\E{\sup_{\substack{0\leq r,t\leq 1 \\ \abs{r-t}<\rho}}\int_r^t {\abs{\eta^\eu_s}}ds}\\
			&\leq C\rho\E{\sup_{0\leq t\leq 1} \abs{\eta^\eu_t}^2}\leq C\rho.
		\end{align*}
	\end{proof}
	\begin{lemma}
		\label{lemma_limit_R2}
		Let $R_2^\epsilon$ be the remainder term that appears
		in equation \eqref{equation_eta_g_part}. Under
		Conditions \ref{condH1} and either
		\ref{condH2-A} or \ref{condH2-B} , it holds that $R_2^\epsilon\to 0$ in $C([0,1];\R^{n})$ in probability along a subsequence.
	\end{lemma}
	\begin{proof}
		For the purpose of identifying the limit of
		$R_2^\epsilon$, we invoke the Skorokhod Representation
		Theorem and assume that $X^\eu\to \bar{X}$ a.s. in
		$C([0,1];\R^{n})$ as $\epsilon\to 0$. As $\bar{X}$ is bounded
		under Condition \ref{condH2-A} or \ref{condH2-B}, the
		Dominated Convergence Theorem implies that
		\begin{align}
			\label{limit_control_X_L2}
			\lim_{\epsilon\to 0}\E{\sup_{0\leq s\leq 1}\abs{X^\eu_s- \bar{X}_s}^2}=0.
		\end{align}
		Now, we employ the bound \eqref{bornepourr2} and get
		\begin{align*}
			\E{\sup_{0\leq s\leq 1}\abs{R^\epsilon_2(s)}}&\leq  \E{\int_0^1 \abs{\nabla^2_x \bar{g}}_\infty \abs{\eta^\eu_s} \abs{X^\eu_s-\bar{X}_s}ds}\\
			&\leq C\E{\sup_{0\leq s\leq 1} \abs{\eta^\eu_s} \sup_{0\leq s\leq 1} \abs{X^\eu_s-\bar{X}_s}}\\
			&\leq C\E{\sup_{0\leq s\leq 1} \abs{\eta^\eu_s}^2}\E{\sup_{0\leq s\leq 1} \abs{X^\eu_s-\bar{X}_s}^2}\\
			&\leq C\E{\sup_{0\leq s\leq 1} \abs{X^\eu_s-\bar{X}_s}^2},
		\end{align*}
		In particular, the second inequality is due to the boundedness of
		$\nabla^2_x\bar{g}$ implied by either Condition
		\ref{condH2-A} or \ref{condH2-B}. The last
		inequality is a consequence of Lemma
		\ref{lemma_estimate_eta} (which requires Conditions
		\ref{condH1} and either \ref{condH2-A}
		or \ref{condH2-B}). \eqref{limit_control_X_L2} then gives us the desired limit.
	\end{proof}

	\bibliography{refs}
\end{document}